\documentclass[12pt,oneside]{book}
\usepackage{amsmath}
\usepackage{amsthm}
\usepackage{graphics}
\usepackage{epsfig,amssymb,latexsym,verbatim}
\usepackage{graphicx}
\usepackage{multirow}
\usepackage{multicol}
\usepackage{float}
\usepackage{booktabs}
\usepackage{fancyhdr}
\usepackage{enumerate,verbatim}
\usepackage[sort&compress]{natbib}
\usepackage{rotating}
\usepackage{hyperref}
\usepackage{subfigure}
\usepackage{ulem}
\usepackage{url}
\usepackage{threeparttable}
\usepackage{xr}
\usepackage{multirow}
\usepackage{multicol}
\usepackage{wrapfig}
\usepackage{lscape}
\usepackage{rotating}
\usepackage{epstopdf}
\usepackage{subfigure}
\usepackage{amsfonts}
\usepackage[usenames,dvipsnames]{color}
\usepackage{times,color}
\usepackage{amsbsy}
\usepackage{amsthm}
\usepackage[sf,small,center, compact]{titlesec}

\def\bs{\boldsymbol}
\def\T{{\mbox{\rm\scriptsize T}}}

\definecolor{red}{rgb}{1,0,0}
\definecolor{green}{rgb}{0,1,0}
\definecolor{blue}{rgb}{0,0,1}

\definecolor{lxy}{RGB}{180,0,180}

\floatstyle{ruled}
\newfloat{algorithm}{tbp}{loa}
\floatname{algorithm}{Algorithm}
\def\bs{\boldsymbol}

\newcommand{\thetheorem}{{\thesection. \arabic{theorem}}}
\newcommand{\thelemma}{{\thesection. \arabic{lemma}}}
\newcommand{\theproposition}{{\thesection. \arabic{proposition}}}
\newcommand{\thecorollary}{{\thesection. \arabic{corollary}}}
\newtheorem{theorem}{{\sc Theorem}}
\newtheorem{lemma}{{\sc Lemma}}

\begin{document}
\renewcommand{\baselinestretch}{1.2}
\markboth{\hfill{\footnotesize\rm Li Wang, Guannan Wang, Min-Jun Lai and Lei Gao}\hfill}
{\hfill {\footnotesize\rm Efficient Estimation of Partially Linear Models for Spatial Data over Complex Domains} \hfill}
\renewcommand{\thefootnote}{}
$\ $\par \fontsize{10.95}{14pt plus.8pt minus .6pt}\selectfont
\vspace{0.8pc} \centerline{\large\bf Efficient Estimation of Partially Linear Models for Spatial Data over Complex Domains}
\vspace{.4cm} \centerline{Li Wang$^{a}$, Guannan Wang$^{b}$, Min-Jun Lai$^{c}$ and Lei Gao$^{a}$
\footnote{\emph{Address for correspondence}: Li Wang, Department of Statistics and the Statistical Laboratory, Iowa State University, Ames, IA, USA. Email: lilywang@iastate.edu}} \vspace{.4cm} \centerline{\it $^{a}$Iowa State University, $^{b}$College of William \& Mary and $^{c}$University of Georgia} \vspace{.55cm}
\fontsize{9}{11.5pt plus.8pt minus .6pt}\selectfont

\begin{quotation}
\noindent \textit{Abstract:} In this paper, we study the estimation of partially linear models for spatial data distributed over complex domains. We use bivariate splines over triangulations to represent the nonparametric component on an irregular two-dimensional domain. The proposed method is formulated as a constrained minimization problem which does not require constructing finite elements or locally supported basis functions. Thus, it allows an easier implementation of piecewise polynomial representations of various degrees and various smoothness over an arbitrary triangulation. Moreover, the constrained minimization problem is converted into an unconstrained minimization via a QR decomposition of the smoothness constraints, which allows for the development of a fast and efficient penalized least squares algorithm to fit the model. The estimators of the parameters are proved to be asymptotically normal under some regularity conditions. The estimator of the bivariate function is consistent, and its rate of convergence is also established. The proposed method enables us to construct confidence intervals and permits inference for the parameters. The performance of the estimators is evaluated by two simulation examples and by a real data analysis.

\vspace{9pt} \noindent \textit{Key words and phrases:} Bivariate splines, Penalty, Semiparametric regression, Spatial data, Triangulation.
\end{quotation}

\fontsize{10.95}{14pt plus.8pt minus .6pt}\selectfont

\thispagestyle{empty}

\setcounter{chapter}{1} \label{SEC:introduction}
\setcounter{equation}{0}
\noindent \textbf{1. Introduction} \vskip 0.1in

In many geospatial studies, spatially distributed covariate information is available. For example, geographic information systems may contain measurements obtained from satellite images at some locations. These spatially explicit data can be useful in the construction and estimation of regression models, however, the domain over which variables of interest are defined is often found to be complicated, such as stream networks, islands and mountains. For example, Figure \ref{FIG:MC_vtrue} (a) and (b) show the largest estuary in New Hampshire together with the location of 97 sites where mercury in sediment concentrations was surveyed in the years 2000, 2001 and 2003; see \cite{Wang:Ranalli:07}. It is well known that many conventional smoothing tools with respect to the Euclidean distance between observations suffer from the problem of ``leakage'' across the complex domains, which refers to the poor estimation over difficult regions by the inappropriate linking of parts of the domain separated by physical barriers; see excellent discussions in \cite{Ramsay:02} and \cite{Wood:Bravington:Hedley:08}. In this paper, we propose to use bivariate splines (smooth piecewise polynomial functions over a triangulation of the domain of interest) to model spatially explicit datasets which enable us to overcome the ``leakage'' problem and provide more accurate estimation and prediction.

%%%%%%%%%%%%%%%%%%%%%%%%%%%%%%%%%%%%%%%%%%%%%%%%%%%%%%%%%%%%%%%%%%%%%%
\begin{figure}[htbp]
	\begin{center}
		\begin{tabular}{cc}
			\includegraphics[height=7.2cm]{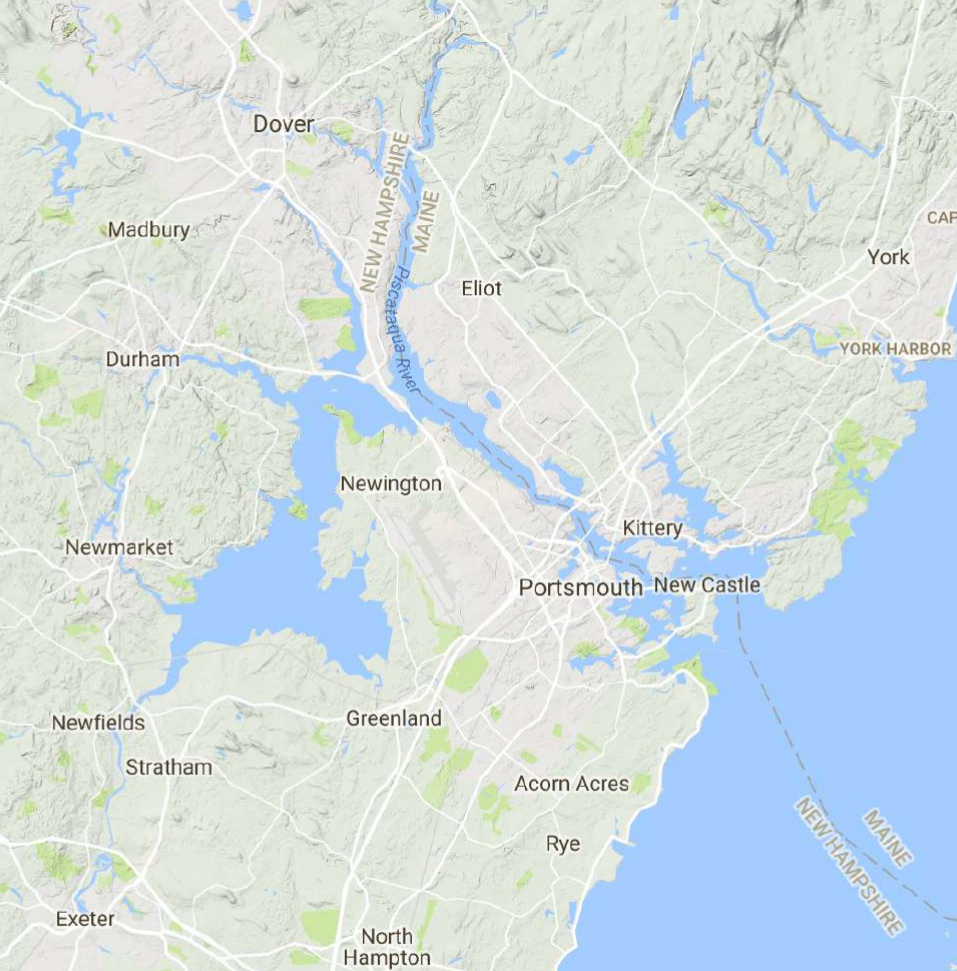} &
			\includegraphics[height=7.2cm]{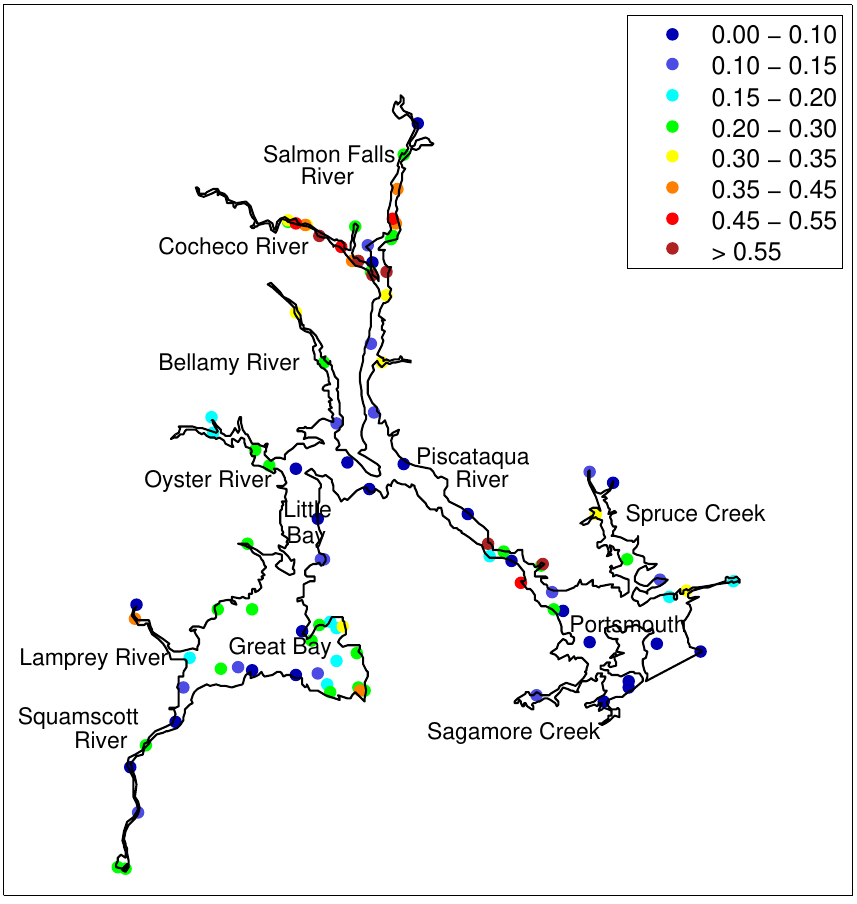}\\
			(a) & (b)
		\end{tabular}
	\end{center} \vskip -.2in
	\caption{Regional map of estuaries. Dots in (b) represent sample locations with different colors indicating different levels of mercury concentrations.}
	\label{FIG:MC_vtrue}
\end{figure}
%%%%%%%%%%%%%%%%%%%%%%%%%%%%%%%%%%%%%%%%%%%%%%%%%%%%%%%%%%%%%%%%%%%%%%

We focus here on the partially linear model \citep{Speckman:88,He:Shi:96,Mammen:DeGeer:97,Liang:Hardle:Carroll:99,Hardle:Liang:Gao:00,Ma:Chiou:Wang:06,Liang:Li:09},
referred to as PLM, for data randomly distributed over 2-D domains. To be more specific, let $\mathbf{X}_{i}=(X_{i1},X_{i2})^{\T}$ be the location of $i$-th point, $i=1,\ldots, n$, which ranges over a bounded domain $\Omega \subseteq \mathbb{R}^2$ of arbitrary shape, for example, the domain of the estuaries in New Hampshire shown in Figure~\ref{FIG:MC_vtrue}. Let $Y_{i}$ be the response variable and
$\mathbf{Z}_{i}=(Z_{i1},\ldots,Z_{ip})^{\T}$ be the predictors at location $\mathbf{X}_{i}$. Suppose that $\left\{(\mathbf{Z}_{i},\mathbf{X}_{i},Y_{i})\right\} _{i=1}^{n}$ satisfies the following model
\begin{equation}
Y_{i}=\mathbf{Z}_{i}^{\T} \bs{\beta}+g\left(\mathbf{X}_{i}\right)+\epsilon_{i}, \quad i=1,\ldots ,n,
\label{model}
\end{equation}
where $\bs{\beta}=(\beta_{1},\ldots,\beta_{p})^{\T}$ are unknown parameters, $g(\cdot)$ is some unknown but smooth bivariate function, and $\epsilon_{i}$'s are i.i.d random noises with $E(\epsilon_{i})=0$ and $\mathrm{Var}(\epsilon _{i})=\sigma^2$. Each $\epsilon_{i}$ is independent of $\mathbf{X}_{i}$ and $\mathbf{Z}_{i}$. In many situations, our main interest is in estimating and making inference for the regression parameters $\bs{\beta}$, which provides measures of the effect of the covariate $\mathbf{Z}$ after adjusting for the location effect of $\mathbf{X}$.

If $g(\cdot)$ is a univariate function, model (\ref{model}) becomes a typical PLM. In the past three decades, flexible and parsimonious PLMs have been extensively studied and widely used in many statistical applications, from biostatistics to econometrics, from engineering to social science; see \cite{Chen:Liang:Wang:11}, \cite{Huang:Zhang:Zhou:07}, \cite{Liu:Wang:Liang:11}, \cite{Wang:Liu:Liang:Carroll:11}, \cite{Ma:Song:Wang:13}, \cite{Wang:Xue:Qu:Liang:14}, \cite{Zhang:Cheng:Liu:11} for some recent works on PLMs. When $g(\cdot)$ is a bivariate function, there are two popular estimation tools: bivariate P-splines \citep{Marx:Eilers:05} and thin plate splines \citep{Wood:03}. Later, \cite{Xiao:Li:Ruppert:13} proposed a sandwich smoother, which has a tensor product structure that simplifies an asymptotic analysis and can be computed fast. The application to spatial data analysis over complex domains, however, has been hampered due to the scarcity of bivariate smoothing tools that are not only computationally efficient but also theoretically reliable to solve the problem of ``leakage'' across the domain. Traditional smoothing methods in practical data analysis, such as kernel smoothing, wavelet-based smoothing, tensor product splines and thin plate splines, usually perform poorly for those data, since they do not take into account the shape of the domain and also smooth across concave boundary regions.

There are several challenges when going from rectangular domains to irregular domains with complex boundaries or holes. Some efforts have recently been devoted to studying the smoothing over irregular domains, and significant progress has been made. To deal with irregular domains, \cite{Eilers:06} utilized the Schwarz-Christoffel transform to convert the complex domains to regular domains, however, this transformation may lead to the artifact distortion of observation density by squeezing observations with vastly different response values together; thus it may make smoothing more difficult. \cite{Wang:Ranalli:07} proposed to replace the Euclidean distance with the geodesic distance in the low-rank thin-plate spline smoothing.  To calculate the geodesic distances, a graph is constructed where each vertex is the location of an observation and is connected only to its $k$ nearest neighbors. Floyd's algorithm is then used to find the shortest path through the graph. This algorithm has a computing complexity of $O(n^3)$ without even considering the selection of the optimal $k$, which makes the approach costly for large datasets. In addition, their method involves computing the square roots of matrices that are not guaranteed to be positive semi-definite.

\cite{Ramsay:02} suggested a penalized least squares approach with a Laplacian penalty and transformed the problem to that of solving a system of partial differential equations (PDEs). Recently, \cite{Sangalli:Ramsay:Ramsay:13} extended the method in \cite{Ramsay:02} to the PLMs, which allows for spatially distributed covariate information to be in the models. The data smoothing problem in \cite{Sangalli:Ramsay:Ramsay:13} is solved using finite element method (FEM), a method mainly developed and used to solve PDEs. Although their method is useful in many practical applications, the theoretical properties of the estimation were not investigated in their paper. In addition, our case study in Section 5 and simulation study in Appendix B reveal that the FEM is not flexible enough to well estimate the functional part of the model. \cite{Wood:Bravington:Hedley:08} also pointed out the FEM method requires a very fine triangulation in order to reach certain approximation power when the underlying function is complicated.

In this paper, we tackle the estimation problem by using the bivariate splines defined on triangulations \citep{Awanou:Lai:Wenston:05,Lai:Schumaker:07}. Our approach is an improvement of \cite{Sangalli:Ramsay:Ramsay:13} in the sense that we use spline functions of more flexible degrees and various smoothness than continuous linear finite elements so that we are able to better approximate the bivariate function $g$. Another important feature of this approach is that it does not require to construct locally supported splines or finite elements of higher degree than one.

To the best of the authors' knowledge, statistical aspects of smoothing for PLMs by using bivariate splines have not been discussed in the literature so far. This paper presents the first attempt at investigating the asymptotic properties of the PLMs for data distributed on a non-rectangular complex region. We study the asymptotic properties of the least squares estimators of $\bs{\beta}$ and $g(\cdot)$ by using bivariate splines defined on triangulations with a penalty term. We show that our estimator of $\bs{\beta}$ is root-$n$ consistent and asymptotically normal, although the convergence rate of the estimator of the nonparametric component $g(\cdot)$ is slower than root-$n$. A standard error formula for the estimated coefficients is provided and tested to be accurate enough for practical purposes. Hence, the proposed method enables us to construct confidence intervals for the regression parameters. We also obtain the convergence rate for the estimator of $g(\cdot)$.

The rest of the paper is organized as follows. In Section 2, we give a brief review of the triangulations and propose our estimation method based on penalized bivariate splines. We also discuss the details on how to choose the penalty parameters. Section 3 is devoted to the asymptotic analysis of the proposed estimators. Section 4 provides a detailed numerical study to compare several methods in two different scenarios and explores the estimation and prediction accuracy. In Section 5, we apply the proposed method to the mercury concentration study where the variables of interest are defined over the estuary in New Hampshire depicted in Figure \ref{FIG:MC_vtrue}. Some concluding remarks are given in Section 6. Technical details are provided in the appendixes.

%%%%%%%%%%%%%%%%%%%%%%%%%%%%%%%%%%%%%%%%%%%%%%%%%%%%%%%%%%%%%%
%%%%%%%%%%%%%%%%%%%%%%%%%%%%%%%%%%%%%%%%%%%%%%%%%%%%%%%%%%%%%%
%%%%%%%%%%%%%%%%%%%%%%%%%%%%%%%%%%%%%%%%%%%%%%%%%%%%%%%%%%%%%%
\setcounter{chapter}{2} \renewcommand{\theproposition}{2.\arabic{proposition}}
\renewcommand{\thetable}{2.\arabic{table}} \setcounter{table}{0} 
\renewcommand{\thefigure}{2.\arabic{figure}} \setcounter{figure}{0}  \setcounter{equation}{0} \setcounter{lemma}{0} \setcounter{theorem}{0} \setcounter{proposition}{0} \setcounter{corollary}{0}
\vskip .12in \noindent \textbf{2. Triangulations and Penalized Spline Estimators} 
\label{sec:spline}

Our estimation method is based on penalized bivariate splines on triangulations. The idea is to approximate the function $g(\cdot)$ by bivariate splines that are piecewise polynomial functions over a 2-D triangulated domain which enables one to fit $g(\cdot)$ more flexibly. We use this approximation to construct least squares estimators of the linear and nonlinear components of the model with a penalization term. In the following of this section, we describe the background of triangulations, B-form bivariate splines and introduce the penalized spline estimators.

%%%%%%%%%%%%%%%%%%%%%%%%%%%%%%%%%%%%%%%%%%%%%%%%%%%%%%%%%%%%%%
\vskip .10in \noindent \textbf{2.1. Triangulations} \vskip .10in

Triangulation is an effective strategy to handle data distribution over irregular regions with complex boundaries and/or interior holes. Recently, it has attracted substantial attention in many applied areas, such as geospatial studies, numerical solutions of PDEs, image enhancements, and computer aided geometric design. Many triangulation software packages have been developed and are available for applications.  Appendix A explains the details of how to choose a triangulation for a given dataset.

We use $\tau$ to denote a triangle which is a convex hull of three points which are not located in one line. A collection $\triangle=\{\tau_1,\ldots,\tau_{N}\}$ of $N$ triangles is called a triangulation of $\Omega=\cup_{i=1}^{N} \tau_i$ provided that if a pair of triangles in $\triangle$ intersect, then their intersection is either a common vertex or a common edge. Although any kind of polygon shapes can be used for the partition of $\Omega$,  we use triangulations because any polygonal domain of arbitrary shape can be partitioned into finitely many triangles to form a triangulation $\triangle$. Given a triangle $\tau\in \triangle$, let $|\tau|$ be its longest edge length, and denote the size of $\triangle$ by $|\triangle|=\max \{|\tau|,\tau\in\triangle\}$, i.e., the length of the longest edge of $\triangle$.

%%%%%%%%%%%%%%%%%%%%%%%%%%%%%%%%%%%%%%%%%%%%%%%%%%%%%%%%%%%%%%
\vskip .12in \noindent \textbf{2.2. B-form bivariate splines} \vskip .10in

In this section we give a brief introduction to the bivariate splines. More in-depth description can be found in \cite{Lai:Schumaker:07}, \cite{Lai:08}, as well as \cite{Zhou:Pan:14} and the details of the implementation is provided in \cite{Awanou:Lai:Wenston:05}. Let $\tau=\langle\mathbf{v}_{1},\mathbf{v}_{2},\mathbf{v}_{3}\rangle$ be a non-degenerate (i.e. with non-zero area) triangle with vertices $\mathbf{v}_{1}$, $\mathbf{v}_{2}$, and $\mathbf{v}_{3}$. Then for any point  $\mathbf{v}\in\mathbb{R}$, there is a unique representation in the form $\mathbf{v} = b_1\mathbf{v}_1 + b_2\mathbf{v}_2 + b_3\mathbf{v}_3$, with $b_1+b_2+b_3=1$, where $b_1$, $b_2$ and $b_3$ are called the barycentric coordinates of the point $\mathbf{v}$ relative to the triangle $\tau$. The Bernstein polynomials of degree $d$ relative to triangle $\tau$ is defined as $B_{ijk}^{\tau,d}(\mathbf{v})=\frac{d}{i!j!k!}b_1^ib_2^jb_3^k$. Then for any $\tau\in\triangle$, we can write the polynomial piece of spline $s$ restricted on $\tau\in\triangle$ as $
s|_{\tau}=\sum_{i+j+k=d} \gamma_{ijk}^{\tau}B_{ijk}^{\tau,d},
$
where $\bs{\gamma}_{\tau}=\{\gamma_{ijk}^{\tau}, i+j+k=d\}$ are called B-coefficients of $s$.

For a nonnegative integer $r$, let $\mathbb{C}^r(\Omega)$ be the collection of all $r$-th continuously differentiable functions over $\Omega$. Given a triangulation $\triangle$, let $\mathbb{S}_{d}^{r}(\triangle)=\{s\in \mathbb{C}^{r}(\Omega):s|_{\tau}\in \mathbb{P}_{d}(\tau), \tau \in \triangle \}$ be a spline space of degree $d$ and smoothness $r$ over triangulation $\triangle $, where  $\mathbb{P}_{d}$ is the space of all polynomials of degree less than or equal to $d$. Let $\mathbb{S}=\mathbb{S}_{3r+2}^{r}(\triangle)$ for a fixed smoothness $r\ge1$, and we know that such a spline space has the optimal approximation order (rate of convergence) for noise-free datasets; see \cite{Lai:Schumaker:98} and \cite{Lai:Schumaker:07}.

For notation simplicity, let $\{B_{\xi}\}_{\xi \in \mathcal{K}}$ be the set of degree-$d$ bivariate Bernstein basis polynomials for $\mathbb{S}$, where $\mathcal{K}$ stands for an index set of all Bernstein basis polynomials. Then for any function $s\in \mathbb{S}$,  we can represent it by using the following basis expansion:
\begin{equation}
s(\mathbf{x})=\sum_{\xi \in \mathcal{K}} B_{\xi}(\mathbf{x})\gamma_{\xi}
=\mathbf{B}(\mathbf{x})^{\T}\bs{\gamma},
\label{EQ:basis-expansion}
\end{equation}
where $\bs{\gamma}^{\T} =(\gamma_{\xi},\xi \in \mathcal{K})$ is the spline coefficient vector. To meet the smoothness requirement of the splines, we need to impose some linear constraints on the spline coefficients $\bs{\gamma}$ in (\ref{EQ:basis-expansion}). We require that $\bs{\gamma}$ satisfies $\mathbf{H}\bs{\gamma}=0$ with $\mathbf{H}$ being the matrix for all smoothness conditions across shared edges of triangles, which depends on $r$ and the structure of the triangulation. See \cite{Zhou:Pan:14} for some examples of $\mathbf{H}$.

%%%%%%%%%%%%%%%%%%%%%%%%%%%%%%%%%%%%%%%%%%%%%%%%%%%%%%%%%%%%%%
\vskip 0.10in \noindent \textbf{2.3. Penalized Spline Estimators} \vskip 0.10in

To define the penalized spline method, for any direction $x_{j}$, $j=1,2$, let $D_{x_{j}}^{q}f(\mathbf{x})$ denote the $q$-th order derivative in the direction $x_{j}$ at the point $\mathbf{x}=(x_1,x_2)$. Let
\begin{equation}
\mathcal{E}_{\upsilon }(f)= \sum_{\tau\in
	\triangle}\int_{\tau} \sum_{i+j=\upsilon }
\binom{\upsilon}{i}
(D_{x_{1}}^{i}D_{x_{2}}^{j}f)^{2}dx_{1}dx_{2}
\label{penalty}
\end{equation}
be the energy functional for a fixed integer $\upsilon\ge1$ \citep{Lai:08}. Although all partial derivatives up to the chosen order $\upsilon$ can be included in (\ref{penalty}), for simplicity, in the remaining part of the paper, we use $\upsilon=2$, and one can study the similar problem for general $\upsilon \geq 2$. When $\upsilon=2$,
\begin{equation}
\mathcal{E}_{2}(f)=\int_{\Omega} \left((D_{x_{1}}^{2}f)^2+2(D_{x_{1}}D_{x_{2}}f)^2 + (D_{x_{2}}^{2}f)^2\right) dx_{1}dx_{2},
\label{penalty-2}
\end{equation}
which is similar to the thin-plate spline penalty \citep{Green:Silverman:94} except the latter is integrated over the entire plane $\mathbb{R}^{2}$. \cite{Sangalli:Ramsay:Ramsay:13} used a different roughness penalty from (\ref{penalty-2}), specifically, they use the integral of the square of the Laplacian of $f$, that is, $\int_{\Omega}(D_{x_{1}}^{2}f+D_{x_{2}}^{2}f)^{2}dx_{1}dx_{2}$. Both forms of penalties are invariant with respect to Euclidean transformations of spatial co-ordinates, thus, the bivariate smoothing does not depend on the choice of the coordinate system.

Given $\lambda>0$ and given the data set $\{(\mathbf{Z}_i,\mathbf{X}_i,Y_i)\}_{i=1}^{n}$, we consider the following minimization problem:
\begin{eqnarray}
\min_{\bs{\beta}}
\min_{s\in \mathbb{S}}\sum_{i=1}^{n}\left\{Y_{i}-\mathbf{Z}_{i}^{\T} \bs{\beta} -s\left(
\mathbf{X}_{i}\right)\right\}^{2}+\lambda
\mathcal{E}_{\upsilon }(s). \label{DEF:minimization}
\end{eqnarray}
where $\mathbb{S}$ is a spline space over triangulation $\triangle$ of $\Omega$.

Let $\mathbf{Y} = (Y_1,\ldots,Y_n)^{\T}$ are the vector of $n$ observations of the response variable,  $\mathbf{X} _{n\times 2}=\{(X_{i1}, X_{i2})\}_{i=1}^{n}$ are the location design matrix, $\mathbf{Z}_{n\times p}= \{(Z_{i1}, \ldots,Z_{ip})\}_{i=1}^{n}$ the  collection of all covariates. Denote by $\mathbf{B}$ the $n\times K$ evaluation matrix of Bernstein basis polynomials whose $i$-th row is given by $\mathbf{B}_{i}^{\T}=\{B_{\xi}(\mathbf{X}_{i}), \xi\in \mathcal{K}\}$. Then according to (\ref{EQ:basis-expansion}), $\{s(\mathbf{X}_i)\}_{i=1}^{n}$ can be written by $\mathbf{B}\bs{\gamma}$.
Thus the minimization in (\ref{DEF:minimization}) can be reduced to
\begin{equation}
\min_{\bs{\beta},\bs{\gamma}} L(\bs{\beta},\bs{\gamma})\!=\!\min_{\bs{\beta},\bs{\gamma}} \left\{\|\mathbf{Y}-\mathbf{Z}\bs{\beta}-\mathbf{B}\bs{\gamma}\|^{2}\!+\!\lambda\bs{\gamma}^{\T}\mathbf{P}\bs{\gamma}\right\}\!~ \mathrm{subject~to} ~ \mathbf{H}\bs{\gamma}=\mathbf{0},
\label{EQ:minimization}
\end{equation}
where $\mathbf{P}$ is the block diagonal penalty matrix satisfying that $\bs{\gamma}^{\T}\mathbf{P}\bs{\gamma}=\mathcal{E}_{\upsilon}(\mathbf{B}\bs{\gamma})$.

To solve the constrained minimization problem (\ref{EQ:minimization}), we first remove the constraint via QR decomposition of the transpose of the constraint matrix $\mathbf{H}$. Specifically, we have
\begin{equation}
\label{QR4H}
\mathbf{H}^{\T}=\mathbf{Q}\mathbf{R} =\left(\mathbf{Q}_1 ~\mathbf{Q}_2\right)
\binom{\mathbf{R}_{1}}{\mathbf{0}}=\mathbf{Q}_1\mathbf{R}_{1},
\end{equation}
where $\mathbf{Q}$ is an orthogonal matrix, $\mathbf{R}_{1}$ is an upper triangle matrix, and the submatrix $\mathbf{Q}_1$ is the first $r$ columns of $\mathbf{Q}$, where $r$ is the rank of matrix $\mathbf{H}$. It is easy to see the following result; see its proof in Appendix D.
\begin{lemma}
	\label{LEM:1}
	Let $\mathbf{Q}_1, \mathbf{Q}_2$ be submatrices as in \eqref{QR4H}. Let $\bs{\gamma}=\mathbf{Q}_2\bs{\theta}$ for a vector $\bs{\theta}$ of appropriate size. Then $\mathbf{H}\bs{\gamma}=\mathbf{0}$. On the other hand, if $\mathbf{H}\bs{\gamma}=\mathbf{0}$, then there exists a vector $\bs{\theta}$ such that $\bs{\gamma} = \mathbf{Q}_2\bs{\theta}$.
\end{lemma}

The problem (\ref{EQ:minimization}), is now converted to a conventional penalized regression problem without any constraints:
\begin{equation*}
\min_{\bs{\beta},\bs{\theta}} \left\{\|\mathbf{Y}-\mathbf{Z}\bs{\beta}-\mathbf{B}\mathbf{Q}_{2}\bs{\theta}\|^{2}+\lambda(\mathbf{Q}_{2}\bs{\theta})^{\T}\mathbf{P}(\mathbf{Q}_{2}\bs{\theta})\right\}.
\label{EQ:minimization1}
\end{equation*}

For a fixed penalty parameter $\lambda$, we have
\[
\left(\!
\begin{array}{c}
\widehat{\bs{\beta}}\\
\widehat{\bs{\theta}}
\end{array}
\!\right)
=\left\{\left(\!\begin{array}{cc}
\mathbf{Z}^{\T}\mathbf{Z} &\mathbf{Z}^{\T}\mathbf{B}\mathbf{Q}_{2}\\
\mathbf{Q}_{2}^{\T}\mathbf{B}^{\T}\mathbf{Z} &\mathbf{Q}_{2}^{\T}\mathbf{B}^{\T}\mathbf{B}\mathbf{Q}_{2}
\end{array}\!\right)\!+\!
\left(\!
\begin{array}{cc}
\mathbf{0}&\\
&\lambda\mathbf{Q}_{2}^{\T}\mathbf{P}\mathbf{Q}_{2}
\end{array}
\!\right)\right\}^{-1}
\left(\!
\begin{array}{c}
\mathbf{Z}^{\T}\mathbf{Y}\\
\mathbf{Q}_{2}^{\T}\mathbf{B}^{\T}\mathbf{Y}
\end{array}
\!\right).
\]
Letting
\begin{equation}
\mathbf{V}=\left(
\begin{array}{cc}
\mathbf{V}_{11} & \mathbf{V}_{12} \\
\mathbf{V}_{21} & \mathbf{V}_{22}
\end{array}%
\right)
=\left(
\begin{array}{cc}
\mathbf{Z}^{\T}\mathbf{Z}                                  & \mathbf{Z}^{\T}\mathbf{B}\mathbf{Q}_{2}\\
\mathbf{Q}_{2}^{\T} \mathbf{B}^{\T}  \mathbf{Z} & \mathbf{Q}_{2}^{\T}(\mathbf{B}^{\T} \mathbf{B}+\lambda\mathbf{P})\mathbf{Q}_{2}
\end{array}
\right), \label{DEF:V}
\end{equation}
we have
\[
\left(
\begin{array}{c}
\widehat{\bs{\beta}}\\
\widehat{\bs{\theta}}
\end{array}
\right)
=\mathbf{V}^{-1}
\left(\begin{array}{c}
\mathbf{Z}^{\T}\mathbf{Y}\\
\mathbf{Q}_{2}^{\T}\mathbf{B}^{\T}\mathbf{Y}
\end{array}\right).
\]
Next, we write
\begin{equation}
\mathbf{V}^{-1}\equiv\mathbf{U}=\left(
\begin{array}{cc}
\mathbf{U}_{11} & \mathbf{U}_{12} \\
\mathbf{U}_{21} & \mathbf{U}_{22}
\end{array}
\right)
=\left(
\begin{array}{cc}
\mathbf{U}_{11} & -\mathbf{U}_{11} \mathbf{V}_{12} \mathbf{V}_{22}^{-1} \\
- \mathbf{U}_{22} \mathbf{V}_{21} \mathbf{V}_{11}^{-1} & \mathbf{U}_{22}
\end{array}
\right) , \label{DEF:V-inverse}
\end{equation}
where
\begin{eqnarray}
\mathbf{U}_{11}^{-1}&=&\mathbf{V}_{11}-\mathbf{V}_{12}\mathbf{V}_{22}^{-1}\mathbf{V}_{21}
=\mathbf{Z}^{\T}\left[\mathbf{I}-\mathbf{B}\mathbf{Q}_2\{\mathbf{Q}_2^{\T}(\mathbf{B}^{\T} \mathbf{B}
+\lambda\mathbf{P})\mathbf{Q}_2\}^{-1} \mathbf{Q}_2^{\T}\mathbf{B}^{\T}\right]\mathbf{Z},\nonumber \label{DEF:U11}\\
\mathbf{U}_{22}^{-1}&=&\mathbf{V}_{22}-\mathbf{V}_{21}\mathbf{V}_{11}^{-1}\mathbf{V}_{12}
=\mathbf{Q}_{2}^{\T}\left[\mathbf{B}^{\T} \left\{\mathbf{I}- \mathbf{Z} (\mathbf{Z}^{\T}\mathbf{Z})^{-1}
\mathbf{Z}^{\T}\right\}\mathbf{B}+\lambda\mathbf{P} \right]\mathbf{Q}_{2}.
\label{DEF:U22}
\end{eqnarray}
Then the minimizers of (\ref{DEF:V}) can be given precisely as follows:
\begin{align*}
\widehat{\bs{\beta}}&=
\mathbf{U}_{11}
\mathbf{Z}^{\T}\left(\mathbf{I}-\mathbf{B}\mathbf{Q}_2\mathbf{V}_{22}^{-1} \mathbf{Q}_2^{\T}\mathbf{B}^{\T}\right)\mathbf{Y}
=\mathbf{U}_{11}
\mathbf{Z}^{\T}\left\{\mathbf{I}-\mathbf{B}\mathbf{Q}_2\{\mathbf{Q}_2^{\T}(\mathbf{B}^{\T} \mathbf{B}
+\lambda\mathbf{P})\mathbf{Q}_2\}^{-1} \mathbf{Q}_2^{\T}\mathbf{B}^{\T}\right\}\mathbf{Y},\\
\widehat{\bs{\theta}}&=\mathbf{U}_{22}\mathbf{Q}_{2}^{\T}\mathbf{B}^{\T} \left(\mathbf{I}
-\mathbf{Z} \mathbf{V}_{11}^{-1}\mathbf{Z}^{\T}\right)\mathbf{Y}
=\mathbf{U}_{22}\mathbf{Q}_{2}^{\T}\mathbf{B}^{\T} \left\{\mathbf{I}
-\mathbf{Z} (\mathbf{Z}^{\T}\mathbf{Z})^{-1}\mathbf{Z}^{\T}\right\}\mathbf{Y}.
\end{align*}
Therefore, one obtains the estimators for $\bs{\gamma}$ and $g(\cdot)$, respectively:
\begin{eqnarray}
&\!\!\!\!\!\!&\widehat{\bs{\gamma}}=\mathbf{Q}_{2}\widehat{\bs{\theta}}
=\mathbf{Q}_{2}\mathbf{U}_{22}\mathbf{Q}_{2}^{\T}\mathbf{B}^{\T} \left\{\mathbf{I}
-\mathbf{Z} (\mathbf{Z}^{\T}\mathbf{Z})^{-1}\mathbf{Z}^{\T}\right\}\mathbf{Y},\notag \\
&\!\!\!\!\!\!&\widehat{g}(\mathbf{x})=\mathbf{B}(\mathbf{x})^{\T}\widehat{\bs{\gamma}}=\sum_{\xi \in \mathcal{K}} B_{\xi}(
\mathbf{x})\widehat{\gamma}_{\xi}.
\label{EQN:g_hat}
\end{eqnarray}
The fitted values at the $n$ data points are $\widehat{\mathbf{Y}}=\mathbf{Z}\widehat{\bs{\beta}}+\mathbf{B} \widehat{\bs{\gamma}}
=\mathbf{S}(\lambda)\mathbf{Y}$, where the hat matrix is
\[
\mathbf{S}(\lambda)=\mathbf{Z}\mathbf{U}_{11}
\mathbf{Z}^{\T}\left(\mathbf{I}-\mathbf{B}\mathbf{Q}_2\mathbf{V}_{22}^{-1} \mathbf{Q}_2^{\T}\mathbf{B}^{\T}
\right)+\mathbf{B}\mathbf{Q}_{2}
\mathbf{U}_{22}\mathbf{Q}_{2}^{\T}\mathbf{B}^{\T} \left(\mathbf{I}
-\mathbf{Z} \mathbf{V}_{11}^{-1}\mathbf{Z}^{\T}\right).
\]

In nonparametric regression, the trace of smoothing matrix, $\mathrm{tr}(\mathbf{S}(\lambda))$, is often called the degrees of freedom of the model fit \citep{green1993nonparametric}. It has the rough interpretation as the equivalent number of parameters and can be thought as a generalization of the definition in linear regression. Finally, we can estimate the variance of the error
term, $\sigma^{2}$ by
\begin{equation}
\label{EQ:sigma2_hat}
\hat{\sigma}^{2}=\frac{\|\mathbf{Y}-\widehat{\mathbf{Y}}\|^{2}}{n-\mathrm{tr}(\mathbf{S}(\lambda))}.
\end{equation}

%%%%%%%%%%%%%%%%%%%%%%%%%%%%%%%%%%%%%%%%%%%%%%%%%%%%%%%%%%%%
\vskip 0.05in \noindent \textbf{2.4. Penalty Parameter Selection} \vskip 0.1in

Selecting a suitable value of smoothing parameter $\lambda$ is critical to good model fitting. A large value of $\lambda$ enforces a smoother fitted function with potentially larger fitting errors, while a small value yields a rougher fitted function and potentially smaller fitting errors with sufficiently many data locations. Since the in-sample fitting errors can not gauge the prediction property of the fitted function, one should target a criterion function that mimics the out-of-sample performance of the fitted model. The GCV is such a criterion and is widely used for choosing the penalty parameter. We choose the smoothing parameter $\lambda$ by minimizing the following generalized cross-validation (GCV) criterion
\[
\mathrm{GCV}(\lambda)=\frac{n\|\mathbf{Y}-\mathbf{S}(\lambda)\mathbf{Y}\|^2}{\{n-\mathrm{tr}(\mathbf{S}(\lambda))\}^2},
\]
over a grid of values of $\lambda$. We use the 10-point grid where the values of $\log_{10}(\lambda)$ are equally spaced between $-6$ and $7$ in our numerical experiments.

%%%%%%%%%%%%%%%%%%%%%%%%%%%%%%%%%%%%%%%%%%%%%%%%%%%%%%%%%%%%
%%%%%%%%%%%%%%%%%%%%%%%%%%%%%%%%%%%%%%%%%%%%%%%%%%%%%%%%%%%%
%%%%%%%%%%%%%%%%%%%%%%%%%%%%%%%%%%%%%%%%%%%%%%%%%%%%%%%%%%%%
\setcounter{chapter}{3} \renewcommand{\thetheorem}{3.\arabic{theorem}}
\renewcommand{\thelemma}{3.\arabic{lemma}}
\renewcommand{\theproposition}{3.\arabic{proposition}}
\renewcommand{\thetable}{3.\arabic{table}} \setcounter{table}{0} 
\renewcommand{\thefigure}{3.\arabic{figure}} \setcounter{figure}{0} 
\setcounter{equation}{0} \setcounter{lemma}{0} \setcounter{theorem}{0}
\setcounter{proposition}{0}\setcounter{corollary}{0}
\vskip .12in \noindent \textbf{3. Asymptotic Results} \vskip 0.10in
\label{sec:asymptotics} 

This section studies the asymptotic properties for the proposed estimators. To discuss these properties, we first introduce some notation. For any function $f$ over the closure of domain $\Omega$, denote $\Vert f\Vert _{\infty} =\sup_{\mathbf{x}\in \Omega} |f(\mathbf{x})|$ the supremum norm of function $f$  and $|f|_{\upsilon,\infty}= \max_{i+j=\upsilon}\Vert D_{x_{1}}^{i}D_{x_{2}}^{j}f(\mathbf{x})\Vert _{\infty}$ the maximum norms of all the $\upsilon $th order derivatives of $f$ over $\Omega$. Let
\begin{equation}
W^{\ell,\infty }(\Omega)=\left\{f \mathrm{~on~} \Omega:|f|_{k,\infty}<\infty, 0\le k\le \ell \right\}
\label{DEF:Sobolev}
\end{equation}
be the standard Sobolev space. For any $j=1,\ldots, p$, let $z_{j}$ be the coordinate mapping that maps $\mathbf{z}$ to its $j$-th component so that $z_{j}(\mathbf{Z}_{i})=Z_{ij}$, and let
\begin{equation}
h_j=\mathrm{argmin}_{h\in L^{2}}\|z_{j}-h\|_{L^{2}}^{2}=\mathrm{argmin}_{h\in L^{2}}E\{(Z_{ij}-h(\mathbf{X}_{i})
)^2\} \label{EQ:h_j}
\end{equation}
be the orthogonal projection of $z_{j}$ onto $L^{2}$.

Before we state the results, we make the following assumptions:
\begin{itemize}
	\item[(A1)]  The random variables $Z_{ij}$ are bounded, uniformly in $i=1, \ldots, n$, $j =1,\ldots, p$.
	
	\item[(A2)]  The eigenvalues of $E\left\{\left.(\begin{array}{cc} 1 & \mathbf{Z}_{i}^{\T} \end{array})^{\T}(\begin{array}{cc} 1&
	\mathbf{Z}_{i}^{\T} \end{array})\right| \mathbf{X}_{i}\right\}$ are bounded away from 0.
	
	\item[(A3)]  The noise $\epsilon$ satisfies that $\lim_{\eta\rightarrow\infty}E\left[ \epsilon^{2}I(\epsilon>\eta)\right]=0$.
\end{itemize}

Assumptions (A1)--(A3) are typical in semi-parametric smoothing literature, see for instance \cite{Huang:Zhang:Zhou:07} and \cite{Wang:Liu:Liang:Carroll:11}. The purpose of Assumption (A2) is to ensure that the vector $(1,\mathbf{Z}_{i}^{\T})$ is not multicolinear.

We next introduce some assumptions on the properties of the true bivariate function in model (\ref{model}) and the data locations related to the triangulation $\triangle$.

\begin{itemize}
	\item[(C1)]  The bivariate functions $h_{j}(\cdot)$, $j=1,\ldots,p$, and the true function in model (\ref{model})  $g(\cdot) \in W^{\ell+1 ,\infty}(\Omega)$ in (\ref{DEF:Sobolev}) for an integer $\ell \ge 1$.
	
	\item[(C2)] For every $s\in \mathbb{S}$ and every $\tau \in \triangle$, there exists a positive constant $F_1$, independent of $s$ and $\tau$, such that
	\begin{eqnarray}
	F_1 \|s \|_{\infty,\tau}\le \left\{\sum_{\mathbf{X}_i\in \tau,
		\ i=1, \cdots, n} s\left(\mathbf{X}_i\right)^2\right\}^{1/2},~\mbox{ for all } \tau\in\triangle,
	\label{EQ:F1}
	\end{eqnarray}
	where $\|s\|_{\infty, \tau}$ denotes the supremum norm of $s$ on triangle $\tau$.
	
	\item[(C3)] Let $F_2$ be the largest among the numbers of observations in triangles $\tau\in \triangle$. That is, $F_2>0$ is a constant
	\begin{eqnarray}
	\left\{\sum_{\mathbf{X}_i\in  \tau, \ i=1, \cdots, n}
	s\left(\mathbf{X}_i\right)^2\right\}^{1/2}\le F_2 \|s\|_{\infty,\tau},~\mbox{ for all } \tau\in\triangle.
	\label{EQ:F2}
	\end{eqnarray}
	We further assume that  the constants $F_1$ and $F_2$ in (\ref{EQ:F1}) and (\ref{EQ:F2}) satisfy $F_2/F_1=O(1)$.
	
	\item[(C4)] The number $N$ of the triangles and the sample size $n$ satisfy that $N=Cn^{\gamma}$ for some constant $C>0$ and $1/(\ell+1)\leq \gamma \leq 1/3$.
	
	\item[(C5)] The penalized parameter $\lambda$ satisfies $\lambda=o(n^{1/2}N^{-1})$.
	
	\item[(C6)] Let $\delta_{\triangle}=\max_{\tau\in\triangle}{|\tau|}/{\rho_{\tau}}$, where $\rho_{\tau}$ is the radius of the largest circle inscribed in $\tau$. The triangulation $\triangle $ is $\delta$-quasi-uniform, that is, there exists a positive constant $\delta$ such that the triangulation $\triangle$ satisfies $\delta_{\triangle} \leq \delta$.
\end{itemize}

Condition (C1) describes the requirement for the true bivariate function as usually used in the literature of nonparametric or semiparametric estimation. Condition (C2) ensures the existence of a discrete least squares spline \citep{vonGolitschek:Schumaker:02}, i.e., an unpenalized spline with $\lambda=0$. Although one can get a decent penalized least squares spline fitting without this condition, we need (C2) to study the convergence of bivariate penalized least squares splines. Heuristically, if a triangle $\tau\in \triangle$ near the boundary of $\triangle$ does not contain enough observations, the penalized least square spline will not fit the function well over the triangle $\tau$. Condition (C3) suggests that we should not put too many observations in one triangle. Similar conditions to (C2) and (C3) are used in \cite{vonGolitschek:Schumaker:02} and \cite{Huang:03}. Condition (C4) requires that the number of triangles is above some minimum depending upon the degree of the spline, which is similar to the requirement of \cite{Li:Ruppert:08} in the univariate case. It also ensures the asymptotic equivalence of the theoretical and empirical inner products/norms defined at the beginning of Section 3. Condition (C5) is required to reduce the bias of the spline approximation through ``under smoothing'' and ``choosing smaller $\lambda$".  The study in \cite{Lai:Schumaker:07} shows that the approximation of bivariate spline space over $\triangle$ is dependent on $\delta_{\triangle}$, i.e., the larger the $\delta_{\triangle}$ is, the worse the spline approximation is. That is, the quality of spline approximation is measured by $\delta_{\triangle}$. Condition (C6) suggests the use of more uniform triangulations with a reasonably small $\delta_{\triangle}$ when constructing triangulations. By choosing a set of appropriate vertices, we are able to have a desired triangulation whose $\delta_{\triangle}$ is small enough, say $\delta_\triangle<10$.

To avoid confusion, in the following we let $\bs{\beta}_0$ and $g_{0}$ be the true parameter value and function in model (\ref {model}). The following theorem states that the rate convergence of $\widehat{\bs{\beta}}$ is root-$n$ and $\widehat{\bs{\beta}}$ is asymptotically normal.
\begin{theorem}
	\label{THM:beta-normality}
	Suppose Assumptions (A1)-(A3), (C1)-(C6) hold, then the estimator $\widehat{\bs{\beta}}$ is asymptotically normal, that is,
	$
	(n\bs{\Sigma})^{1/2}(\widehat{\bs{\beta}}-\bs{\beta}_{0}) \rightarrow N(\mathbf{0}, \mathbf{I}),
	$
	where  $\mathbf{I}$ is a $p\times p$ identity matrix,
	\begin{equation}
	\bs{\Sigma}=\sigma^{-2}E\{(\mathbf{Z}_{i}-\widetilde{\mathbf{Z}}_{i})
	(\mathbf{Z}_{i}-\widetilde{\mathbf{Z}}_{i})^{\T}\}
	\label{DEF:Sigma}
	\end{equation}
	with $\widetilde{\mathbf{Z}}_{i}=\left\{h_{1}(\mathbf{X}_{i}), \ldots, h_{p}(\mathbf{X}_{i})\right\}^{\T}$, for $h_{j}(\cdot)$ defined in (\ref{EQ:h_j}), $j=1,\ldots,p$. In addition, $\bs{\Sigma}$ can be consistently estimated by
	\begin{equation}
	\bs{\Sigma}_{n}=\frac{1}{n\widehat{\sigma}^{2}}\sum_{i=1}^{n} (\mathbf{Z}_{i}-\widehat{\mathbf{Z}}_{i})
	(\mathbf{Z}_{i}-\widehat{\mathbf{Z}}_{i})^{\T} =\frac{1}{n\widehat{\sigma}^{2}} (\mathbf{Z}-\widehat{\mathbf{Z}})^{\T}(\mathbf{Z}-\widehat{\mathbf{Z}}).
	\label{DEF:Sigma_n}
	\end{equation}
	where $\widehat{\mathbf{Z}}_{i}$ is the $i$-th column of
	$\widehat{\mathbf{Z}}^{\T}=\mathbf{Z}^{\T}\mathbf{B}\mathbf{Q}_2\mathbf{V}_{22}^{-1}
	\mathbf{Q}_2^{\T}\mathbf{B}^{\T}$ and $\widehat{\sigma}^{2}$ is given by (\ref{EQ:sigma2_hat}).
\end{theorem}
The results in Theorem \ref{THM:beta-normality} enable us to construct confidence intervals for the parameters. The next theorem provides the global convergence of the nonparametric estimator $\widehat{g}(\cdot)$.
\begin{theorem}
	\label{THM:g-convergence}
	Suppose Assumptions (A1)-(A3), (C1)-(C6) hold, then the bivariate penalized estimator $\widehat{g}(\cdot)$ in (\ref{EQN:g_hat}) is consistent with the true function $g_{0}$, and satisfies that
	\[
	\Vert \widehat{g}-g_{0}\Vert _{L^2}=O_{P}\left(
	\frac{\lambda }{n\left| \triangle \right| ^{3}}|g_{0}|_{2,\infty}
	+\left(1+\frac{\lambda }{n\left| \triangle \right| ^{5}}\right) \frac{F_{2}}{F_{1}}
	|\triangle |^{\ell +1}|g_{0}|_{\ell,\infty}+\frac{1}{\sqrt{n}|\triangle|}
	\right).
	\]
\end{theorem}
The proofs of the above two theorems are given in Appendix. We notice that the rate of convergence given in Theorem \ref{THM:g-convergence} is the same as those for nonparametric spline regression without including the covariate information obtained in \cite{Lai:Wang:13}.

%%%%%%%%%%%%%%%%%%%%%%%%%%%%%%%%%%%%%%%%%%%%%%%%%%%%%%%%%%%%%%
%%%%%%%%%%%%%%%%%%%%%%%%%%%%%%%%%%%%%%%%%%%%%%%%%%%%%%%%%%%%%%
%%%%%%%%%%%%%%%%%%%%%%%%%%%%%%%%%%%%%%%%%%%%%%%%%%%%%%%%%%%%%%
\setcounter{chapter}{4}\setcounter{equation}{0} 
\renewcommand{\thetable}{4.\arabic{table}} \setcounter{table}{0} 
\renewcommand{\thefigure}{4.\arabic{figure}} \setcounter{figure}{0} 
\vskip .12in \noindent \textbf{4. Simulation} \label{sec:simulation} \vskip 0.10in

In this section, we carry out a numerical study to assess the performance of the proposed estimators using the bivariate penalized splines over triangulations (BPST) over a horseshoe domain. We compare the BPST with filtered kriging (KRIG), thin plate splines (TPS), linear finite elements method (FEM) in \cite{Sangalli:Ramsay:Ramsay:13} and the geodesic low rank thin plate splines (GLTPS) in \cite{Wang:Ranalli:07}.  More simulation studies can be found in Appendix B.

For $50\times 20$ grid points on the domain, we simulate data as follows, the response variable $Y$ is generated from the following PLM:
\begin{equation*}
Y=\beta_{1}Z_{1}+\beta_{2}Z_{2}+g(X_1,X_2)+\epsilon.
\end{equation*}
Figure \ref{FIG:eg1_true} (a) shows the surface of the true function $g(\cdot)$, which was used by \cite{Wood:Bravington:Hedley:08} and \cite{Sangalli:Ramsay:Ramsay:13}. The random error, $\epsilon$, is generated from an $N(0,\sigma_{\epsilon}^{2})$ distribution with $\sigma_{\epsilon}=0.5$. In addition, we set the parameters as $\beta_{1}=-1$, $\beta_{2}=1$. For the design of the explanatory variables, $Z_{1}$ and $Z_{2}$, two scenarios are considered based on the relationship between the location variables $(X_1,X_2)$ and covariates $(Z_1, Z_2)$. Under both scenarios, $Z_{1}\sim uniform[-1,1]$. On the other hand, the variable $Z_{2}=\cos[4\pi (\rho(X_{1}^{2}+X_{2}^{2})+(1-\rho)U)]$ where $U\sim uniform[-1,1]$ and is independent from $(X_1,X_2)$ as well as $Z_{1}$. We consider both independent design: $\rho=0.0$ and dependent design: $\rho=0.7$ in this example. Under both scenarios, 100 Monte Carlo replicates are generated. Figure \ref{FIG:eg1_true} (b) demonstrates the sampled location points of replicate 1. For each replication, we randomly sample $n=200$ locations uniformly from the grid points inside the horseshoes domain.

\begin{figure}[htbp]
	\begin{center}
		\begin{tabular}{ccc}
			\includegraphics[width=5.75cm, height=4.5cm]{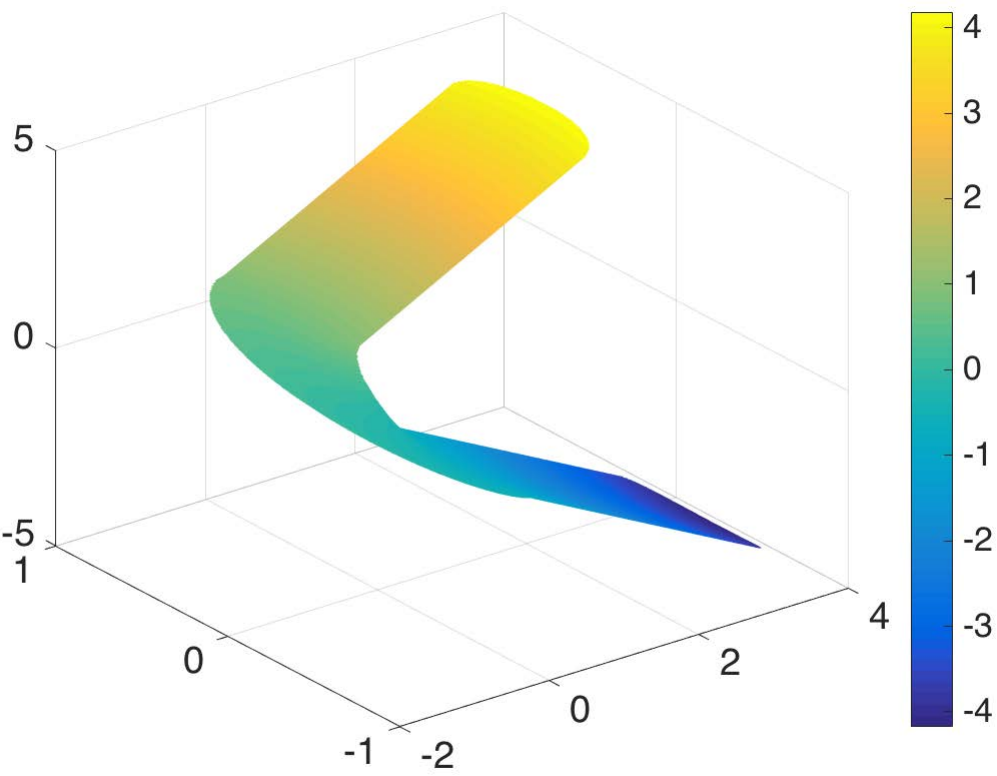} & &
			\includegraphics[height=4cm]{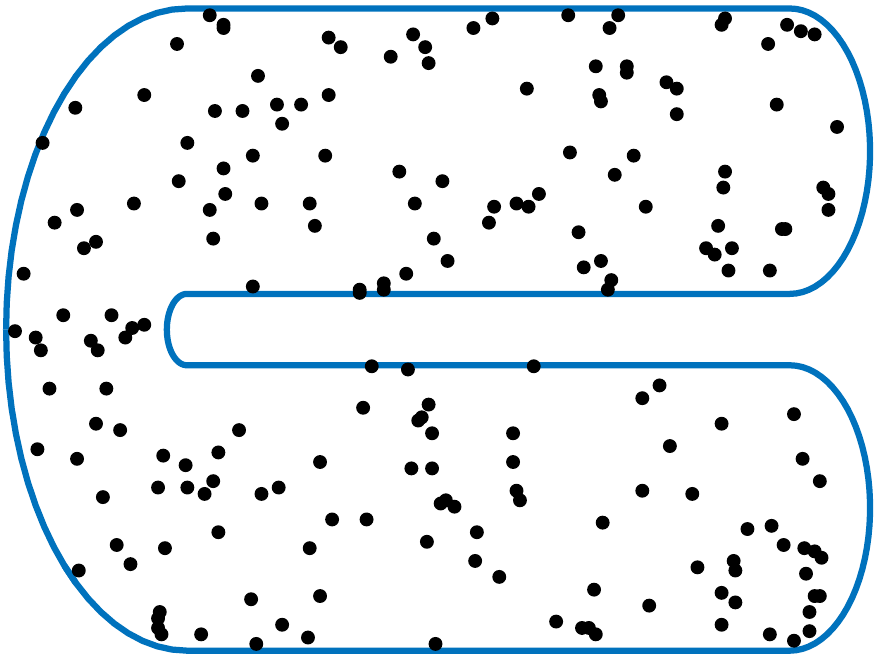} \\[-3pt]
			(a) & & (b)  \\
		\end{tabular}
	\end{center}
	\caption{(a) true function of $g(\cdot)$; (b) sampled location points of replicate 1.}
	\label{FIG:eg1_true}
\end{figure}

Figure \ref{FIG:eg1_tri} (a)-(c) illustrate three different triangulations used in the BPST method. In the first triangulation ($\triangle_1$), we use $89$ triangles ($73$ vertices), and there are $158$ triangles ($114$ vertices) and $286$ triangles ($186$ vertices) in $\triangle_2$ and  $\triangle_3$, respectively. To implement the TPS and KRIG methods, we use the R package \textit{fields} under the standard implementation setting of \citep{Furrer:Nychka:Sainand:11}. For KRIG, we try different covariance structures, and we choose the Mat\'{e}rn covariance with smoothness parameter $\nu=1$, which gives the best prediction.  For the GLTPS, following \cite{Wang:Ranalli:07}, we also use 40 knots with locations selected using the ``\texttt {cover.design}" method in the package \textit{fields}.  For all the methods requiring a smoothing or roughness parameter, GCV is used to choose the values of the parameter.

\begin{figure}[htbp]
	\begin{center}
		\begin{tabular}{ccc}
			\includegraphics[height=3.25cm]{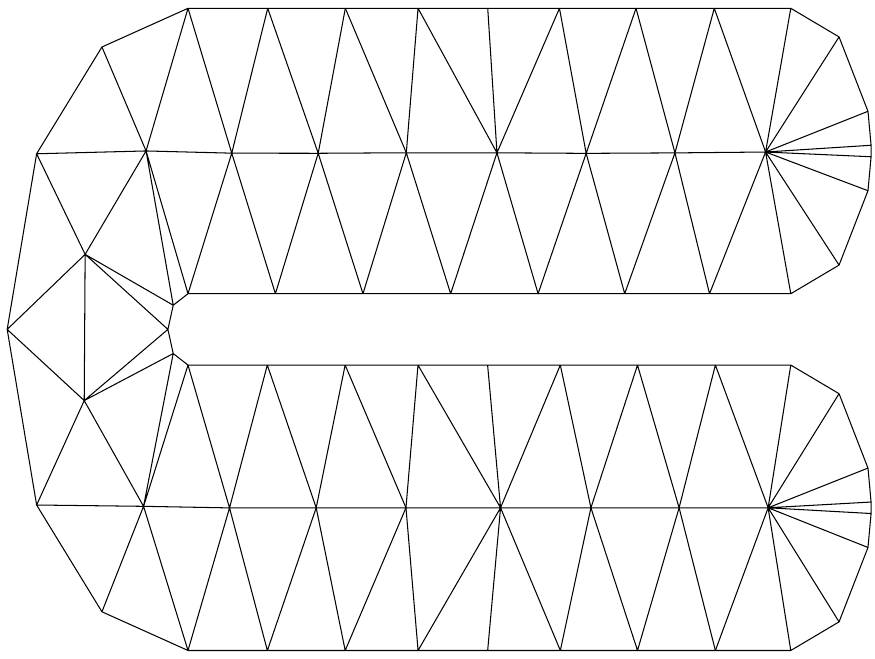} & \includegraphics[height=3.25cm]{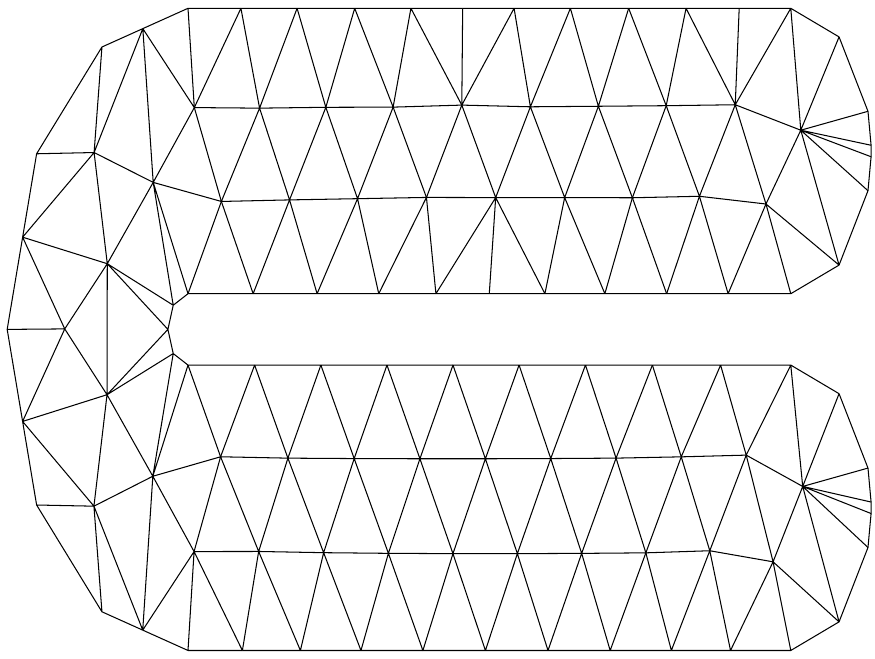} & \includegraphics[height=3.25cm]{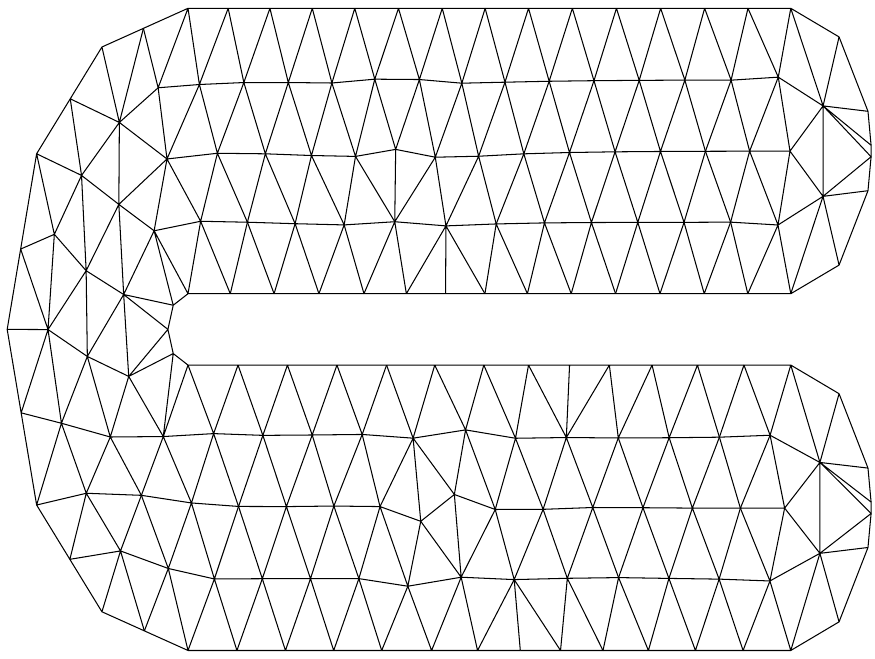} \\
			(a) triangulation $\triangle_1$ & (b) triangulation $\triangle_2$ & (c) triangulation $\triangle_3$ \\
		\end{tabular}
	\end{center}
	\caption{Three different triangulations on the horseshoe domain.}
	\label{FIG:eg1_tri}
\end{figure}

To see the accuracy of the estimators, we compute the root mean squared error (RMSE) for each of the components based on $100$ Monte Carlo samples. Table \ref{TAB:eg1_1} shows the RMSEs of the estimate of the parameters $\beta_{1}$, $\beta_{2}$, $\sigma_{\varepsilon}$. The RMSE for the nonlinear function $g (\cdot)$ is computed as the average of $[1000^{-1}\sum_{i=1}^{1000}\{\widehat{g}(\mathbf{X}_{i})-g(\mathbf{X}_{i})\}^{2}]^{1/2}$ based on $1000=50\times 20$ grid points over the 100 Monte Carlo replications. From Table \ref{TAB:eg1_1}, one sees that BPST produces the best estimation of the nonlinear function $g(\cdot$), followed by the GLTPS and FEM. The RMSE is nearly constant for all three triangulations, which shows that $\triangle_1$ might be sufficiently fine to capture the feature in the dataset. It also suggests that, when this minimum number of triangles is reached, further refining the triangulation will have little effect on the fitting process, but makes the computational burden unnecessarily heavy. Table \ref{TAB:eg1_1} also provides the 10-fold cross-validation root mean squared prediction error (CV-RMSPE) for the response variable, defined as $\left\{{n}^{-1}\sum_{m = 1}^{10} \sum_{i \in \kappa_{m}} (\widehat{Y}_i-Y_i) ^ 2\right\}^{1/2}$ over the 100 Monte Carlo replications, where $\kappa_1, \ldots, \kappa_{10}$ comprise a random partition of the dataset into $10$ disjoint subsets of equal size. The CV-RMSPE also shows the superior performance of the BPST method as it provides the most accurate predictions.

%%%%%%%%%%%%%%%%%%%%%%%%%%%%%%%%%%%%%%%%%%%%%%%%%%%%%%%%%%%%%%%%%%%%%%
\begin{table}[htbp]
	\caption{\label{TAB:eg1_1}Root mean squared errors of the estimates.}
	\centering
	\begin{tabular}{clccccc}\hline\hline
		\multirow{2}{*}{$\rho$}   & \multirow{2}{*}{Method} & \multicolumn{4}{c}{RMSE} & \multicolumn{1}{c}{CV-RMSPE}\\ \cline{3-6}
		& &$\beta_{1}$ &$\beta_{2}$ &$\sigma_{\varepsilon}$ &$g(\cdot)$ & Y\\ \hline
		
		\multirow{7}{*}{0.0} &KRIG &0.0582 &0.0433 &0.0455 &0.3972 &0.6728\\
		&TPS &0.0543 &0.0426 &0.0365 &0.3013 &0.6037\\
		&GLTPS &0.0625 &0.0544 &0.0233 &0.1565 &0.5326\\
		&FEM &0.0560 &0.0480 &0.0348 &0.1558 &0.5333\\
		&BPST ($\triangle_1$) &0.0526 &0.0498 &0.0209 &0.1473 &0.5299\\
		&BPST ($\triangle_2$) &0.0483 &0.0489 &0.0220 &0.1483 &0.5210\\
		&BPST ($\triangle_3$) &0.0544 &0.0544 &0.0222 &0.1458 &0.5248\\\hline
		
		\multirow{7}{*}{0.7} &KRIG &0.0586 &0.0440 &0.0460 &0.3973 &0.6728\\
		&TPS &0.0547 &0.0402 &0.0363 &0.3010&0.6038\\
		&GLTPS &0.0612 &0.0411 &0.0220 &0.1553 &0.5326\\
		&FEM &0.0562 &0.0597 &0.0352 &0.1567 &0.5336\\
		&BPST ($\triangle_1$) &0.0521 &0.0563 &0.0209 &0.1473 &0.5294\\
		&BPST ($\triangle_2$) &0.0481 &0.0502 &0.0222 &0.1479 &0.5209\\
		&BPST ($\triangle_3$) &0.0543 &0.0479 &0.0220 &0.1457 &0.5251\\  \hline\hline
	\end{tabular}
\end{table}

Figures \ref{FIG:eg1_2} shows the estimated functions over a grid of $500\times 200$ points via different methods for replicate 1 for $\rho=0.0$. Since such high resolution prediction is computationally too expensive for the GLTPS, so the prediction map for the GLTPS is based on $100\times 40$ grid points. From those plots, one sees that the BPST and GLTPS estimates look visually better than the other four estimates. In addition, one notices that there is a ``leakage effect'' in KRIG and TPS estimates, and this poor performance is because KRIG and TPS do not take the complex boundary into any account and smooth across the gap inappropriately. Finally, one sees that the BPST estimators based on the three different triangulations are very similar, which agrees with our findings for penalized splines that the number of triangles is not very critical for the fitting as long as it is sufficiently large enough to capture the pattern and features of the data. Similar estimation results are obtained for the case $\rho=0.7$. Some sample estimated functions are presented in Figure \ref{FIG:eg1_3} in Appendix B to save space.

%%%%%%%%%%%%%%%%%%%%%%%%%%%%%%%%%%%%%%%%%%%%%%%%%%%%%%%%%%%%%%%%%%%%%%
\begin{figure}[htbp]
	\begin{center}
		\begin{tabular}{cccc}
			\includegraphics[height=2.85cm]{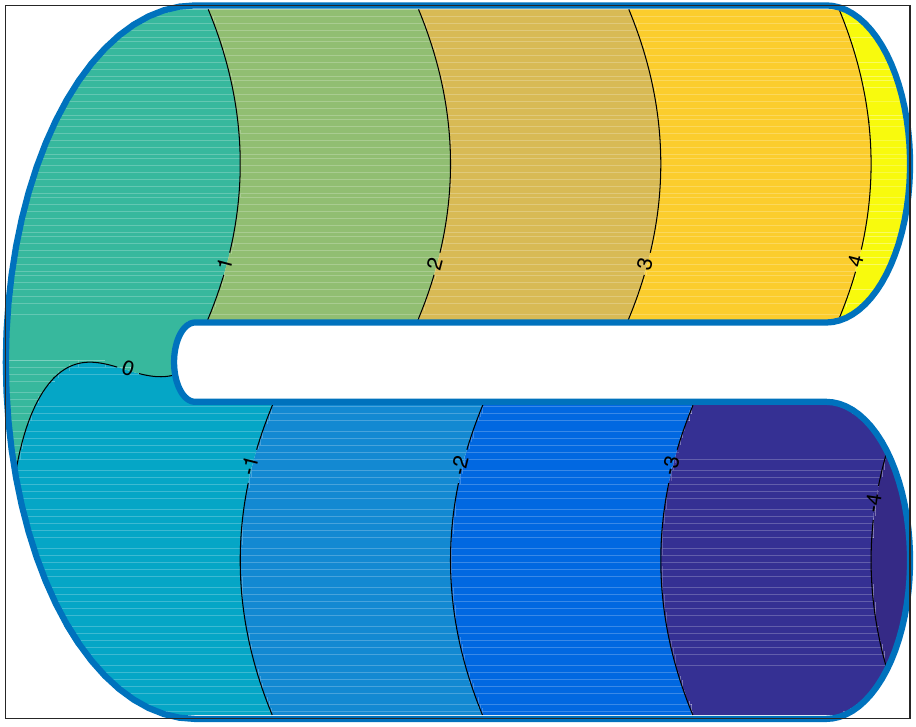} & \includegraphics[height=2.85cm]{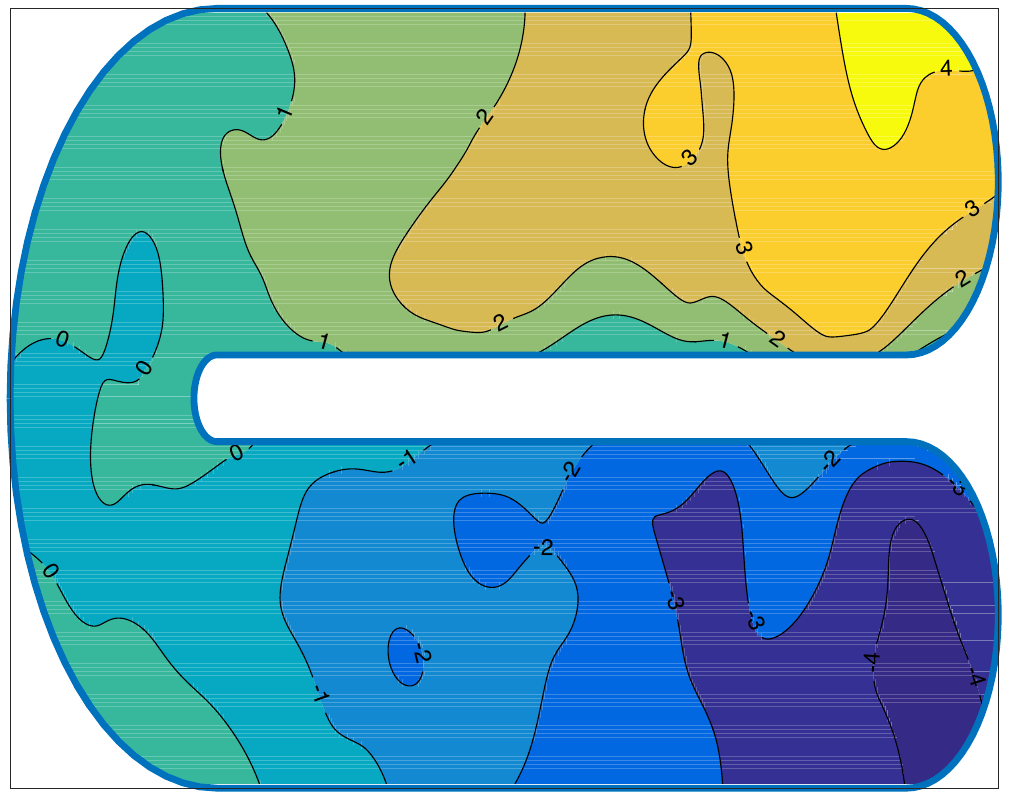} &\includegraphics[height=2.85cm]{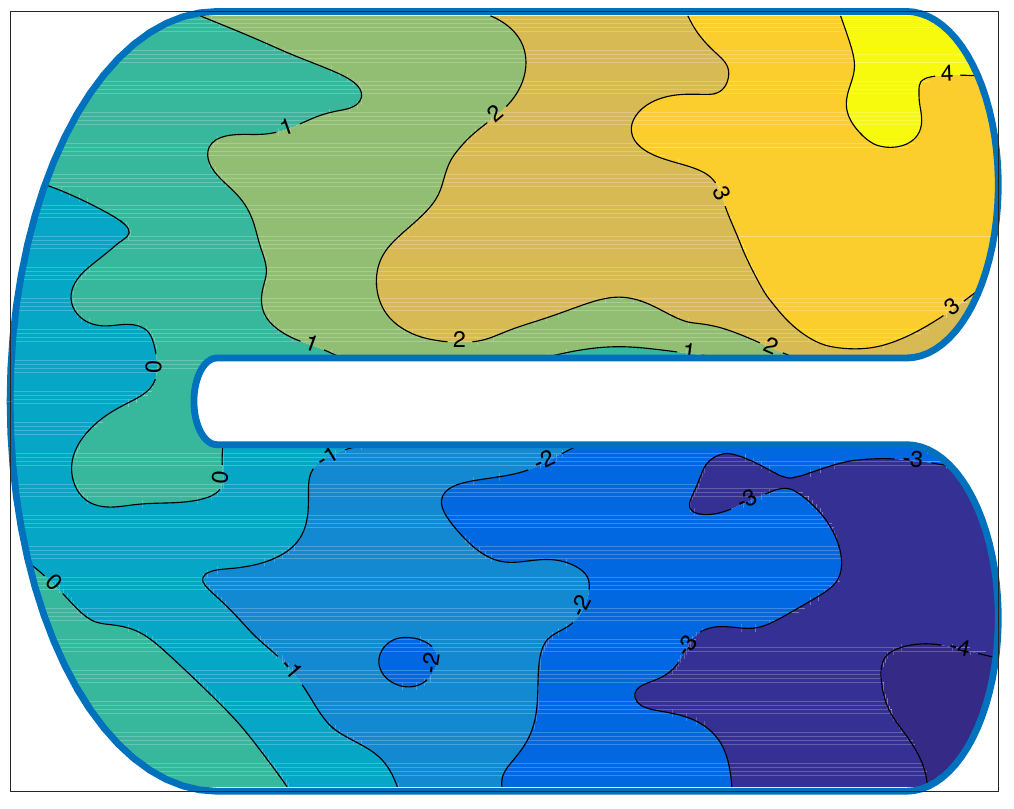}  &
			\includegraphics[height=2.85cm]{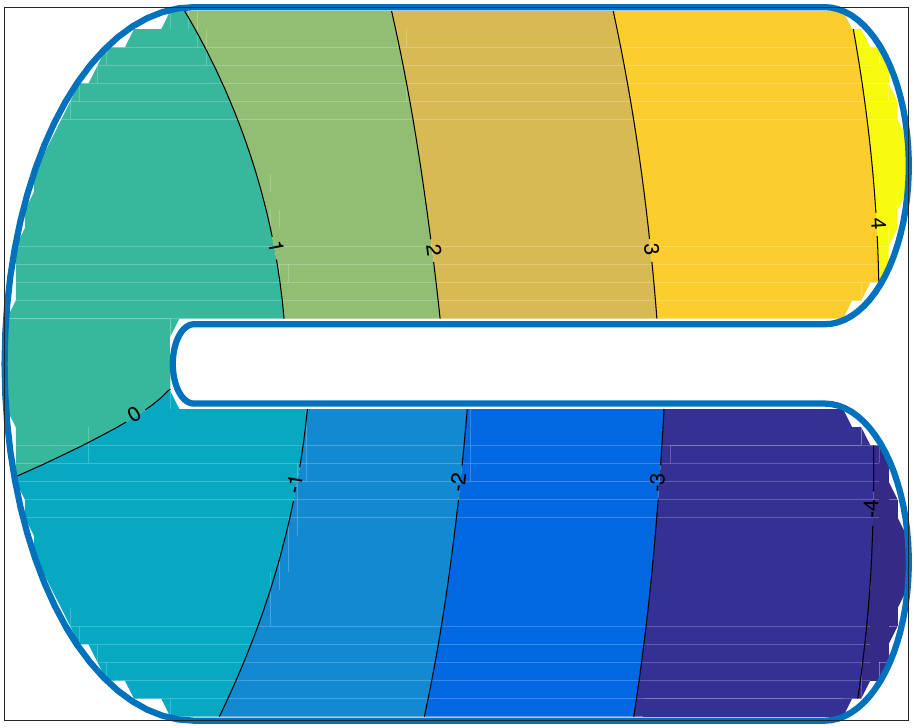}\\
			(a) True Contour & (b) KRIG & (c) TPS & (d) GLTPS\\[6pt]
			\includegraphics[height=2.85cm]{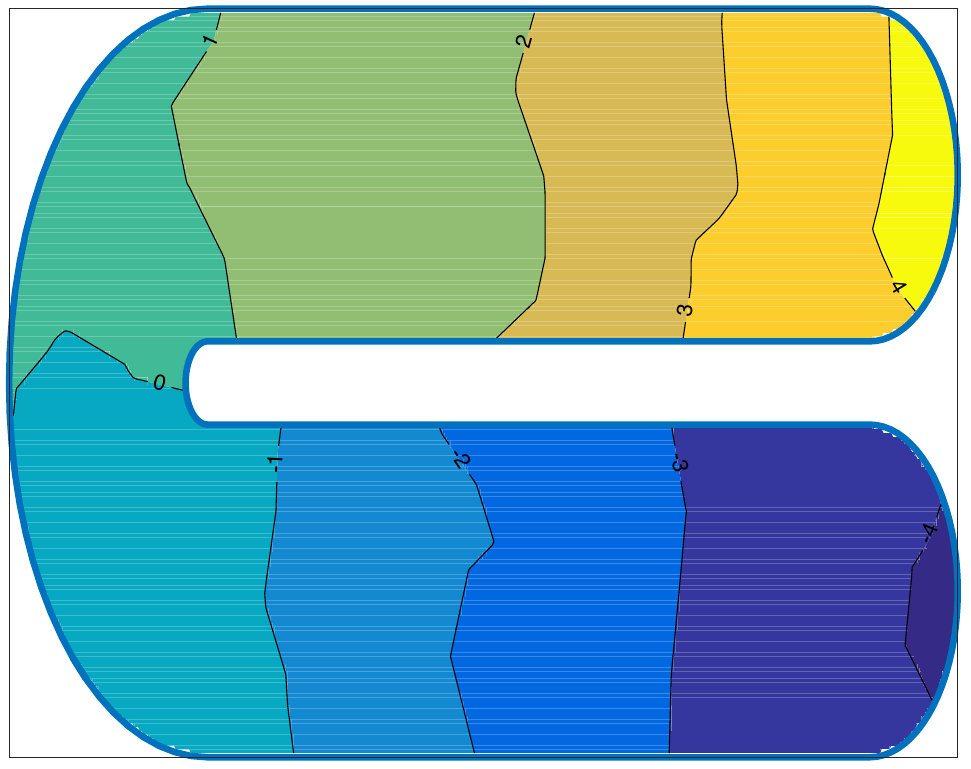} & \includegraphics[height=2.85cm]{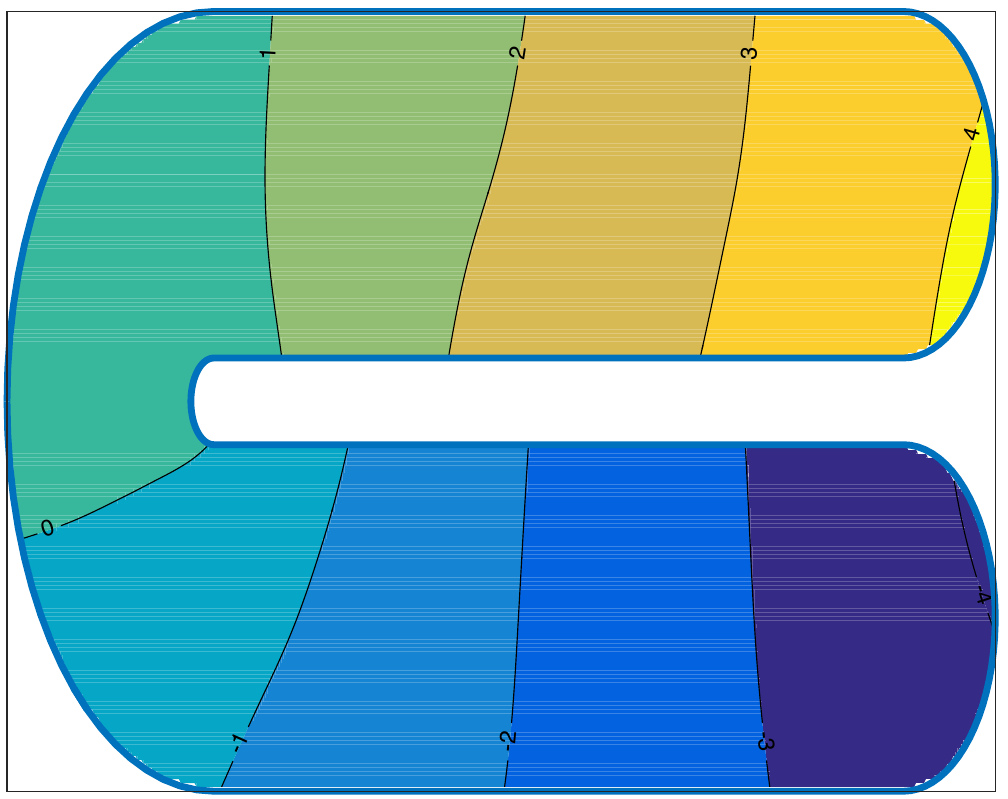} &
			\includegraphics[height=2.85cm]{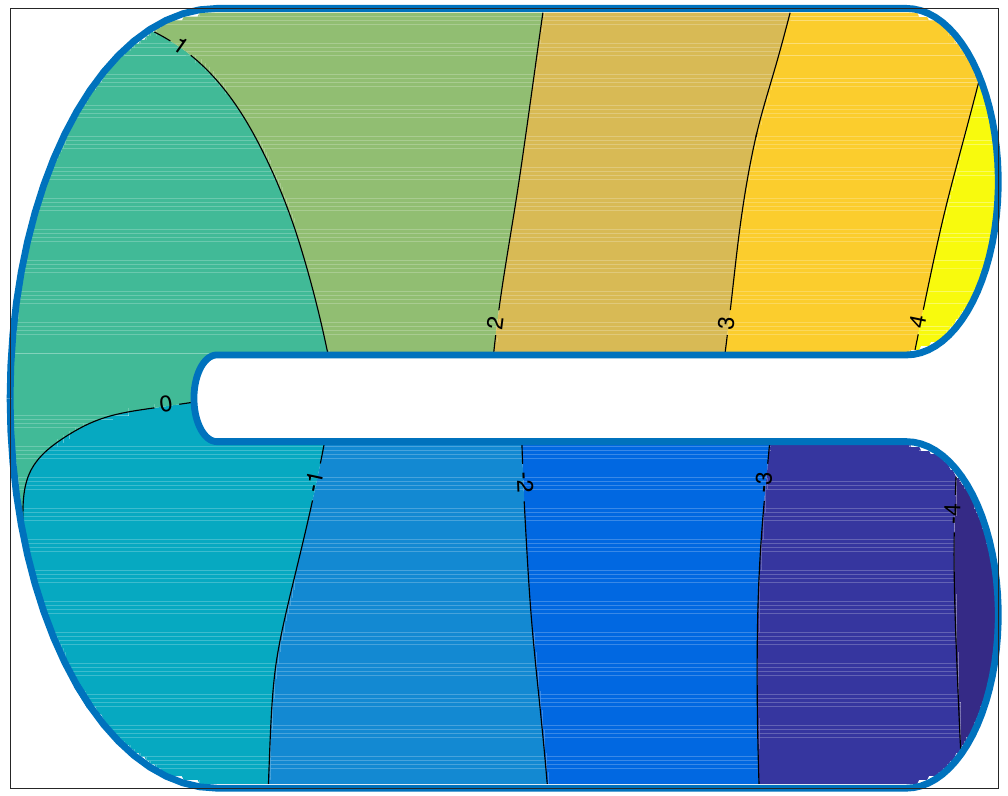} &
			\includegraphics[height=2.85cm]{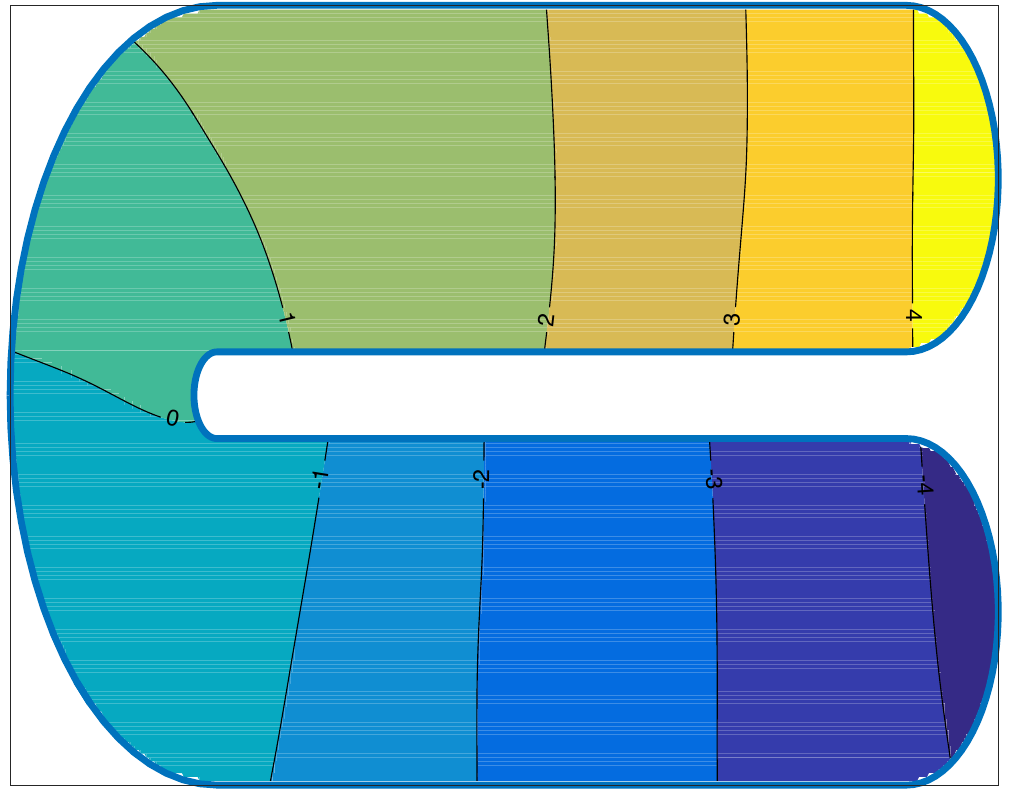}\\
			(e) FEM & (f) BPST ($\triangle_1$) & (g) BPST ($\triangle_2$) & (h) BPST ($\triangle_3$) \\
		\end{tabular}
	\end{center}
	\caption{Contour maps for the true function and its estimators ($\rho=0.0$).}
	\label{FIG:eg1_2}
\end{figure}

%%%%%%%%%%%%%%%%%%%%%%%%%%%%%%%%%%%%%%%%%%%%%%%%%%%%%%%%%%%%%%%%%%%%%%
Next we test the accuracy of the standard error (SE) formula in (\ref{DEF:Sigma_n}) for $\hat{\beta}_{1}$ and $\hat{\beta}_{2}$, and the results are listed in Table \ref{TAB:eg1_2}. The standard deviations of the estimated parameters are computed based on $100$ replications, which can be regarded as the true standard errors (column labeled ``$\mathrm{SE}_{\mathrm{mc}}$") and compared with the mean and median of the $100$ estimated standard errors calculated using (\ref{DEF:Sigma_n}) (columns labeled `` $\mathrm{SE}_{\mathrm{mean}}$" and ``$\mathrm{SE}_{\mathrm{median}}$", respectively). The column labeled `` $\mathrm{SE}_{\mathrm{mad}}$" is the interquartile range of the $100$ estimated standard errors divided by $1.349$, which is a robust estimate of the standard deviation. From Table \ref{TAB:eg1_2} one observes that the averages or medians of the SEs calculated using the formula are very close to the true standard deviations, which confirms the accuracy of the proposed SE formula.

%%%%%%%%%%%%%%%%%%%%%%%%%%%%%%%%%%%%%%%%%%%%%%%%%%%%%%%%%%%%%%%
\begin{table}[htbp]
	\caption{\label{TAB:eg1_2}
		Standard error estimates of the coefficients via BPST ($\triangle_2$).}
	\centering
	\begin{tabular}{cccccc}\hline\hline
		$\rho$ &Parameter &$\mathrm{SE}_{\mathrm{mc}}$ &$\mathrm{SE}_{\mathrm{mean}}$ &$\mathrm{SE}_{\mathrm{median}}$ &
		$\mathrm{SE}_{\mathrm{mad}}$\\ \hline
		
		\multirow{2}{*}{0.0} &$\beta_{1}$ &0.0479 &0.0651 &0.0654 &0.0031\\
		&$\beta_{2}$ &0.0446 &0.0532 &0.0530 &0.0028\\ \hline
		
		\multirow{2}{*}{0.7} &$\beta_{1}$ &0.0477 &0.0651 &0.0653 &0.0029\\
		&$\beta_{2}$ &0.0420 &0.0518 &0.0522 &0.0024\\ \hline\hline
		
	\end{tabular}
\end{table}

%%%%%%%%%%%%%%%%%%%%%%%%%%%%%%%%%%%%%%%%%%%%%%%%%%%%%%%%%%%%%%%
In terms of the computational complexity, since the GLTPS technique is largely based on Floyd's algorithm, it has cubic time complexity \citep{miller2014finite} like the ordinary kriging. In contrast, TPS, FEM and BPST can be formulated as one single least squares problem, thus, the computing is very easy and fast. Taking the prediction as an example, we find that as the prediction size increases (sample size is fixed), the computation time for GLTPS and KRIG increases dramatically, while BPST provides an almost linear complexity of the prediction size. On a standard PC with processor Core i5 @2.9GHz CPU and 16.00GB RAM, the BPST($\triangle_{1}$) prediction over $2500\times 1000$ grid points needs only 10 seconds of computing, BPST($\triangle_{2}$) and BPST($\triangle_{3}$) with finer triangulations takes just a few seconds longer than BPST($\triangle_{1}$). However, the GLTPS usually has to spend hours to complete one estimation and prediction at the $100 \times 40$ resolution level. In addition, in our numerical study, we notice that KRIG requires a large amount of memory. When the prediction resolution goes is finer than $2500\times 1000$, KRIG will crash on a standard PC due to lack of memory.

%%%%%%%%%%%%%%%%%%%%%%%%%%%%%%%%%%%%%%%%%%%%%%%%%%%%%%%%%%%%%%
%%%%%%%%%%%%%%%%%%%%%%%%%%%%%%%%%%%%%%%%%%%%%%%%%%%%%%%%%%%%%%
%%%%%%%%%%%%%%%%%%%%%%%%%%%%%%%%%%%%%%%%%%%%%%%%%%%%%%%%%%%%%%
\setcounter{chapter}{5} \renewcommand{\thetheorem}{5.\arabic{theorem}}
\renewcommand{\thelemma}{5.\arabic{lemma}}
\renewcommand{\theproposition}{5.\arabic{proposition}}
\renewcommand{\thetable}{5.\arabic{table}} \setcounter{table}{0} 
\renewcommand{\thefigure}{5.\arabic{figure}} \setcounter{figure}{0} 
\setcounter{equation}{0} \setcounter{lemma}{0} \setcounter{theorem}{0}
\setcounter{proposition}{0}\setcounter{corollary}{0}
\vskip .12in \noindent \textbf{5. Application to Mercury Concentration Studies in New Hampshire Estuary} \vskip 0.10in
\label{sec:application} 

In this section we apply the proposed method to map the mercury in sediment concentration over the estuary in New Hampshire; see Figure \ref{FIG:MC_vtrue} (a) for a regional map of the estuary. Mercury contamination is a significant public health and environmental problem. When released into the environment, mercury accumulates in water laid sediments, is ingested by fish and passed along the food chain to humans. Several rivers flowing into the Great Bay are contaminated with mercury according to the new Environment New Hampshire report. Estuaries such as Great Bay are ideal locations for the accumulation of contaminants like mercury that settle out from inputs of the surrounding watershed \citep{brown2015effect}. The coastal monitoring program -- National Coastal Assessment -- in the US Environmental Protection Agency (EPA) and the New Hampshire Department of Environmental Services have developed surveys that can reveal useful information on the status and trends of contaminants.

The spatial dataset in our study consists the mercury concentrations surveyed in the years 2000/2001 and 2003 at 97 locations in the largest estuary in New Hampshire; see Figure \ref{FIG:MC_vtrue} (b) for different measurements of mercury concentrations at different sampled locations. To assist decision-makers to develop effective environmental protection strategies, it is critical to provide the measurement of mercury at spatial scales much finer than those at which the mercury was monitored.

This dataset has been studied in \cite{Wang:Ranalli:07} via the GLTPS. Following \cite{Wang:Ranalli:07}, we consider a PLM with a linear term for the year effect (Year $=0$, if survey was conducted in year 2000/2001; and Year $=1$ if survey was conducted in 2003):
\begin{equation}
\textrm{Mercury Concentration}=\beta\textrm{Year}+g(\textrm{Latitude, Longitude}).
\label{app:plm}
\end{equation}
To fit model (\ref{app:plm}), we use five different methods: KRIG, TPS, GLTPS, FEM and BPST. For KRIG, we choose the Mat\'{e}rn covariance structure to fit the model. The GLTPS is calculated using the setting $k = 5$ as in \cite{Wang:Ranalli:07}.  For BPST and FEM, the smoothing or roughness parameter is selected by the GCV. Figure \ref{FIG:MC_tri} shows the triangulation adopted by the BPST. Table \ref{TAB:APPest} summarizes the coefficient estimation results based on different methods.

%%%%%%%%%%%%%%%%%%%%%%%%%%%%%%%%%%%%%%%%%%%%%%%%%%%%%%%%%%%%%%%%%%%%%%
\begin{figure}[htbp]
	\begin{center}
		\includegraphics[height=10cm]{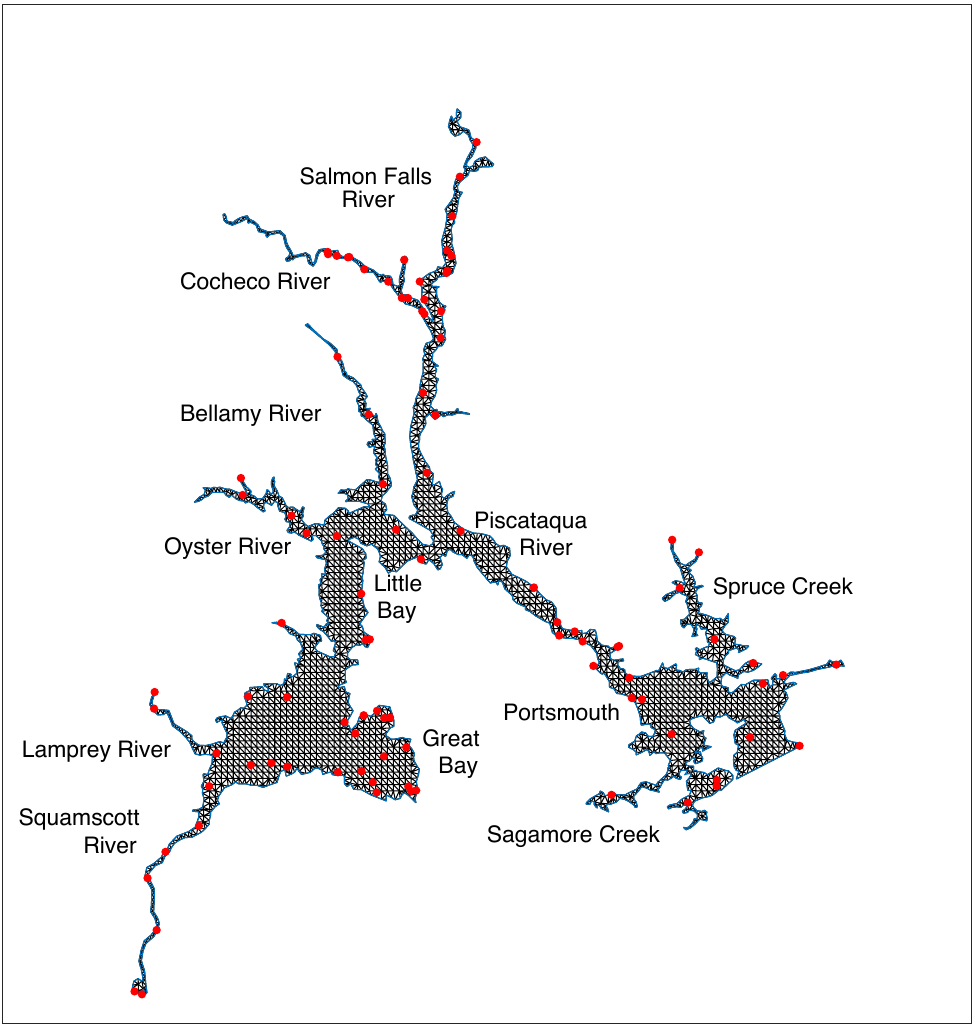}
	\end{center} \vskip -.4in
	\caption{Domain triangulation for estuaries in New Hampshire.}
	\label{FIG:MC_tri}
\end{figure}

%%%%%%%%%%%%%%%%%%%%%%%%%%%%%%%%%%%%%%%%%%%%%%%%%%%%%%%%%%%%%%%
\begin{table}[htbp]
	{\textwidth=3\linewidth
		\caption{\label{TAB:APPest} Estimated coefficients with the standard errors (SE)}
		\centering
		\begin{tabular}{lccccccc} \hline\hline
			&KRIG &TPS  &GLTPS  &FEM &BPST\\ \hline
			Year &0.096 &0.095 &0.051 &0.076 &0.044\\
			%SE &0.0334 &0.0339 &0.0346 &0.0381 &0.0442\\ \hline\hline
			SE &0.03 &0.03 &0.02 &0.04 &0.04\\ \hline\hline
		\end{tabular}\\}
\end{table}

%%%%%%%%%%%%%%%%%%%%%%%%%%%%%%%%%%%%%%%%%%%%%%%%%%%%%%%%%%%%%%%
The Great Bay estuary is a tidally-dominated system and is the drainage confluence of the Lamprey River and Squamscott River. Four additional rivers flowing into the system include the Cocheco, Salmon Falls, Bellamy, and Oyster rivers. Mercury deposited in the estuaries in New Hampshire is both emitted from in-state sources and carried here from sources upwind. Emissions upwind of New Hampshire are primarily attributable to coal-fired utilities and municipal and medical waste incinerators in the Northeast and Midwest \citep{abbott2008atmospheric}. The spatial distribution in Figure \ref{FIG:MC_vtrue} (b) shows generally higher values in the Salmon Falls River and Cocheco River, and lower values in the Piscataqua River and the Portsmouth area, and some localized low spots the Great Bay.

Prediction maps at $20\times 20$ m resolution level using different methods are shown in Figure \ref{FIG:MC_pred}. The computation-intensive GLTPS procedure has a problem in making such a high-resolution prediction, so we decrease its resolution to $150\times 150$ m.  All methods in Figure \ref{FIG:MC_pred} have identified relatively high mercury contamination in the Salmon Falls River and Cocheco River, which is consistent with known historical pollution sources  \citep{abbott2008atmospheric}. Figure \ref{FIG:MC_pred} also illustrates the overspill from the Northern part to the middle area when an ordinary spatial smoothing (such as KRIG and TPS) is used, as it smoothes across the Salmon Falls River and Cocheco River with high concentration levels in the northern part. This problem is mitigated for GLTPS and FEM. The BPST smoother does not show signs of leakage in the Piscataqua River and the Portsmouth area of the estuaries, as other methods do. Note the way in which the KRIG and TPS smooth, inappropriately, across the east coast of the Great Bay, so that relatively high mercury concentrations are estimated for the Portsmouth in the southeastern part of the estuaries. The poor prediction performance of KRIG and TPS suggests that we should not assume that densities in geographically neighboring areas will be similar if these areas are in fact separated by physical barriers.

%%%%%%%%%%%%%%%%%%%%%%%%%%%%%%%%%%%%%%%%%%%%%%%%%%%%%%%%%%%%%%%%%%%%%%
\begin{figure}[htbp]
	\begin{center}
		\begin{tabular}{cc}
			\includegraphics[scale=0.25]{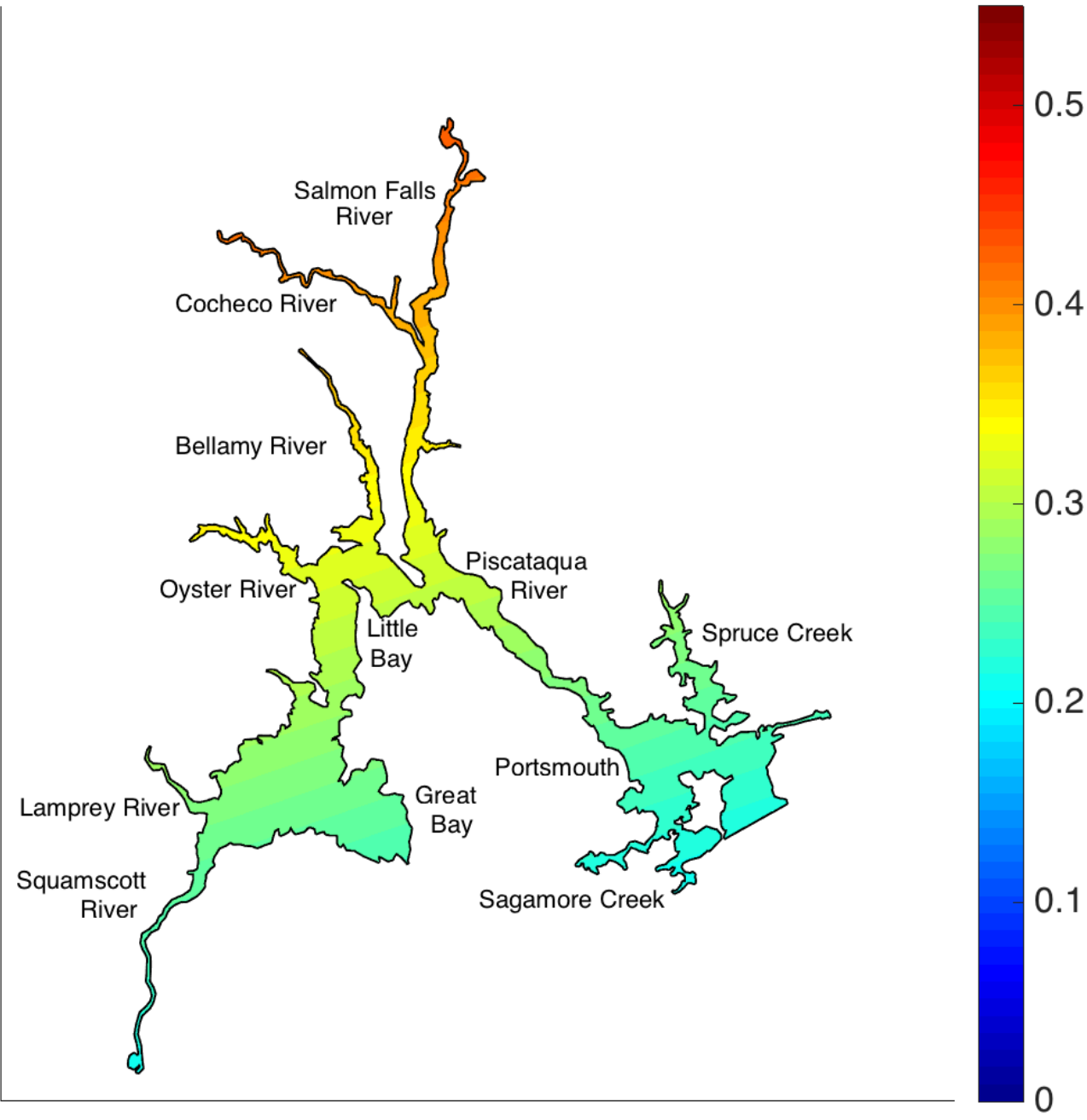} &
			\includegraphics[scale=0.25]{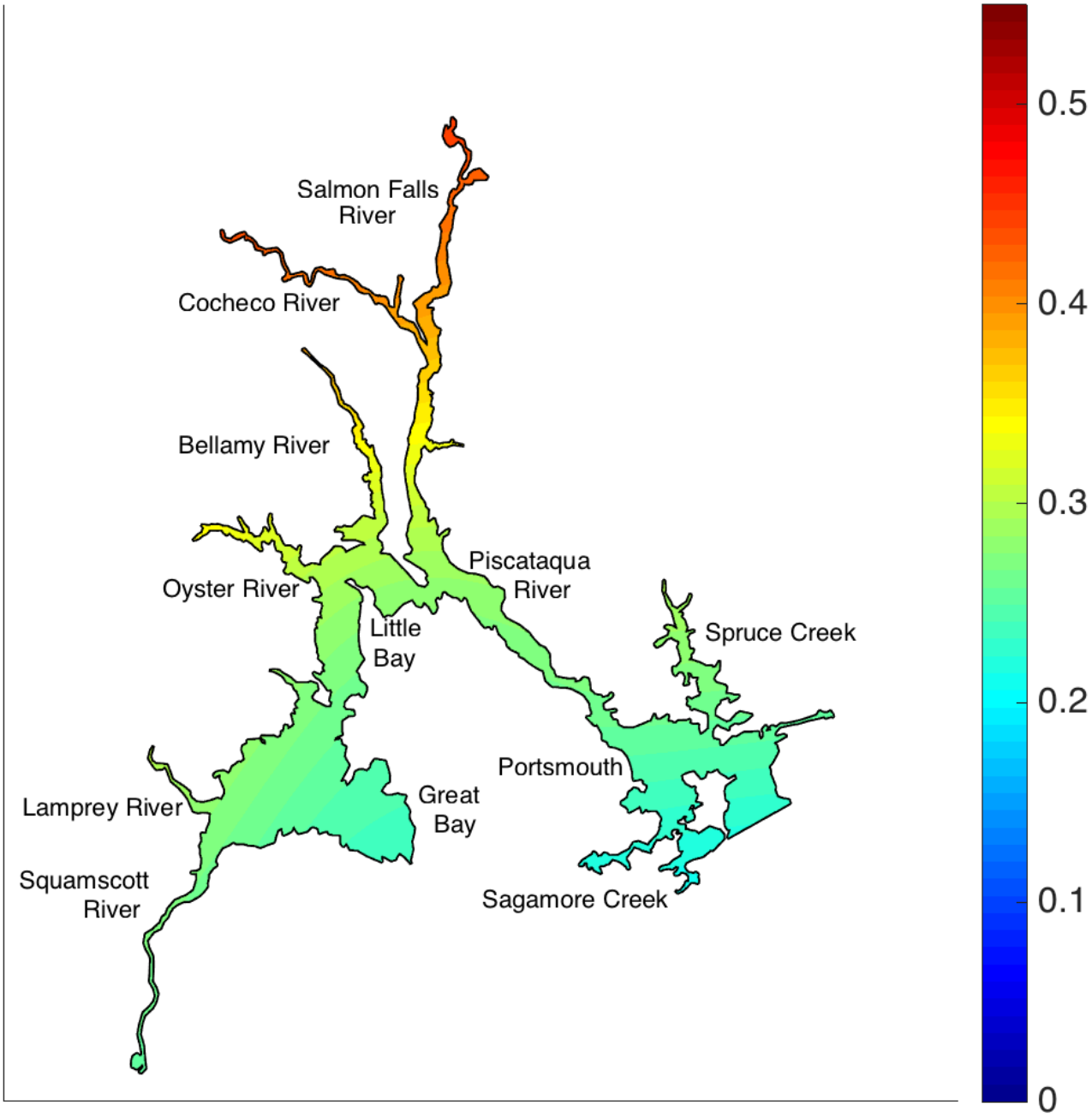}  \\
			(a) KRIG & (b) TPS \\
			\includegraphics[scale=0.25]{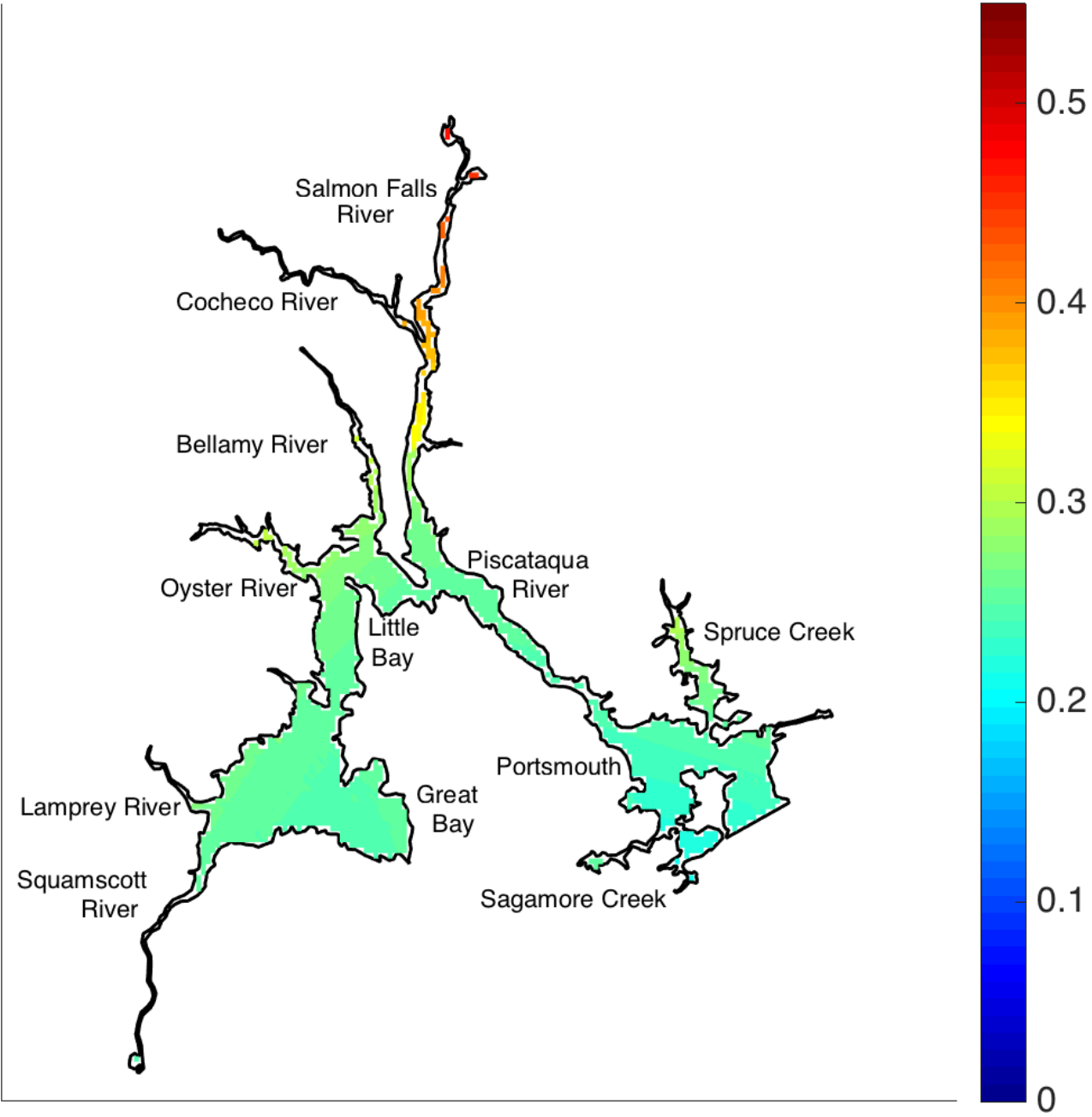} &
			\includegraphics[scale=0.25]{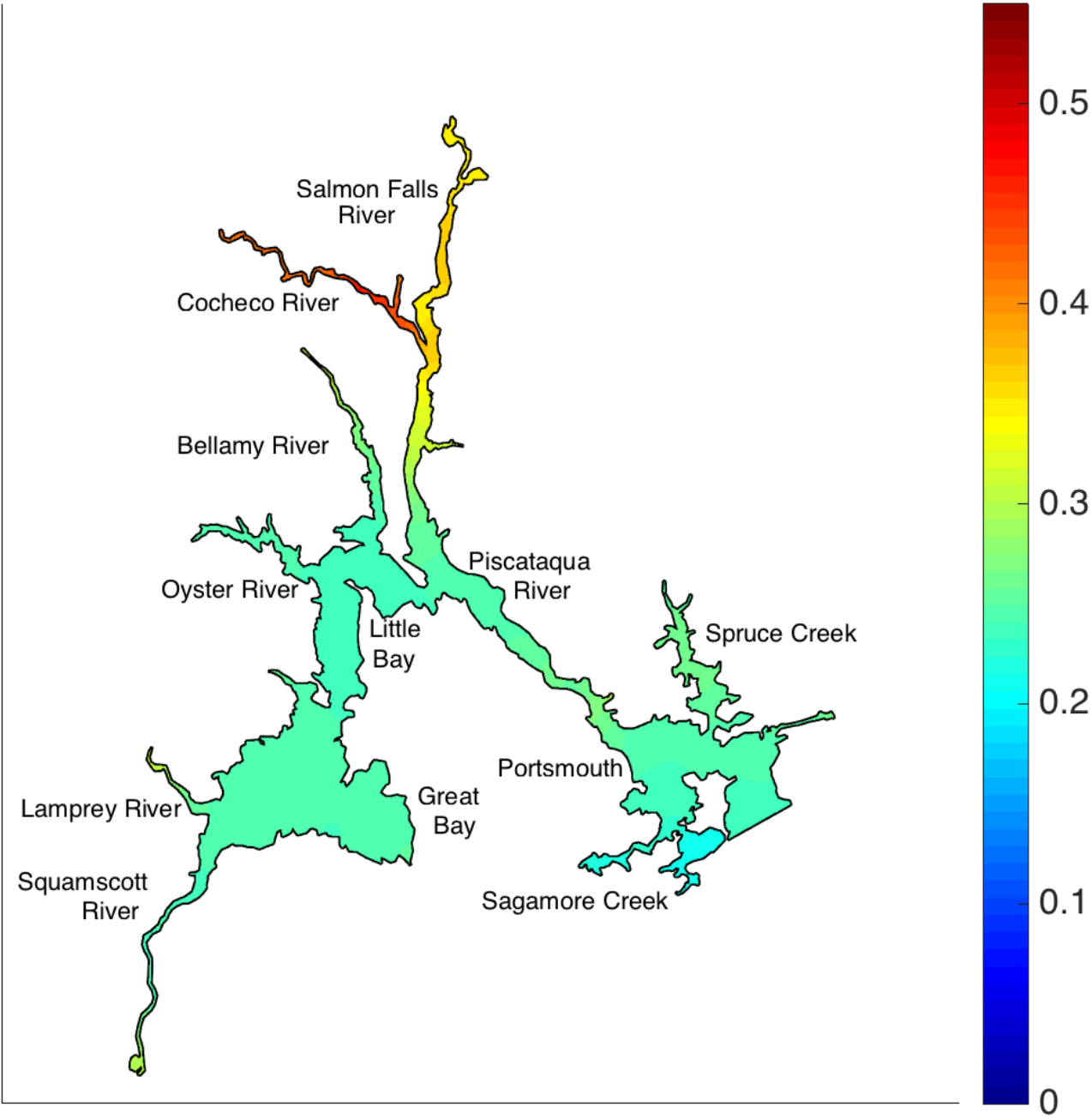} \\
			(c) GLTPS & (d) FEM\\
			\includegraphics[scale=0.25]{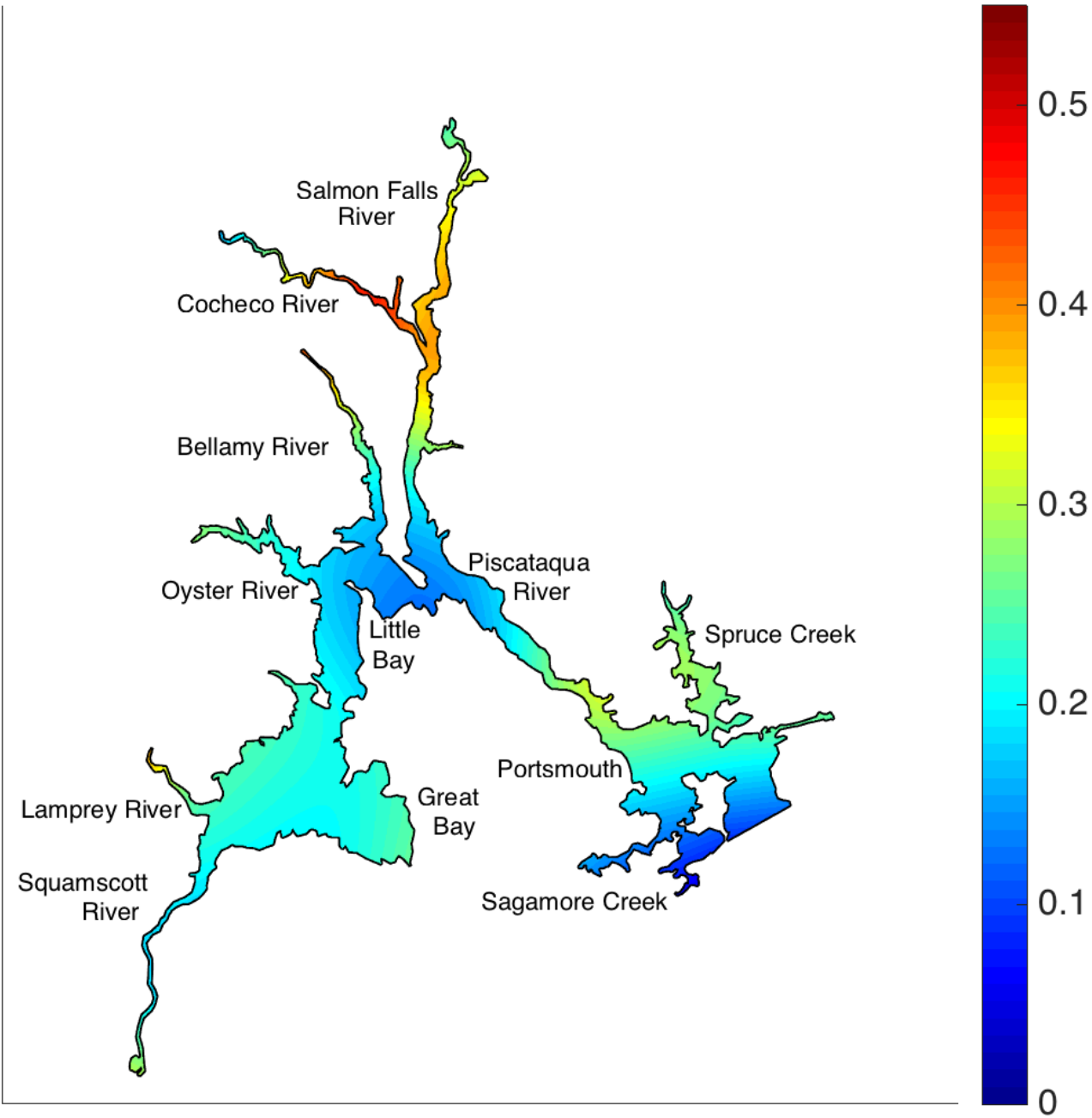}
			&\includegraphics[scale=0.36]{mc_map.pdf}\\
			(e) BPST &(f) Observed Data
		\end{tabular}
	\end{center} \vskip -.2in
	\caption{Prediction maps of mercury concentrations over the estuaries in New Hampshire.}
	\label{FIG:MC_pred}
\end{figure}

%%%%%%%%%%%%%%%%%%%%%%%%%%%%%%%%%%%%%%%%%%%%%%%%%%%%%%%%%%%%%%%%%%%%%%
To evaluate different methods, for each method, we report both the in-sample root mean squared errors (RMSE): $\{n^{-1}\sum_{i=1}^{n}(Y_i-\widehat{Y}_i)^2\}^{1/2}$, and the cross-validation root mean squared prediction errors  (RMSPE) of the mercury concentrations. Since there are only 97 observations in this dataset, we consider leave-one-out cross-validation (LOOCV) prediction error instead of the 10-fold cross-validation as conducted in simulation studies. Specifically, for each $i=1, \ldots, 97$, we train the model on every point except $i$, and then obtain the prediction error on the held out point. Table \ref{TAB:APPmspe} summarizes the RMSE and the LOOCV-RMSPE using different methods. As expected, when the shape of the boundary is complex, smoothers respecting the complicated boundary shape appropriately are able to reduce the prediction errors. The LOOCV-RMSPE is in favor of the model with the BPST smoother, which not only gives the best model fit, but also provides the most accurate prediction of the concentration values among all the methods.

%%%%%%%%%%%%%%%%%%%%%%%%%%%%%%%%%%%%%%%%%%%%%%%%%%%%%%%%%%%%%%%
\begin{table}[htbp]
	\caption{\label{TAB:APPmspe}
		In-sample RMSEs and LOOCV-RMSPE of mercury concentrations.}
	\centering
	\begin{tabular}{cccccc} \hline\hline
		Method &KRIG &TPS &GLTPS &FEM &BPST\\ \hline
		RMSE &0.1397 &0.1381 &0.1366 &0.1263 &0.1197\\
		RMSPE &0.1480 &0.1473 &0.1459 &0.1467 &0.1402\\ \hline\hline
	\end{tabular}
\end{table}

%%%%%%%%%%%%%%%%%%%%%%%%%%%%%%%%%%%%%%%%%%%%%%%%%%%%%%%%%%%%%%
%%%%%%%%%%%%%%%%%%%%%%%%%%%%%%%%%%%%%%%%%%%%%%%%%%%%%%%%%%%%%%
%%%%%%%%%%%%%%%%%%%%%%%%%%%%%%%%%%%%%%%%%%%%%%%%%%%%%%%%%%%%%%
\setcounter{chapter}{6} \setcounter{equation}{0} \vskip .10in
\noindent \textbf{6. Concluding Remarks} \label{sec:conclusion} \vskip 0.05in

In this paper, we have considered PLMs for modeling spatial data with complicated domain boundaries. We introduce a framework of bivariate penalized splines defined on triangulations in the semi-parametric estimation. Our BPST method has demonstrated competitive performance compared to existing methods, while providing a number of possible advantages.

First, the proposed method greatly enhances the application of non/semiparametric methods to spatial data analysis. It solves the problem of ``leakage'' across the complex domains where many conventional smoothing tools suffer from. The numerical results from the simulation studies and application show our method is very effective to account for complex domain boundaries. Our method does not require the data to be evenly distributed or on regular-spaced grids like the tensor product smoothing methods. When we have regions of sparse data, bivariate penalized splines provides a more convenient tool for data fitting than the unpenalized splines since the roughness penalty helps regularize the estimation. Relative to the conventional FEM, our method provides a more flexible way to use piecewise polynomials of various degrees and various smoothness over an arbitrary triangulation for spatial data analysis.

Secondly, we provide new statistical theories for estimating the PLM for data distributed on complex spatial domains. It is shown that our estimates of both parametric part and non-parametric part of the model enjoy excellent asymptotic properties. In particular, we have shown that our estimates of the coefficients in the parametric part are asymptotically normal and derived the convergence rate of the nonparametric component under regularity conditions. We have also provided a standard error formula for the estimated parameters and our simulation studies show that the standard errors are estimated with good accuracy. The theoretical results provide measures of the effect of covariates after adjusting for the location effect. In addition, they give valuable insights into the accuracy of our estimate of the PLM and permit joint inference for the parameters.

Finally, our proposed method is much more computationally efficient compared with other approaches such as kriging and GLTPS. Specifically, for model fitting with $n$ locations, the computational complexity of the ordinary kriging and GLTPS is $O(n^3)$, while the computational complexity of our method is only $O(nN^2)$, where $N$ is the number of triangles in the triangulation and is usually much smaller than $n$ as suggested in Condition (C4).

%%%%%%%%%%%%%%%%%%%%%%%%%%%%%%%%%%%%%%%%%%%%%%%%%%%%%%%%%%%%%%
%%%%%%%%%%%%%%%%%%%%%%%%%%%%%%%%%%%%%%%%%%%%%%%%%%%%%%%%%%%%%%
%%%%%%%%%%%%%%%%%%%%%%%%%%%%%%%%%%%%%%%%%%%%%%%%%%%%%%%%%%%%%%
\setcounter{chapter}{7} \setcounter{equation}{0} \vskip .10in
\noindent \textbf{Acknowledgment}

The first author's research was supported in part by National Science Foundation grants DMS-1106816 and DMS-1542332, the second author's research was supported in part by College of William \& Mary Faculty Summer Research Grant and the third author's research was supported in part by National Science Foundation grant DMS-1521537 and Simons collaboration grant \#280646. The authors would like to thank Haonan Wang and M. Giovanna Ranalli for providing the New Hampshire estuary data. This paper has not been formally reviewed by the EPA. The views expressed here are solely those of the authors. The EPA does not endorse any products or commercial services mentioned in this report. Finally, the authors would like to thank the editor, the associate editor and reviewers for their valuable comments and suggestions to improve the quality of the paper.

%%%%%%%%%%%%%%%%%%%%%%%%%%%%%%%%%%%%%%%%%%%%%%%%%%%%%%%%%%%%%%
%%%%%%%%%%%%%%%%%%%%%%%%%%%%%%%%%%%%%%%%%%%%%%%%%%%%%%%%%%%%%%
%%%%%%%%%%%%%%%%%%%%%%%%%%%%%%%%%%%%%%%%%%%%%%%%%%%%%%%%%%%%%%
\newpage
\vskip 0.10in \noindent \textbf{Appendices}

\setcounter{chapter}{8} \renewcommand{\thetheorem}{A.\arabic{theorem}}
\renewcommand{\theproposition}{A.\arabic{proposition}}
\renewcommand{\thelemma}{A.\arabic{lemma}}
\renewcommand{\thecorollary}{A.\arabic{corollary}}
\renewcommand{\theequation}{A.\arabic{equation}} \renewcommand{\thesubsection}{A.\arabic{subsection}}
\renewcommand{\thetable}{{\arabic{table}}} \setcounter{table}{0}
\renewcommand{\thefigure}{\arabic{figure}} \setcounter{figure}{0}
\setcounter{equation}{0} \setcounter{lemma}{0} \setcounter{proposition}{0}
\setcounter{theorem}{0} \setcounter{subsection}{0}\setcounter{corollary}{0}

\vskip .05in \noindent \textbf{A. Choosing the Triangulation}
\label{SEC:triangulation}

The triangulation selection is one of the key ingredients for obtaining good performance of the bivariate splines estimation. An optimal triangulation is a partition of the domain which is best according to some criterion that measures the shape, size or number of triangles. For example, one of the well-known criteria used to control the shape with a triangulation is the ``max-min" criterion which maximizes the minimum angle of all the angles of the triangles in the triangulation. Based on the ``max-min" criterion, the Delaunay triangulation algorithm can be implemented to avoid sliver triangles (a triangle that is almost flat) when a set of appropriate vertices is chosen. In the past few decades, various packages have been developed to realize the Delaunay algorithm; see MATLAB program \textit{delaunay.m} or MATHEMATICA function \textit{DelaunayTriangulation}. ``Triangle" \citep{Shewchuk:96} is also widely used in many applications, and one can download it for free from \url{http://www.cs.cmu.edu/~quake/triangle.html}. It is a C$++$ program for two-dimensional mesh generation and construction of Delaunay triangulations. ``DistMesh'' is another method to generate unstructured triangular and tetrahedral meshes; see the \textit{DistMesh} generator on \url{http://persson.berkeley.edu/distmesh/}. A detailed description of the program is provided by \cite{Persson:Strang:04}. Once the shape of triangulations is handled, we can simply focus on how to select the number of triangles, $K$, for quasi-uniform triangulations in all the numerical studies.

As is usual with the one-dimensional (1-D) penalized least squares (PLS) splines, the number of knots is not important given that it is above some minimum depending upon the degree of the smoothness; see \cite{Li:Ruppert:08}. For bivariate PLS splines, \cite{Lai:Wang:13} and \cite{Wang:Mu:Wang:17} also observed that the number of triangles $K$ is not very critical, provided $K$ is larger than some threshold. In fact, one of the main advantages of using PLS splines over unpenalized splines is the flexibility of choosing knots in the 1-D setting and choosing triangles in the 2-D setting. For unpenalized splines, one has to have large enough sample according to the requirement of the degree of splines on each subinterval in the 1-D case or each triangle in the 2-D case to guarantee that a solution can be found. However, there is no such requirement for PLS splines. When the smoothness $r \geq 1$, the only requirement for bivariate PLS splines is that there is at least one triangle containing three points which are not in one line \citep{Lai:08}. Also, PLS splines perform similarly to unpenalized splines as long as the penalty parameter $\lambda$ is very small. So in summary, the proposed bivariate PLS splines are very flexible and convenient for data fitting, even for smoothing sparse and unevenly sampled data over a domain with complicated boundary.

In practice, to form a good triangulation, we need to make certain that the triangulation is sufficiently fine to capture the feature in the dataset and not so large that computational burden is unnecessarily heavy. \cite{Wang:Mu:Wang:17} proposed to choose the number of triangles by generalized cross-validation (GCV) (\cite{Craven:Wahba:79}; \cite{Wahba:90}). As suggested by \cite{Wang:Mu:Wang:17}, we consider a sequence of trial values of the number of vertices of the triangles ``equally-spaced'' on the domain, and apply the Delaunay triangulation method. The more vertices we insert, the finer the triangulation. For each trial value, the PLS spline is fitted, and the value in that trial sequence that minimizes the GCV is selected.  \cite{Wang:Mu:Wang:17} provides extensive numerical studies to illustrate the practical performance of the GCV triangulation selection scheme.

%%%%%%%%%%%%%%%%%%%%%%%%%%%%%%%%%%%%%%%%%%%%%%%%%%%%%%%%%%%%%%
%%%%%%%%%%%%%%%%%%%%%%%%%%%%%%%%%%%%%%%%%%%%%%%%%%%%%%%%%%%%%%
%%%%%%%%%%%%%%%%%%%%%%%%%%%%%%%%%%%%%%%%%%%%%%%%%%%%%%%%%%%%%%
\setcounter{chapter}{9} \renewcommand{\thetheorem}{B.\arabic{theorem}}
\renewcommand{\theproposition}{B.\arabic{proposition}}
\renewcommand{\thelemma}{B.\arabic{lemma}}
\renewcommand{\thecorollary}{B.\arabic{corollary}}
\renewcommand{\theequation}{B.\arabic{equation}} \renewcommand{\thesubsection}{B.\arabic{subsection}}
\renewcommand{\thetable}{{B.\arabic{table}}} \setcounter{table}{0}
\renewcommand{\thefigure}{B.\arabic{figure}} \setcounter{figure}{0}
\setcounter{equation}{0} \setcounter{lemma}{0} \setcounter{proposition}{0}
\setcounter{theorem}{0} \setcounter{subsection}{0} \setcounter{corollary}{0}
\vskip .05in \noindent \textbf{B. More Simulation Results}

%%%%%%%%%%%%%%%%%%%%%%%%%%%%%%%%%%%%%%%%%%%%%%%%%%%%%%%%%%%%%%%%%%%%%%
\vskip .05in \noindent \textbf{B.1. Additional simulation result from Example 1}

Figure \ref{FIG:eg1_3} shows the estimated functions over a grid of $500\times 200$ points via different methods for replicate 1 with $\rho=0.7$. From those plots, it is clear that the BPST and GLTPS estimates perform better than the other four estimates. There seems to be some ``leakage effect'' in KRIG and TPS estimates, which is likely caused by the fact that KRIG and TPS do not take the complex boundary into any account and smooth across the gap inappropriately. Finally, as what we expected that the BPST estimators based on the three different triangulations are very similar, which confirms that the number of triangles is not very critical for the penalized spline fitting as long as it is sufficiently large enough to capture the pattern and features of the data.

%%%%%%%%%%%%%%%%%%%%%%%%%%%%%%%%%%%%%%%%%%%%%%%%%%%%%%%%%%%%%%%%%%%%%%
\begin{figure}[htbp]
	\begin{center}
		\begin{tabular}{cccc}
			\includegraphics[height=2.5cm]{eg1_mtrue_contour.pdf} & 
			\includegraphics[height=2.5cm]{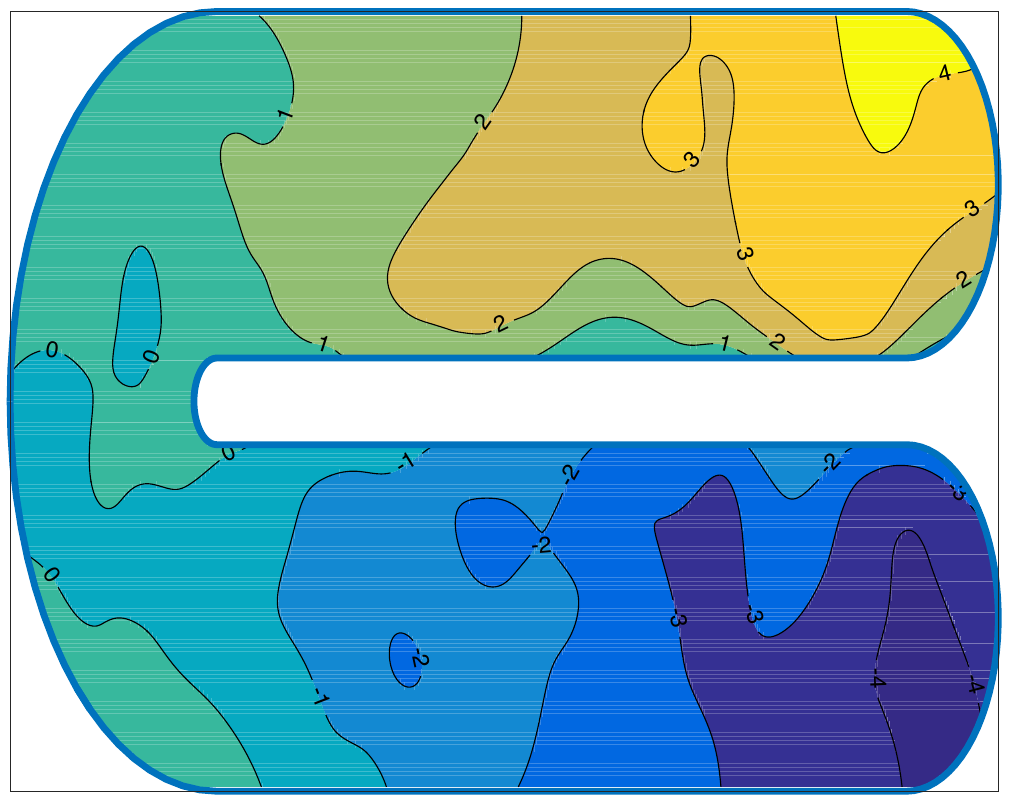} & 
			\includegraphics[height=2.5cm]{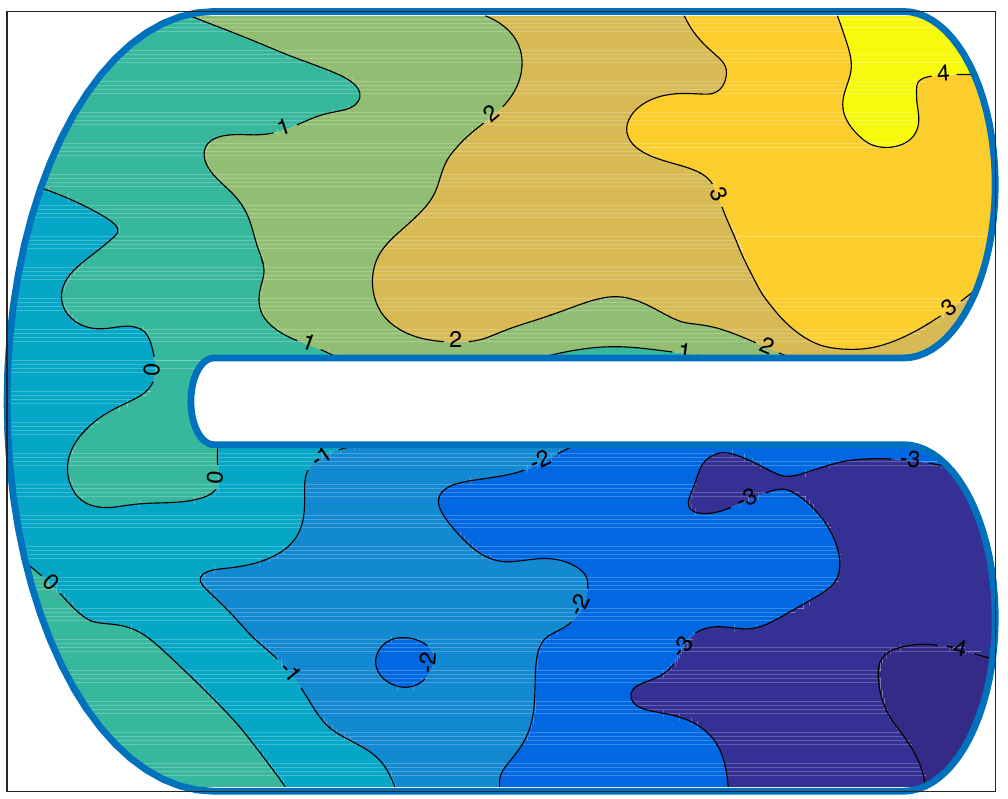}  &
			\includegraphics[height=2.5cm]{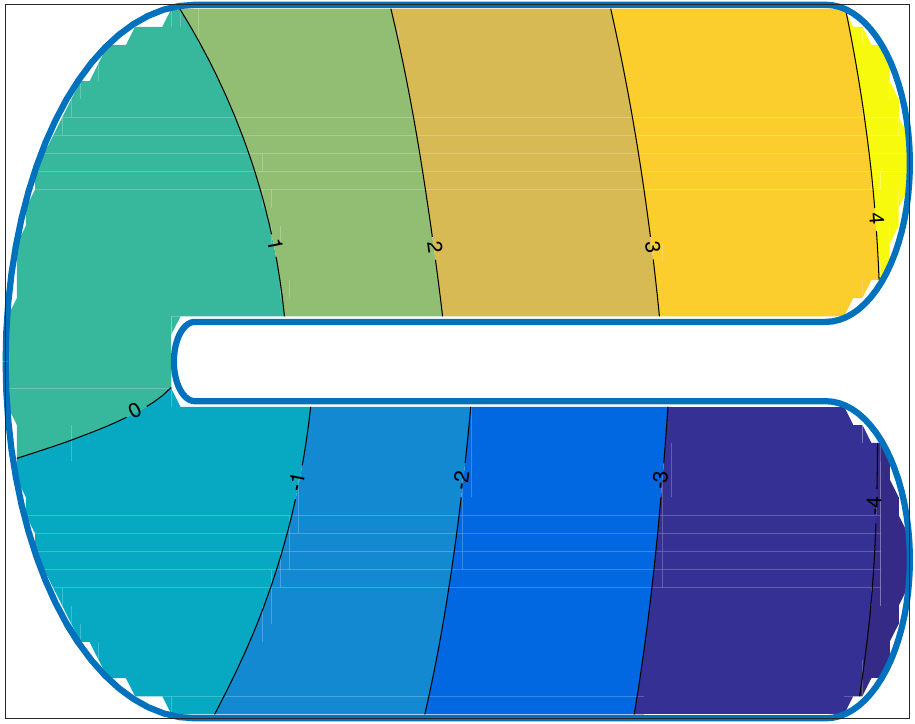}\\
			(a) True Contour & (b) KRIG & (c) TPS & (d) GLTPS\\[6pt]
			\includegraphics[height=2.5cm]{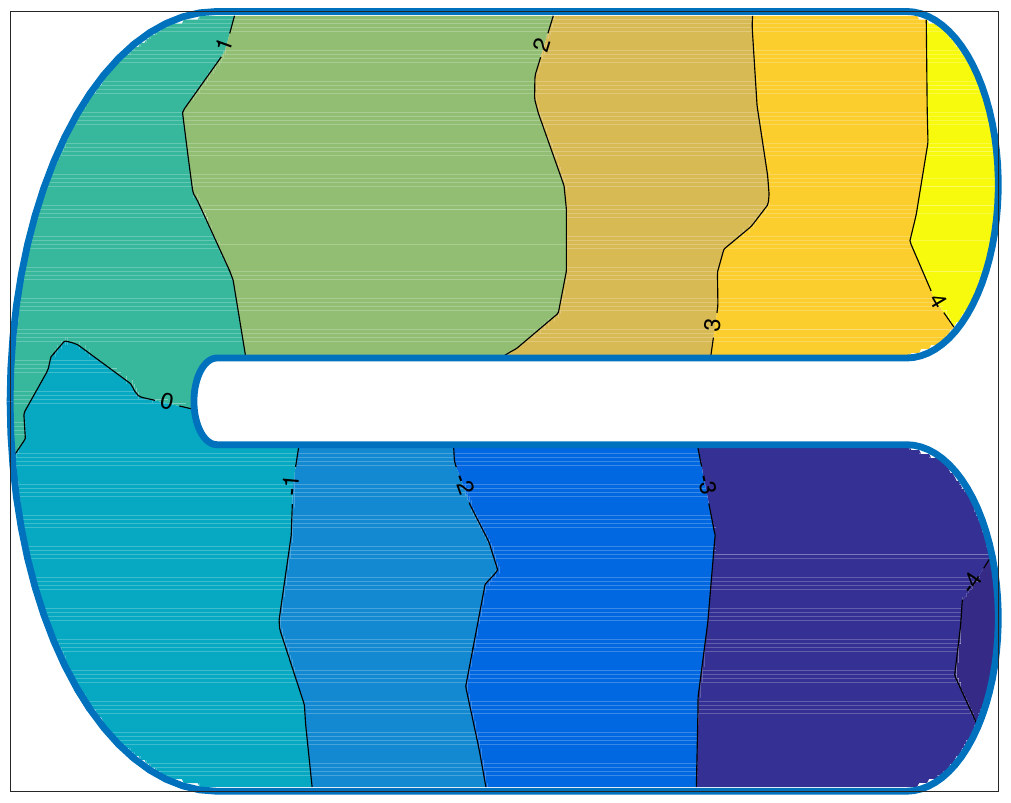} & \includegraphics[height=2.5cm]{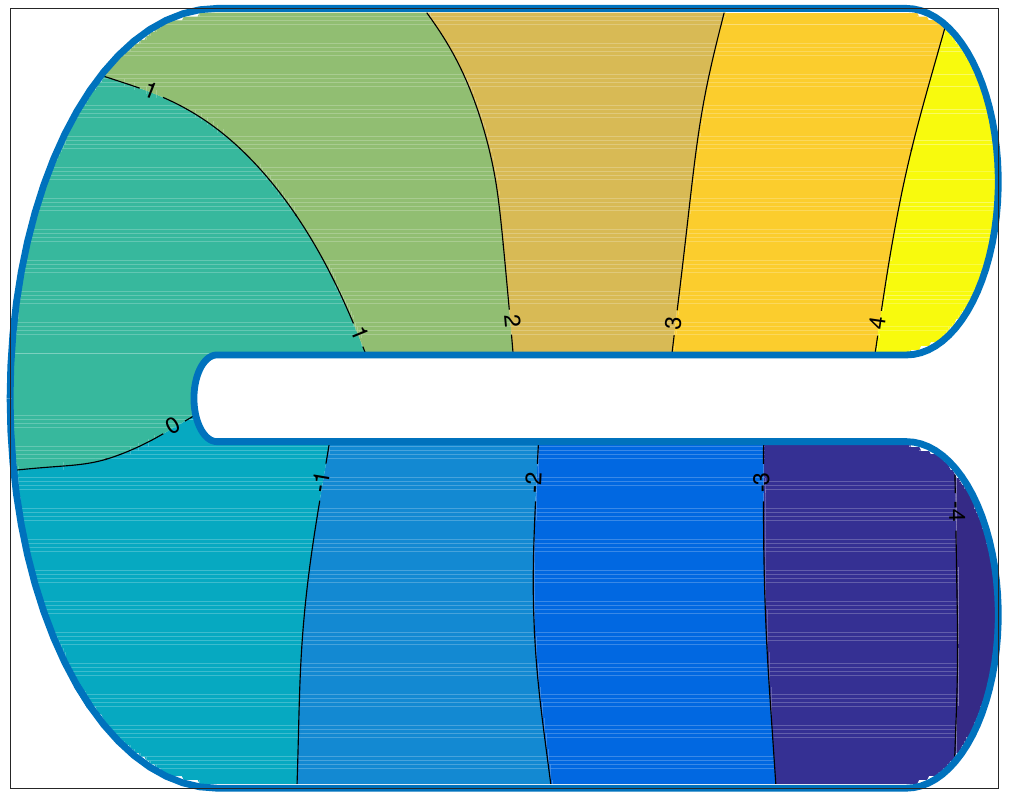} & \includegraphics[height=2.5cm]{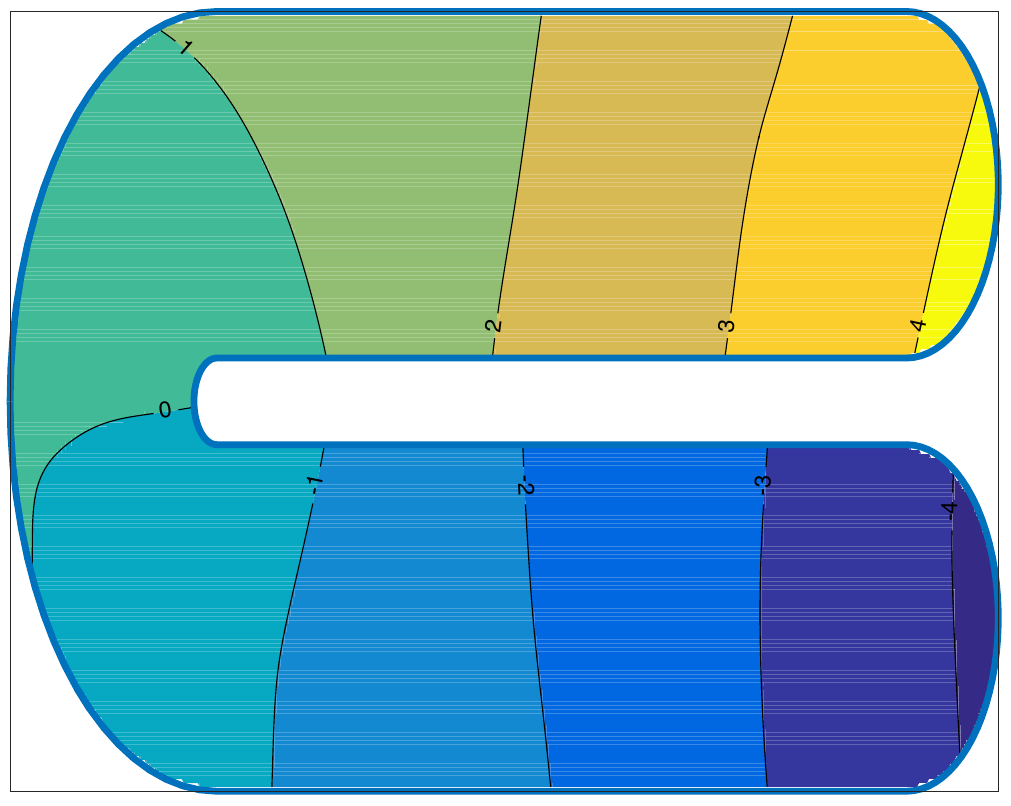}
			& \includegraphics[height=2.5cm]{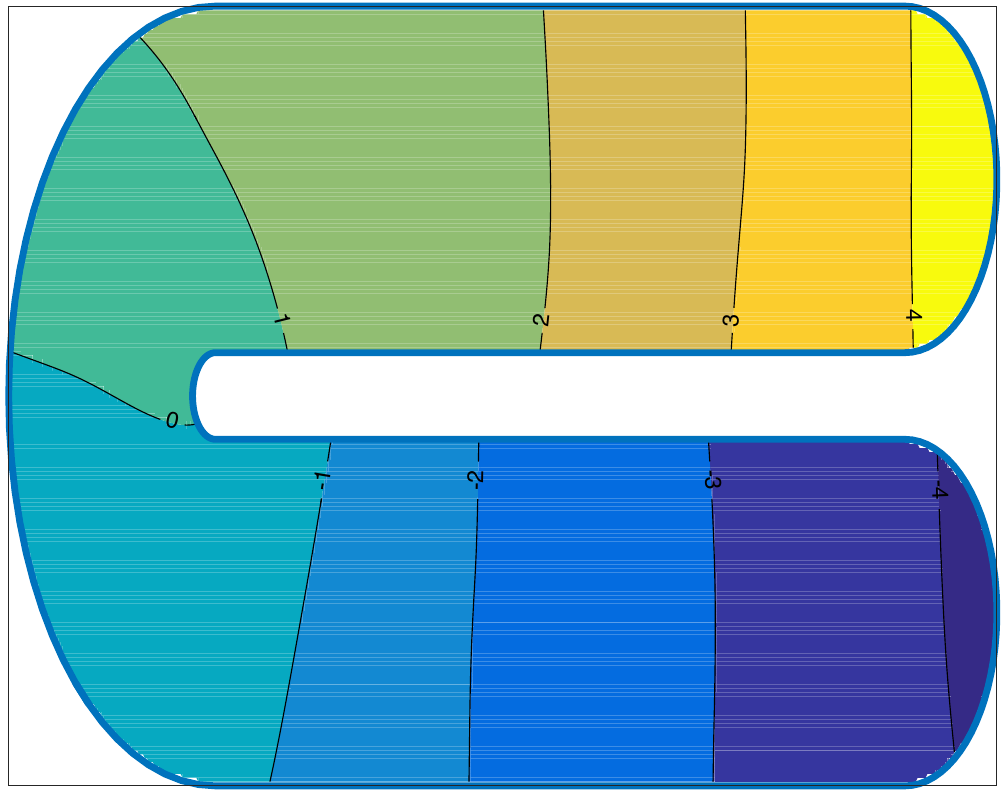} \\
			(e) FEM & (f) BPST ($\triangle_1$) & (g) BPST ($\triangle_2$) & (h) BPST ($\triangle_3$) \\
		\end{tabular}
	\end{center}
	\caption{Contour maps for the true function and its estimators ($\rho=0.7$).}
	\label{FIG:eg1_3}
\end{figure}

%%%%%%%%%%%%%%%%%%%%%%%%%%%%%%%%%%%%%%%%%%%%%%%%%%%%%%%%%%%%%%%%%%%%%%
\vskip 0.10in \noindent \textbf{B.2. A new simulation example}

In this example, we consider a rectangular domain, $[0,1]^2$, where there is no irregular shape or complex boundaries problem. In this case, classical methods for spatial data analysis, such as KRIG and TPS, will not encounter any difficulty. We obtain the true signal and noisy observation for each coordinate pair lying on a $101\times 101$-grid over $[0,1]^2$ using the following model:
\[
Y=\mathbf{Z}^{\T}\bs{\beta}+g(X_{1},X_{2})+\epsilon,
\]
where $\bs{\beta}=(-1,1)^{\T}$ and $g(x_1,x_2)=10\{(x_1-0.5)^2+(x_2-0.5)^2\}$. The random error, $\epsilon$, is generated from an $N(0,\sigma_{\epsilon}^{2})$ distribution with $\sigma_{\epsilon}=0.5$. Similar to Example 1, we simulate $Z_{1}\sim uniform[-1,1]$, and $Z_{2}=\cos[4\pi (\rho(X_{1}^{2}+X_{2}^{2})+(1-\rho)U)]$, where $\rho=0.0~\mathrm{or}~0.7,~U\sim uniform[-1,1]$ and is independent from $(X_{1},X_{2})$ and $Z_{1}$. Next we take $100$ Monte Carlo random samples of size $n=200$ from the $101\times101$ points.

Figure \ref{FIG:eg2_1} (a) and (b) display the true quadratic surface and the contour map, respectively. We use the triangulation in Figure \ref{FIG:eg2_1} (e) and (f), and there are 8 triangles and 9 vertices as well as 18 triangles and 16 vertices, respectively. In addition, the points in Figure \ref{FIG:eg2_1} (d) demonstrate the sampled location points of replicate 100.

%%%%%%%%%%%%%%%%%%%%%%%%%%%%%%%%%%%%%%%%%%%%%%%%%%%%%%%%%%%%%%%%%%%%%%
\begin{figure}[htbp]
	\begin{center}
		\begin{tabular}{cc}
			\includegraphics[height=1.8in]{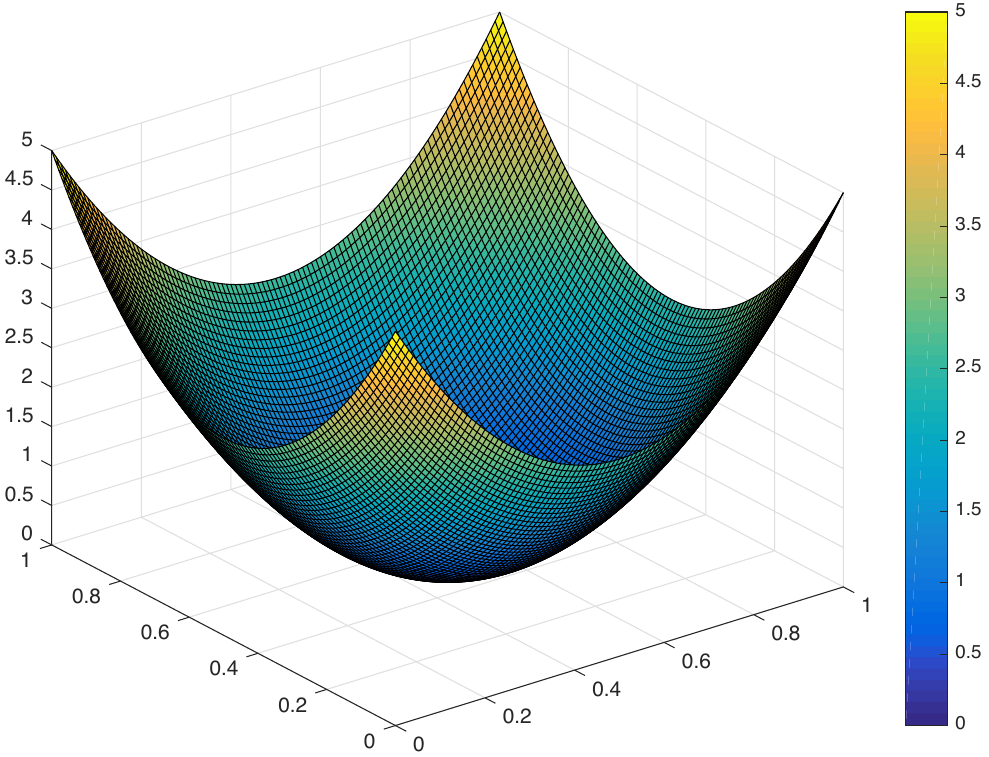} &
			\includegraphics[height=1.8in]{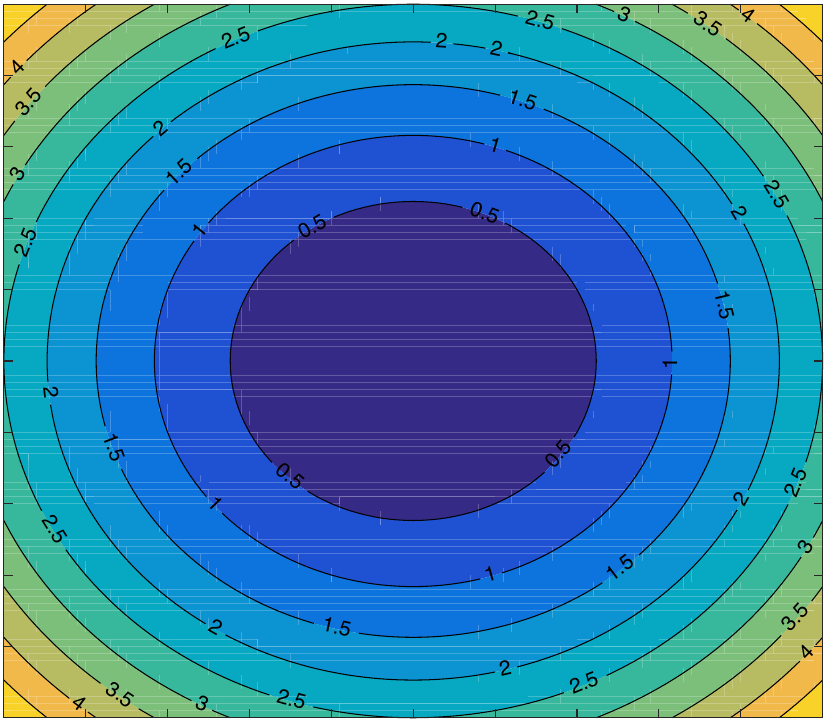} \\[10pt]
			(a) & (b)\\
			\includegraphics[height=1.8in]{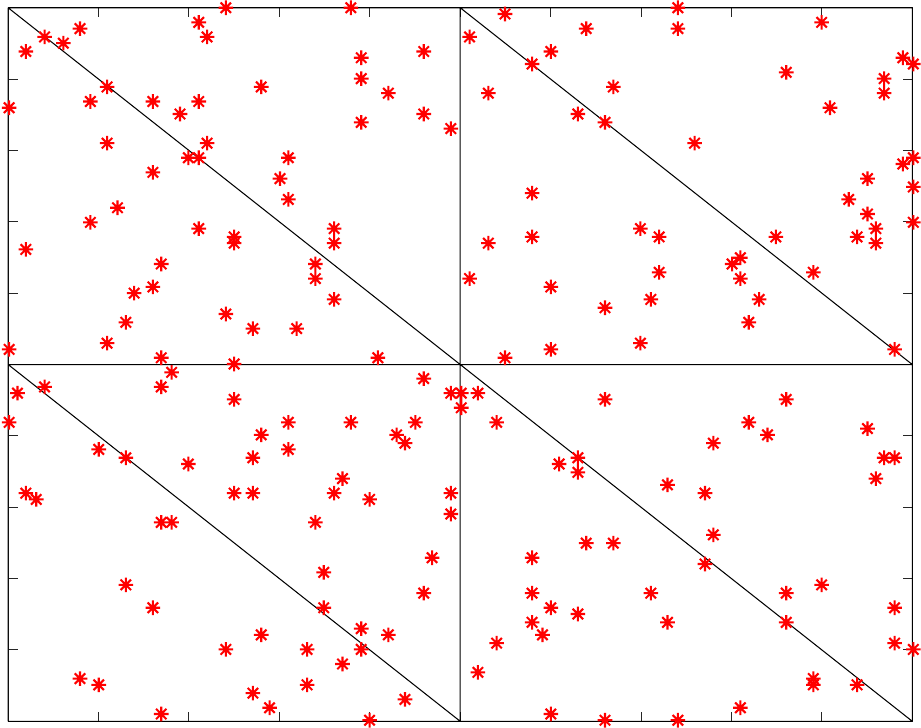} &
			\includegraphics[height=1.8in]{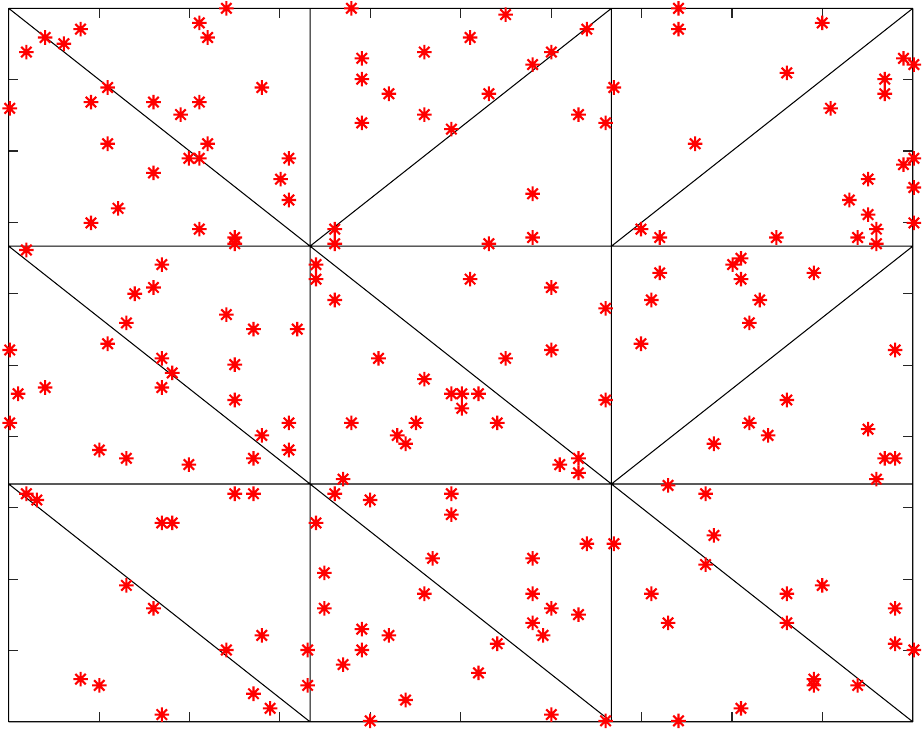}\\
			(c) & (d)\\
		\end{tabular}
	\end{center}
	\caption{(a) true function of $g(\cdot)$; (b) contour map of $g(\cdot)$; (c) first triangulation  ($\triangle_1$); and (d) second triangulation ($\triangle_2$) on the domain.}
	\label{FIG:eg2_1}
\end{figure}
%%%%%%%%%%%%%%%%%%%%%%%%%%%%%%%%%%%%%%%%%%%%%%%%%%%%%%%%%%%%%%%%%%%%%%

We compare the proposed BPST estimator with estimators from the KRIG, TPS, LFE methods, which are implemented in the same way as in Section 4. To see the accuracy of the estimators, we compute the RMSEs of the coefficient estimators and the estimator of $\sigma_{\epsilon}$. To see the overall prediction accuracy, we make prediction on the $101 \times 101$ grid points on the domain for each replication using different methods, and compare the predicted values with the true function of $g(\cdot)$ at these grid points, and we report the average mean squared prediction errors (MSPE) over all replications.

All the results are summarized in Table \ref{TAB:eg2_1}. As expected, KRIG and TPS work pretty well since the domain is regular in this example. In both scenarios, BPST performs the best. One also notices that, compared with the FEM, our BPST estimator shows much better performance in terms of both estimation and prediction, because BPST provides a more flexible and easier construction of splines with piecewise polynomials of various degrees and smoothness than the FEM method.  As pointed out in \cite{Wood:Bravington:Hedley:08}, the FEM method may require a very fine triangulation in order to reach certain approximation power, however, BPST doesn't need such a strict fineness requirement as it uses piecewise polynomials of higher degree  yielding an larger order approximation power.

Figures \ref{FIG:eg2_2} and \ref{FIG:eg2_3} show the estimated functions via different methods for the last replicate. Compare with the true function in Figure \ref{FIG:eg2_1}, the BPST estimate looks visually better than the other estimates. In addition, from Figures \ref{FIG:eg2_2} and \ref{FIG:eg2_3}, one also sees that the BPST estimators based on $\triangle_1$ and $\triangle_2$ are very similar, which agrees our findings for penalized splines. In summary, Monte Carlo experiment in this study also shows that once the minimum necessary number of triangles has been reached for BPST, further increasing of the number of triangles usually have little effect on the fitting process.

%%%%%%%%%%%%%%%%%%%%%%%%%%%%%%%%%%%%%%%%%%%%%%%%%%%%%%%%%%%%%%%
\begin{table}[htbp]
	\caption{\label{TAB:eg2_1}Root mean squared errors of the estimates.}
	\centering
	\begin{tabular}{clcccccc}\hline\hline &Method &$\beta_{1}$ &$\beta_{2}$ &$\sigma_{\varepsilon}$ &$g(\cdot)$\\ \hline
		\multirow{5}{*}{0.0} &KRIG &0.0640 &0.0557 &0.0369 &0.1797\\
		&TPS &0.0647 &0.0551 &0.0286 &0.1640\\
		&LFE &0.0772 &0.0604 &0.0669 &0.2978\\
		&BPST($\triangle_1$) &0.0642 &0.0546 &0.0266 &0.1495\\
		&BPST($\triangle_2$) &0.0640 &0.0556 &0.0273 &0.1395\\ \hline
		
		\multirow{5}{*}{0.7} &KRIG &0.0647 &0.0530 &0.0365 &0.1800\\
		&TPS &0.0653 &0.0515 &0.0281 &0.1640\\
		&LFE &0.0769 &0.0607 &0.0668 &0.2978\\
		&BPST($\triangle_1$) &0.0645 &0.0513 &0.0263 &0.1497\\
		&BPST($\triangle_2$) &0.0644 &0.0512 &0.0265 &0.1476\\ \hline\hline
	\end{tabular}
\end{table}

%%%%%%%%%%%%%%%%%%%%%%%%%%%%%%%%%%%%%%%%%%%%%%%%%%%%%%%%%%%%%%%%%%%%%%
\begin{figure}[htbp]
	\begin{center}
		\begin{tabular}{ccc}
			\includegraphics[height=4cm]{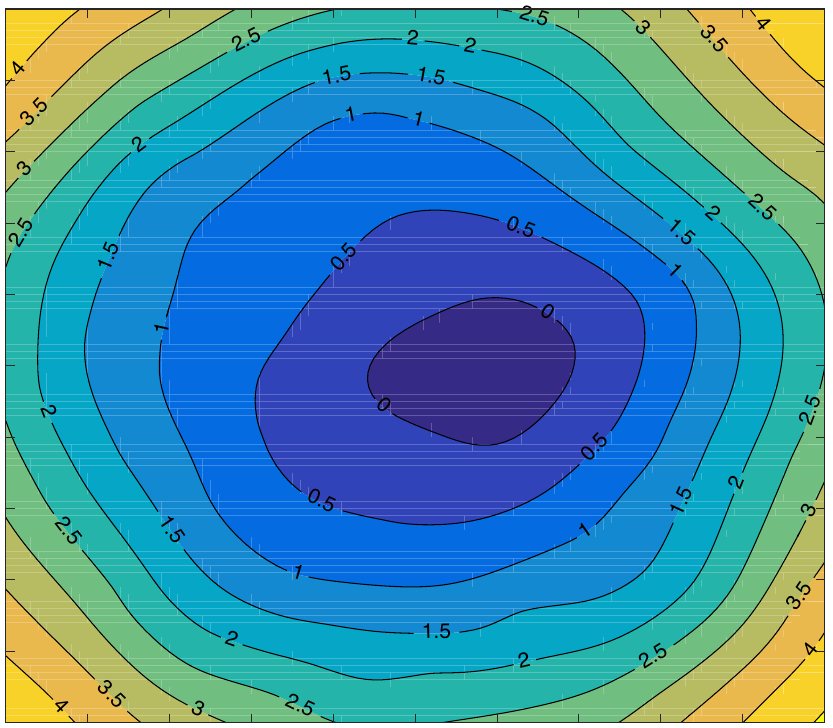}  & 
			\includegraphics[height=4cm]{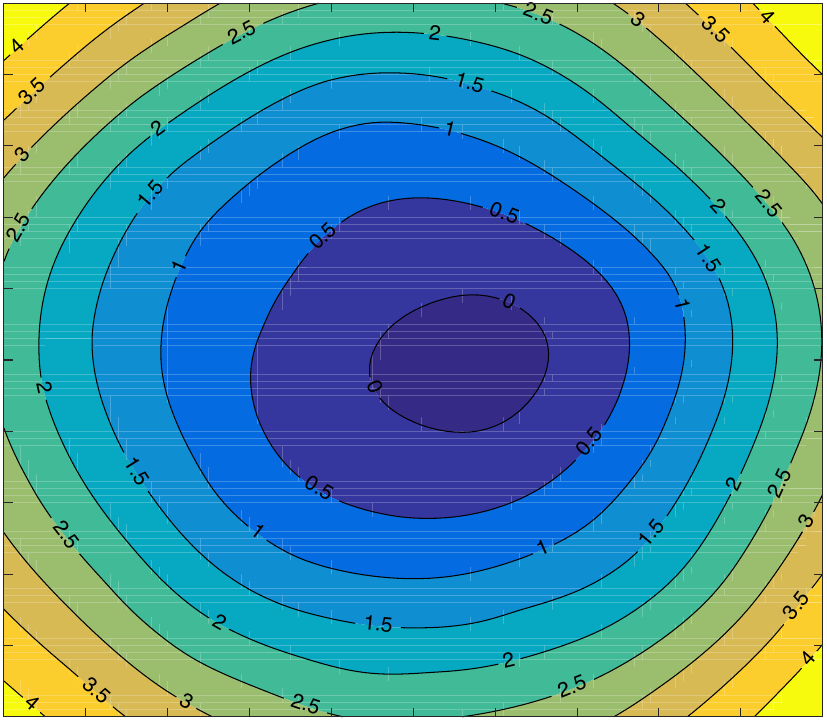}
			& \includegraphics[height=4cm]{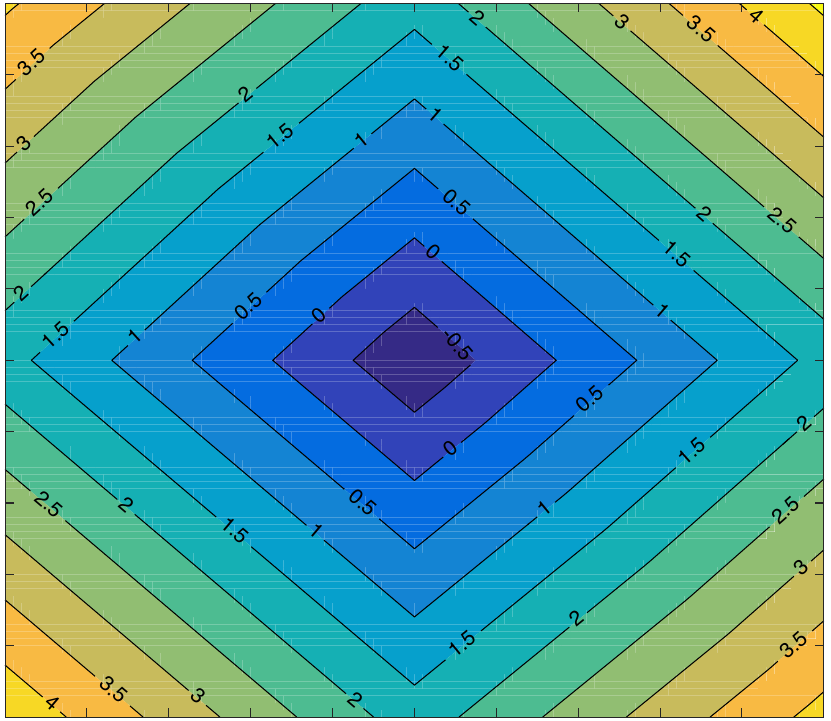} \\
			(a) KRIG & (b) TPS & (c) FEM\\
			\includegraphics[height=4cm]{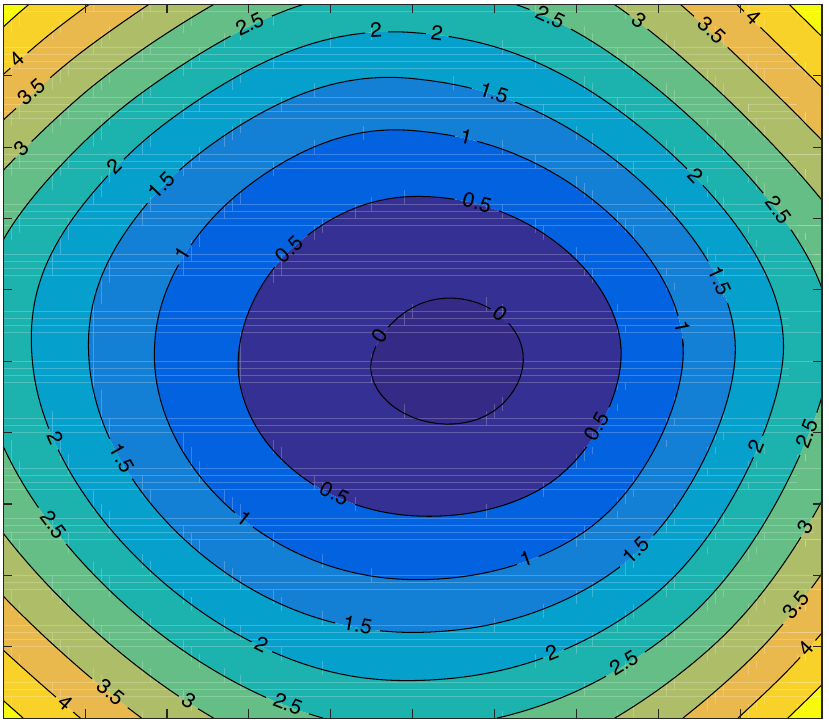} & \includegraphics[height=4cm]{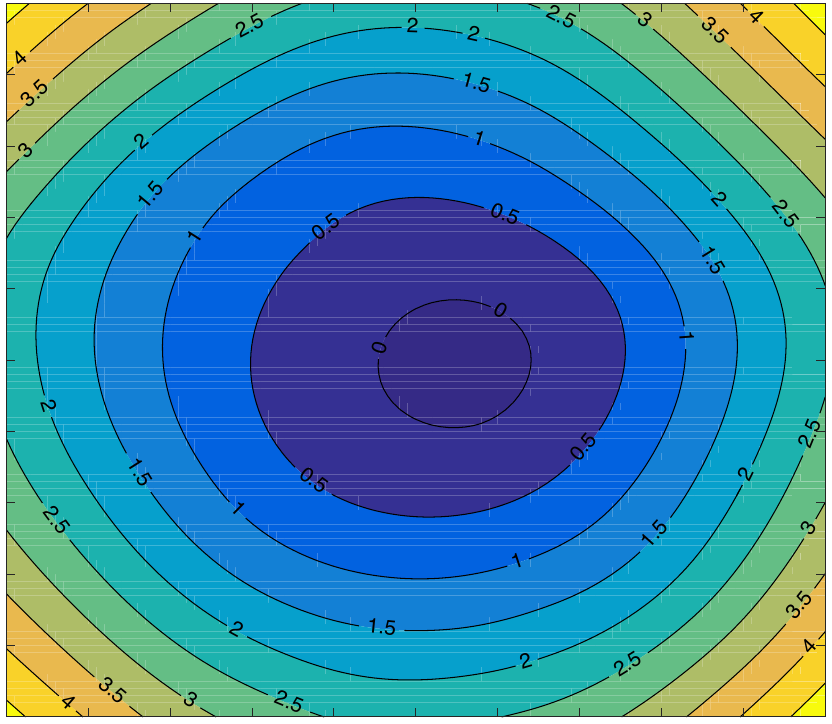}\\
			(d) BPST ($\triangle_1$)  & (e) BPST ($\triangle_2$) \\
		\end{tabular}
	\end{center}
	\caption{Contour maps for the estimators ($\rho=0.0$).}
	\label{FIG:eg2_2}
\end{figure}

%%%%%%%%%%%%%%%%%%%%%%%%%%%%%%%%%%%%%%%%%%%%%%%%%%%%%%%%%%%%%%%%%%%%%%
\begin{figure}[htbp]
	\begin{center}
		\begin{tabular}{ccc}
			\includegraphics[height=4cm]{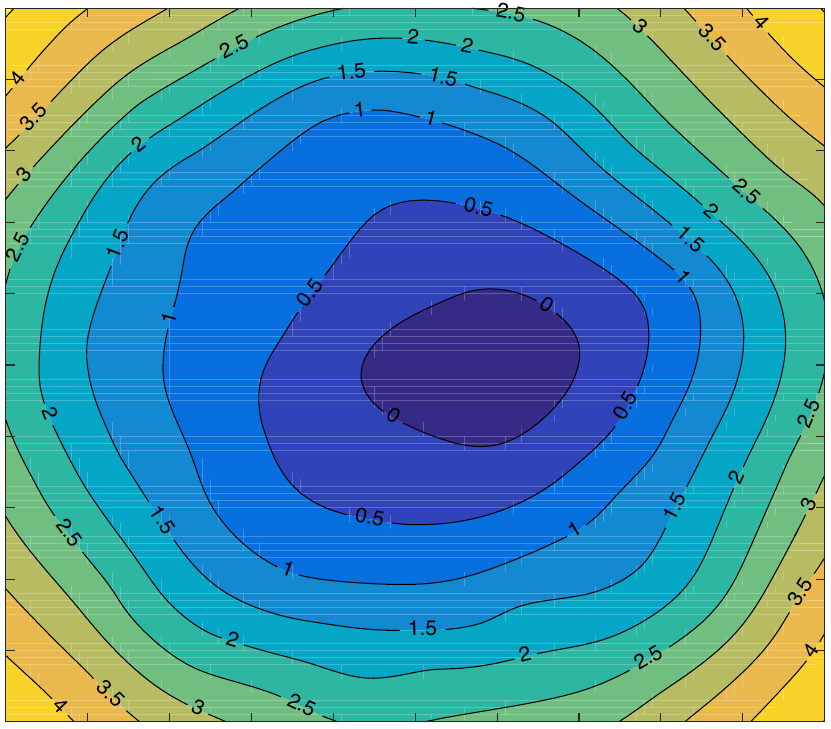} & 
			\includegraphics[height=4cm]{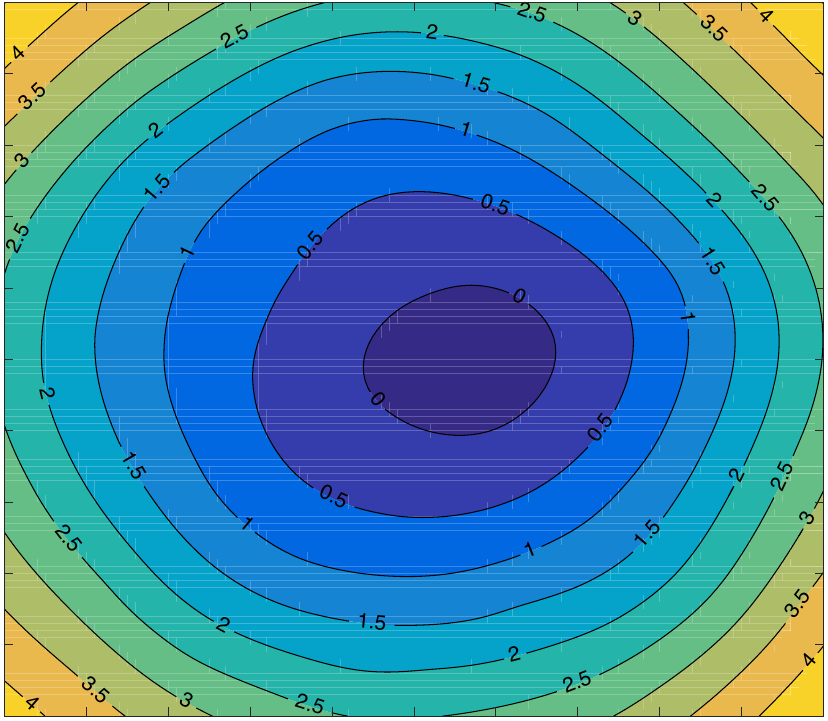}
			& \includegraphics[height=4cm]{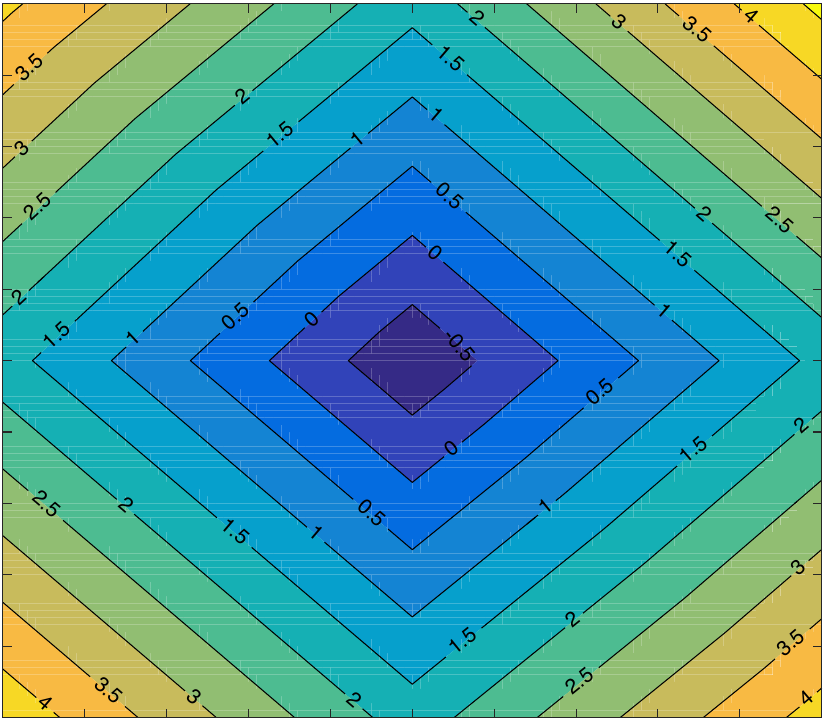} \\
			(a) KRIG & (b) TPS & (c) FEM\\
			\includegraphics[height=4cm]{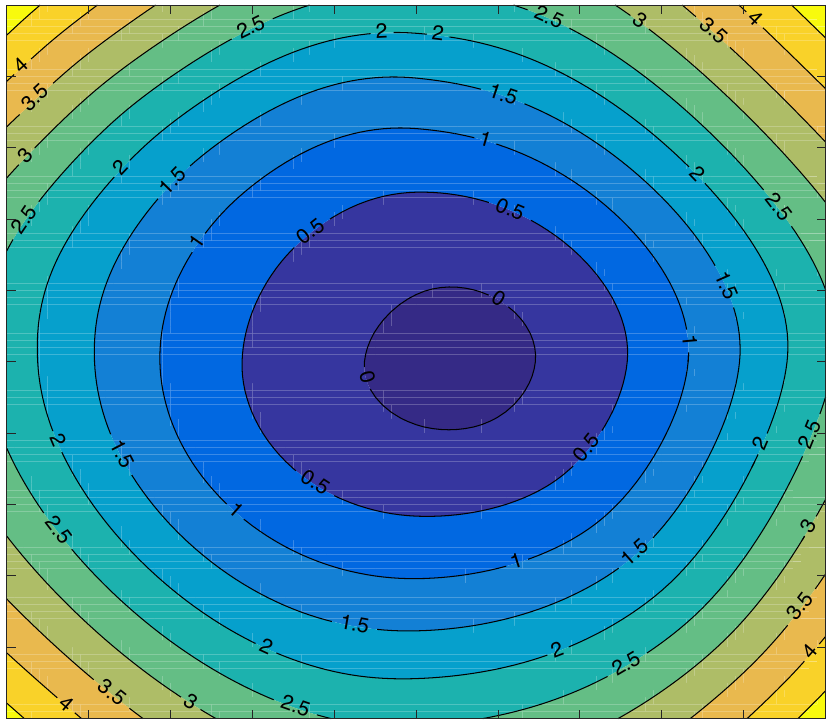} & \includegraphics[height=4cm]{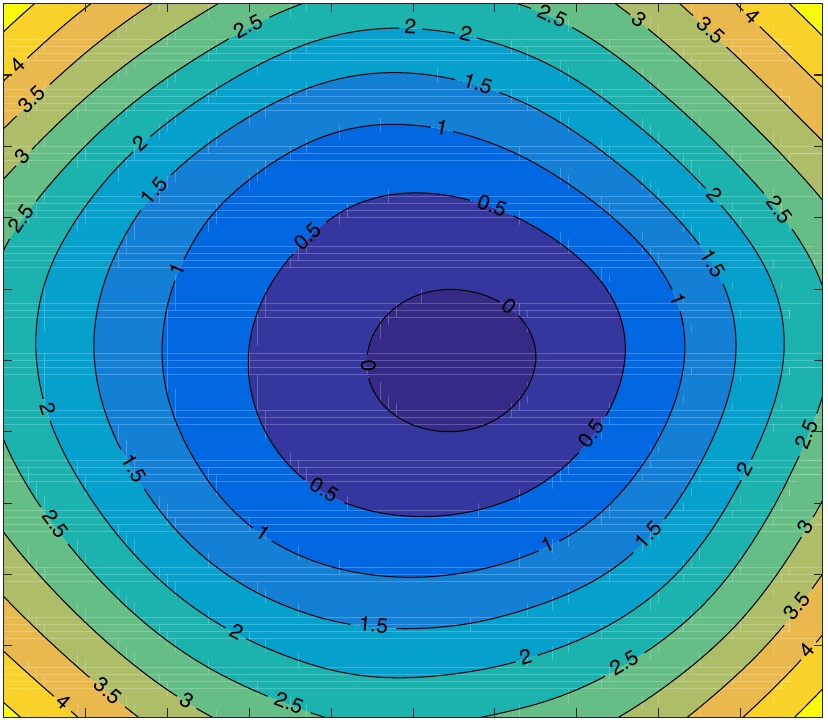}\\
			(d) BPST ($\triangle_1$)  & (e) BPST ($\triangle_2$) \\
		\end{tabular}
	\end{center}
	\caption{Contour maps for the estimators  ($\rho=0.7$).}
	\label{FIG:eg2_3}
\end{figure}
%%%%%%%%%%%%%%%%%%%%%%%%%%%%%%%%%%%%%%%%%%%%%%%%%%%%%%%%%%%%%%%%%%%%%%

%%%%%%%%%%%%%%%%%%%%%%%%%%%%%%%%%%%%%%%%%%%%%%%%%%%%%%%%%%%%%%%%%%%%%%

Table \ref{TAB:eg2_2} lists the accuracy results of the standard error formula in (\ref{DEF:Sigma_n}) for $\hat{\beta}_{1}$ and $\hat{\beta}_{2}$ using BPST with triangulation $\triangle_1$. From Table \ref{TAB:eg2_2}, one sees that the estimated standard errors based on sample size $n=200$ are very accurate.

%%%%%%%%%%%%%%%%%%%%%%%%%%%%%%%%%%%%%%%%%%%%%%%%%%%%%%%%%%%%%%%
\begin{table}[htbp]
	\caption{\label{TAB:eg2_2}
		Standard error estimates of the BPST coefficients.}
	\centering
	\begin{tabular}{cccccc}\hline\hline
		~~$\rho$ ~~ &~~~~Parameter ~~~~ &~~~~$\mathrm{SE}_{\mathrm{mc}}$ ~~~~ &~~~~$\mathrm{SE}_{\mathrm{mean}}$ ~~~~ &~~~~$\mathrm{SE}_{\mathrm{median}}$ ~~~~ &~~~~$\mathrm{SE}_{\mathrm{mad}}~~$\\ \hline
		
		\multirow{2}{*}{0.0} &$\beta_{1}$ &0.0643 &0.0622 &0.0621 &0.0032\\
		&$\beta_{2}$ &0.0546 &0.0517 &0.0516 &0.0028\\ \hline
		
		% Triangulation 1;
		\multirow{2}{*}{0.7} &$\beta_{1}$ &0.0645 &0.0621 &0.0622 &0.0030\\
		&$\beta_{2}$ &0.0515 &0.0519 &0.0518 &0.0026\\ \hline\hline
	\end{tabular}
\end{table}

%%%%%%%%%%%%%%%%%%%%%%%%%%%%%%%%%%%%%%%%%%%%%%%%%%%%%%%%%%%%%%
%%%%%%%%%%%%%%%%%%%%%%%%%%%%%%%%%%%%%%%%%%%%%%%%%%%%%%%%%%%%%%
%%%%%%%%%%%%%%%%%%%%%%%%%%%%%%%%%%%%%%%%%%%%%%%%%%%%%%%%%%%%%%
\setcounter{chapter}{10} \renewcommand{\thetheorem}{C.\arabic{theorem}}
\renewcommand{\theproposition}{C.\arabic{proposition}}
\renewcommand{\thelemma}{C.\arabic{lemma}}
\renewcommand{\thecorollary}{C.\arabic{corollary}}
\renewcommand{\theequation}{C.\arabic{equation}} \renewcommand{\thesubsection}{C.\arabic{subsection}}
\renewcommand{\thetable}{{C.\arabic{table}}} \setcounter{table}{0}
\renewcommand{\thefigure}{C.\arabic{figure}} \setcounter{figure}{0}
\setcounter{equation}{0} \setcounter{lemma}{0} \setcounter{proposition}{0}
\setcounter{theorem}{0} \setcounter{subsection}{0} \setcounter{corollary}{0}
\vskip .05in \noindent \textbf{C. Residual Plots from Mercury Concentration Studies}

In this section, we provide some diagnosis plots of the residuals.  Figure \ref{FIG:MC_fitted_res} provides the residuals vs fits plots for five different methods. From Figure \ref{FIG:MC_fitted_res}, one sees that the residuals ``bounce randomly" around the zero line, and no residual ``stands out" from the basic random pattern of residuals. Figure \ref{FIG:MC_res} further demonstrates the residual scatter plot using five different methods. As seen in Figure  \ref{FIG:MC_res}, the absolute values of the residuals are relatively higher in the middle of the Piscataqua river for KRIG and TPS compared to that of the BPST. Due to the small sample size and the complex terrain, all methods have some difficulty in the estimation at the confluence of the Salmon Falls River and Cocheco River. According to \cite{steve2009distribution}, the accumulation of mercury in this area is complex and includes aspects of transport from urban point sources, atmospheric deposition from local and distant sources, prevailing currents, equilibrium processes between overlying water and the quality of sediments. Further research is warranted.

%%%%%%%%%%%%%%%%%%%%%%%%%%%%%%%%%%%%%%%%%%%%%%%%%%%%%%%%%%%%%%%%%%%%%%
\begin{figure}[htbp]
	\begin{center}
		\begin{tabular}{cc}
			\includegraphics[height=5cm]{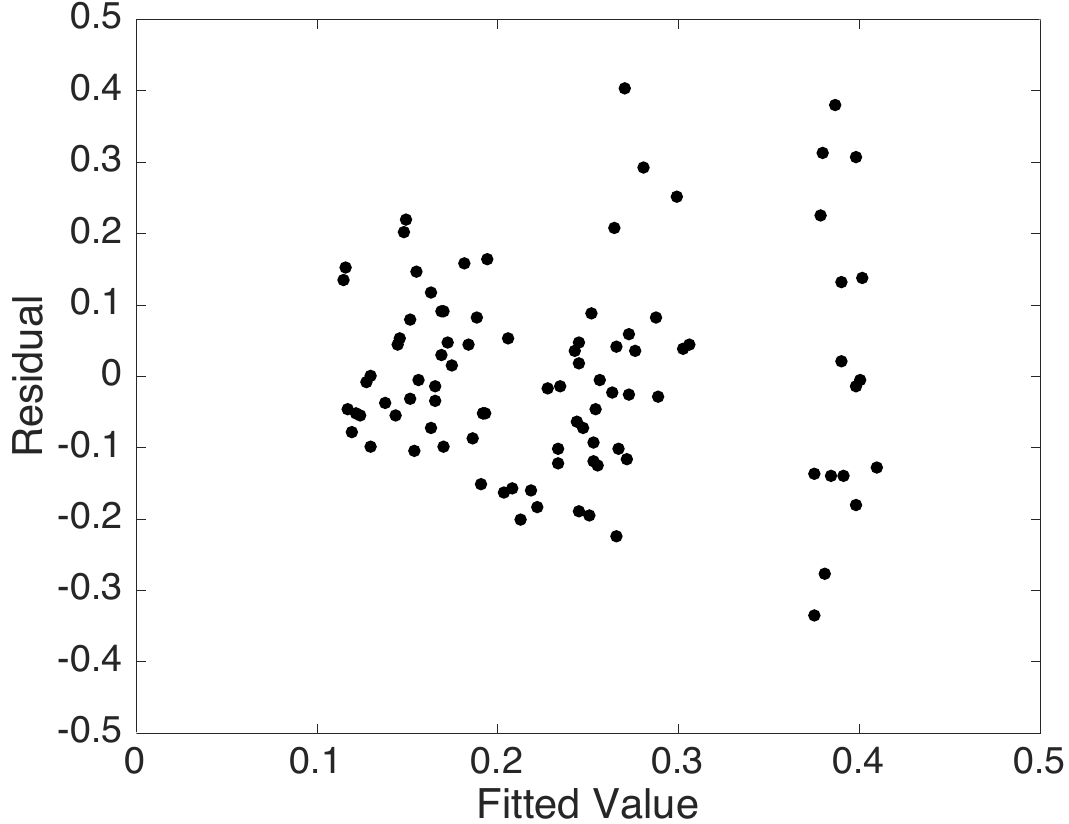} &
			\includegraphics[height=5cm]{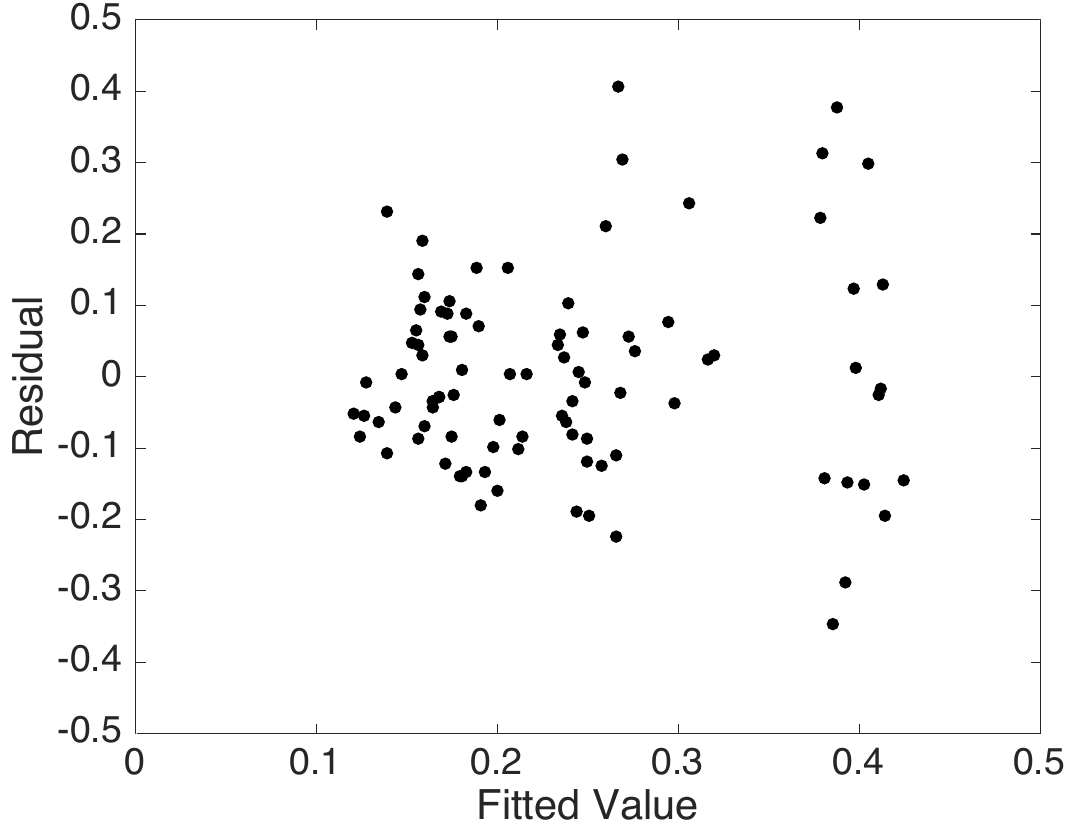}  \\
			(a) KRIG & (b) TPS \\[3pt]
			\includegraphics[height=5cm]{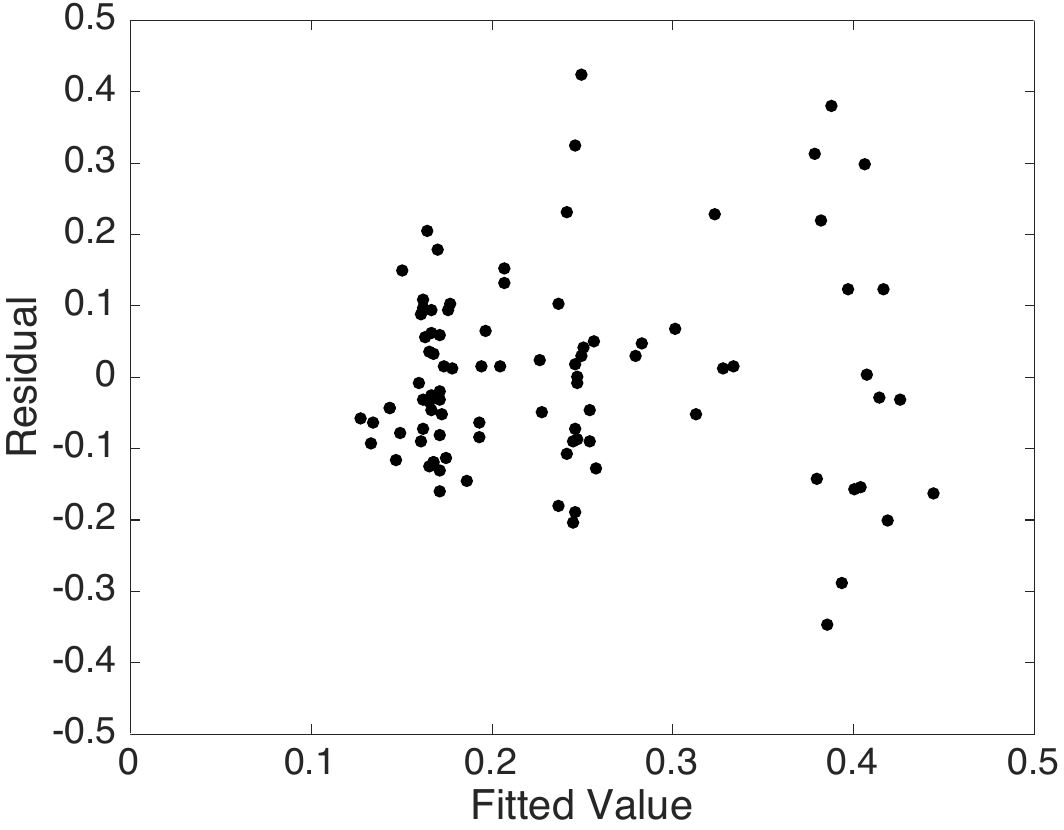} &
			\includegraphics[height=5cm]{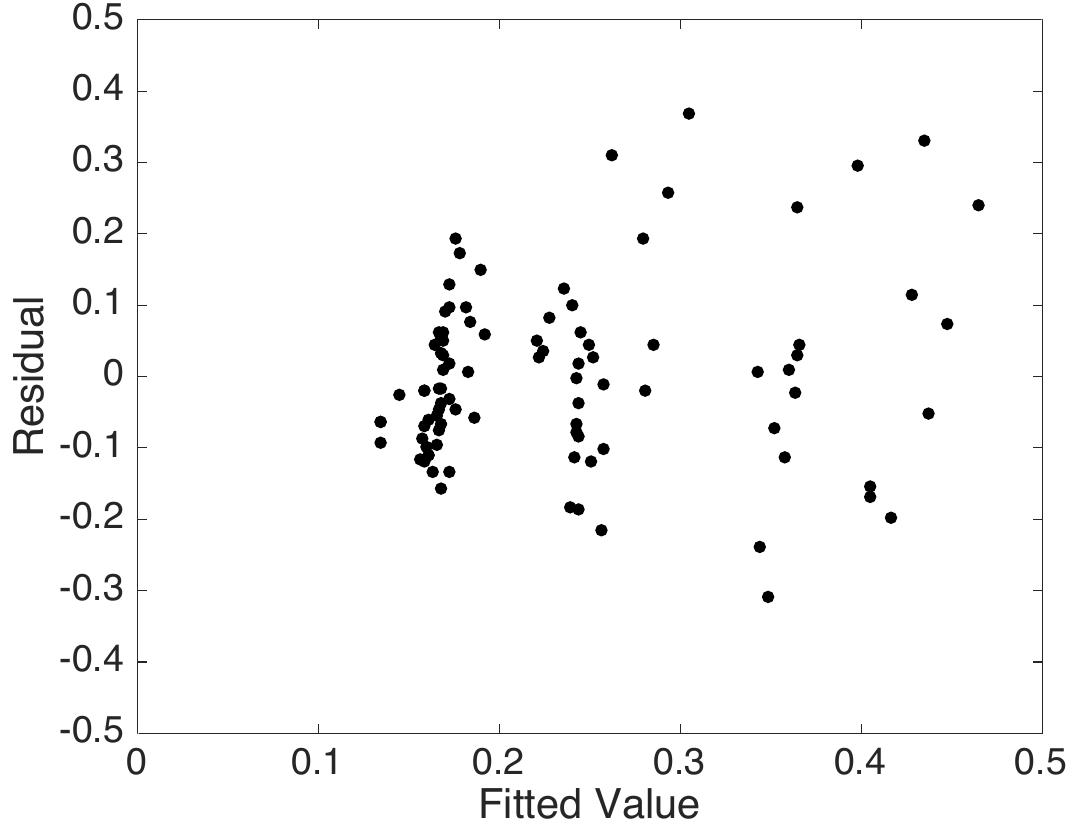} \\
			(c) GLTPS & (d) FEM\\[3pt]
			\includegraphics[height=5cm]{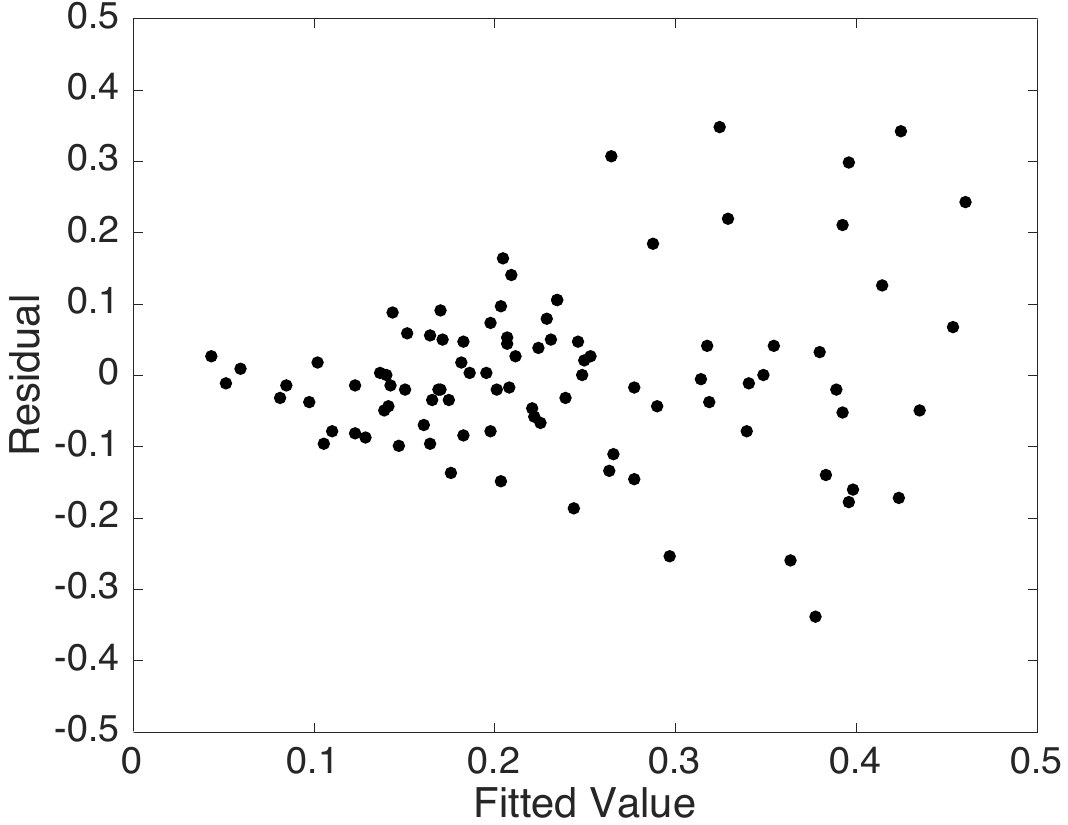}
			& \\
			(e) BPST &
		\end{tabular}
	\end{center} \vskip -.2in
	\caption{Plots of the residuals vs fitted values of mercury concentrations.}
	\label{FIG:MC_fitted_res}
\end{figure}

%%%%%%%%%%%%%%%%%%%%%%%%%%%%%%%%%%%%%%%%%%%%%%%%%%%%%%%%%%%%%%%%%%%%%%
\begin{figure}[htbp]
	\begin{center}
		\begin{tabular}{cc}
			\includegraphics[height=5.75cm,width=5.75cm]{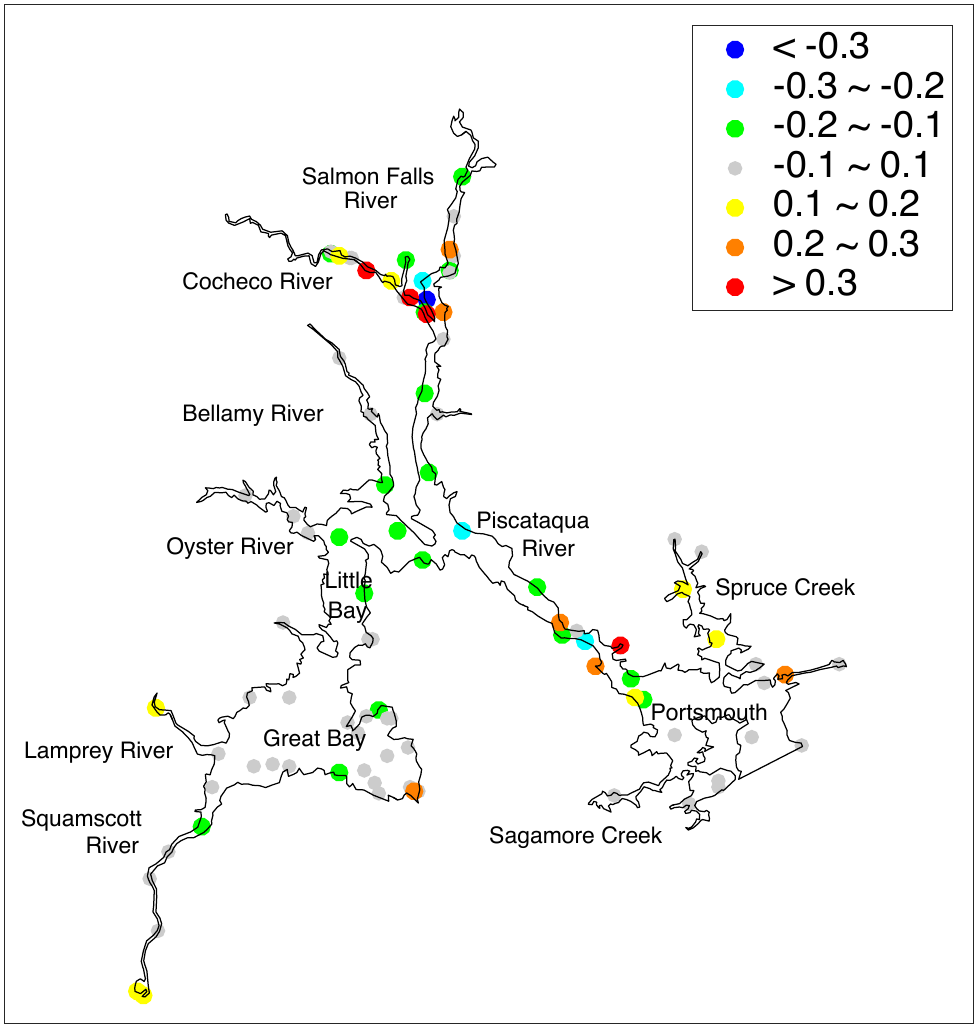} &
			\includegraphics[height=5.75cm,width=5.75cm]{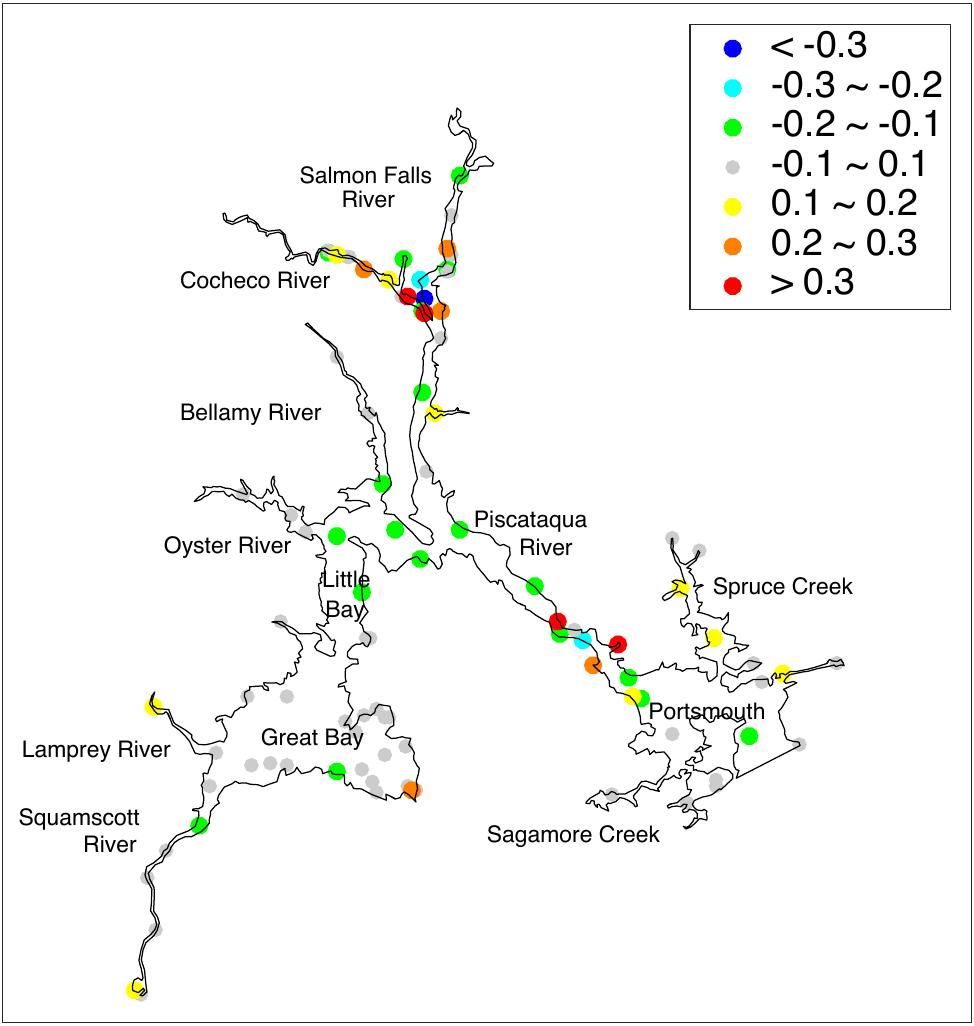}  \\
			(a) KRIG & (b) TPS \\[3pt]
			\includegraphics[height=5.75cm,width=5.75cm]{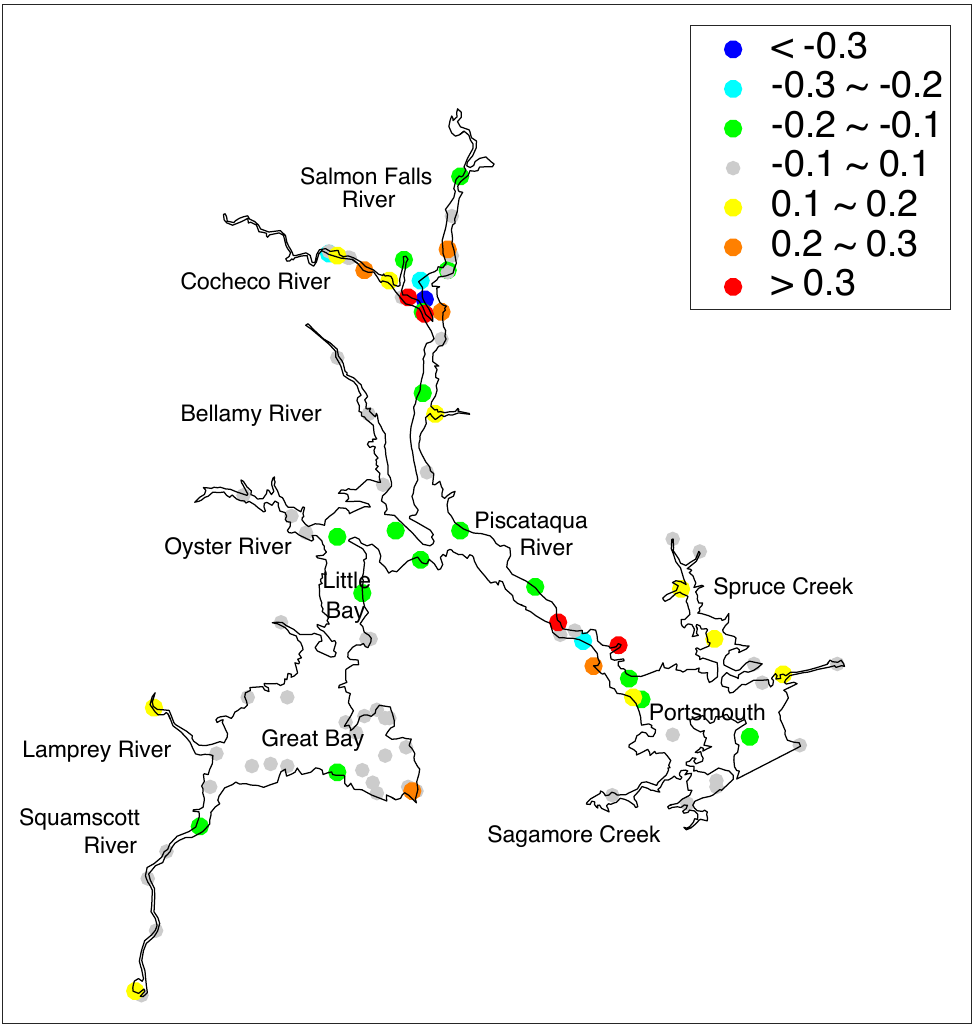} &
			\includegraphics[height=5.75cm,width=5.75cm]{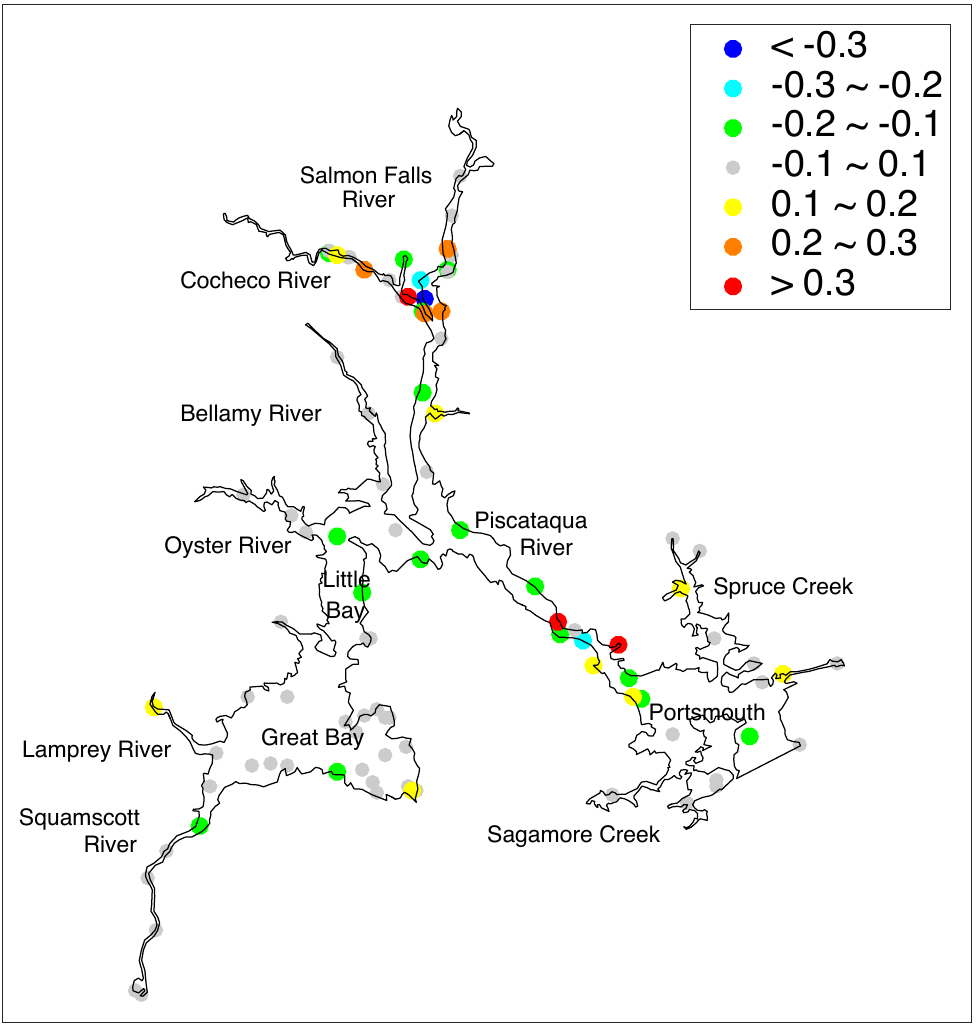} \\
			(c) GLTPS & (d) FEM\\[3pt]
			\includegraphics[height=5.75cm,width=5.75cm]{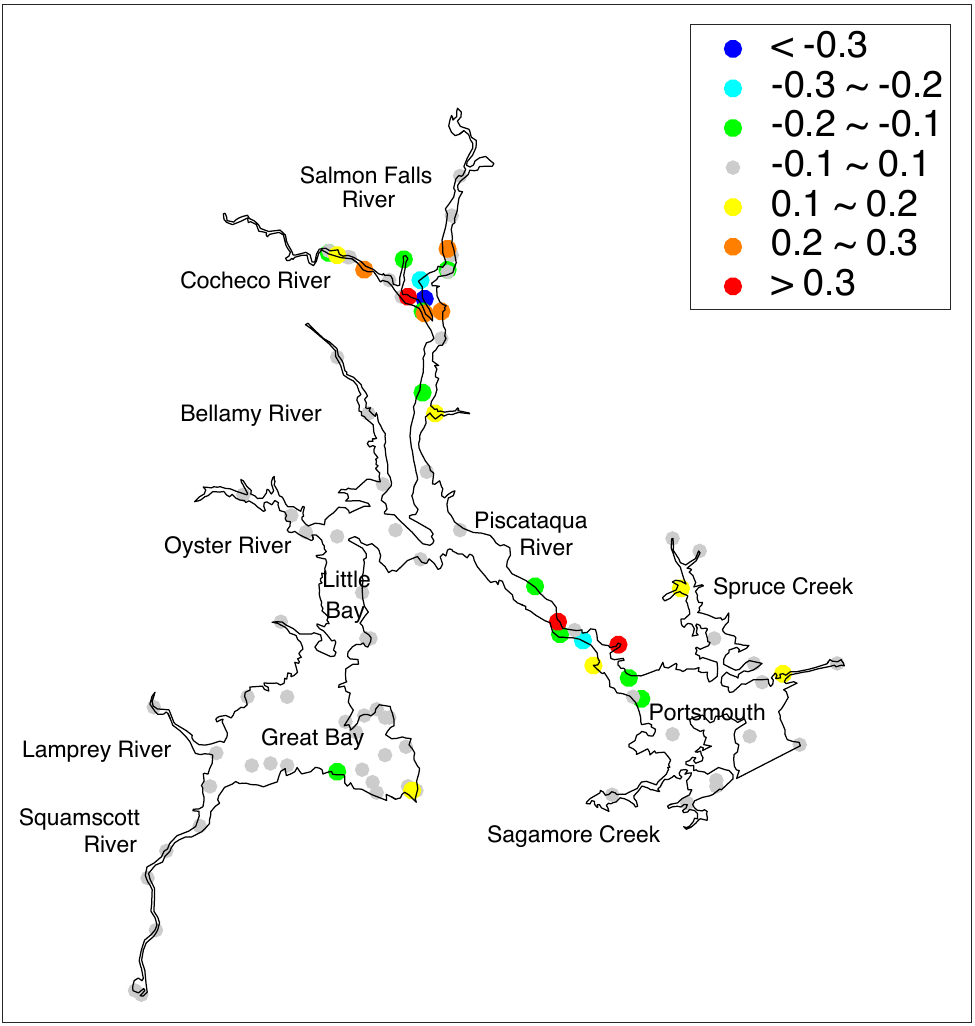}
			& \includegraphics[height=5.75cm,width=5.75cm]{MC_map.pdf}\\
			(e) BPST & (f) Observed Data
		\end{tabular}
	\end{center} \vskip -.2in
	\caption{Residual maps of mercury concentrations over the estuaries in New Hampshire.}
	\label{FIG:MC_res}
\end{figure}

%%%%%%%%%%%%%%%%%%%%%%%%%%%%%%%%%%%%%%%%%%%%%%%%%%%%%%%%%%%%%%
%%%%%%%%%%%%%%%%%%%%%%%%%%%%%%%%%%%%%%%%%%%%%%%%%%%%%%%%%%%%%%
%%%%%%%%%%%%%%%%%%%%%%%%%%%%%%%%%%%%%%%%%%%%%%%%%%%%%%%%%%%%%%
\setcounter{chapter}{11} \renewcommand{\thetheorem}{D.\arabic{theorem}}
\renewcommand{\theproposition}{D.\arabic{proposition}}
\renewcommand{\thelemma}{D.\arabic{lemma}}
\renewcommand{\thecorollary}{D.\arabic{corollary}}
\renewcommand{\theequation}{D.\arabic{equation}} \renewcommand{\thesubsection}{D.\arabic{subsection}}
\renewcommand{\thetable}{{D.\arabic{table}}} \setcounter{table}{0}
\renewcommand{\thefigure}{D.\arabic{figure}} \setcounter{figure}{0}
\setcounter{equation}{0} \setcounter{lemma}{0} \setcounter{proposition}{0}
\setcounter{theorem}{0} \setcounter{subsection}{0} \setcounter{corollary}{0}
\vskip .05in \noindent \textbf{D. Technical Lemmas}

In the following, we use $c$, $C$, $c_1$, $c_2$, $C_1$, $C_2$, etc. as generic constants, which may be different even in the same line. For functions $f_{1}$ and $f_{2}$ on $\Omega \times \mathbb{R}^{p}$, we define the empirical inner product and norm as
$\left\langle f_{1},f_{2}\right\rangle_{n}=\frac{1}{n}\sum_{i=1}^{n} f_{1}(\mathbf{X}_{i},\mathbf{Z}_{i})f_{2}(\mathbf{X}_{i},\mathbf{Z}_{i})$ and $\left\|
f_{1}\right\| _{n}^{2}=\left\langle f_{1},f_{1}\right\rangle_{n}$. If $f_{1}$ and $f_{2}$ are $L^2$-integrable, we define the theoretical inner
product and theoretical $L^2$ norm as $\left\langle f_{1},f_{2}\right\rangle_{L^2} =E\left\{ f_{1}(\mathbf{X}_{i},\mathbf{Z}_{i})f_{2}(\mathbf{X}_{i},\mathbf{Z}_{i})\right\}$ and $\left\|
f_{1}\right\| _{L^2}=\left\langle f_{1},f_{1}\right\rangle_{L^2}$. Furthermore, let $\left\| \cdot\right\|
_{\mathcal{E}_{\upsilon }}$ be the norm introduced by the inner product $\left\langle \cdot, \cdot\right\rangle _{\mathcal{E}_{\upsilon
}}$, where, for $g_{1}$ and $g_{2}$ on $\Omega$,
%\[
%	\left\langle g_{1},g_{2}\right\rangle_{\mathcal{E}_{\upsilon
%	}}=\int_{\Omega}\left\{\sum_{i+j=\upsilon }
%	\binom{\upsilon}{i}
%	(D_{x_{1}}^{i}D_{x_{2}}^{j}g_{1})^{2}\right\}^{1/2}\!\left\{\sum_{i+j=\upsilon }
%	\binom{\upsilon}{i}
%	(D_{x_{1}}^{i}D_{x_{2}}^{j}g_{2})^{2}\right\}^{1/2}dx_{1}dx_{2}.
%\]
\[
\left\langle g_{1},g_{2}\right\rangle_{\mathcal{E}_{\upsilon}} = \int_{\Omega} \sum_{i+j=\upsilon} \binom{\upsilon}{i} (D_{x_1}^iD_{x_2}^jg_{1}) (D_{x_1}^iD_{x_2}^jg_{2})dx_1dx_2.
\]
%%%%%%%%%%%%%%%%%%%%%%%%%%%%%%%%%%%%%%%%%%%%%%%%%%%%%%%%%%%%%%%%%%%%%%
\begin{proof}[\textbf{Proof of Lemma 1}]
	By \eqref{QR4H}, we have $\mathbf{H}^{\T}=\mathbf{Q}_1 \mathbf{R}_{1}$ since $\mathbf{R}_2=\mathbf{0}$. That is,
	$\mathbf{H}= \mathbf{R}_{1}^{\T} \mathbf{Q}_1^{\T}$. Thus,
	\begin{align*}
	\mathbf{H} \bs{\gamma}=& \mathbf{H} \mathbf{Q}_2\bs{\theta}= \mathbf{R}_{1}^{\T} \mathbf{Q}_1^{\T} \mathbf{Q}_2\bs{\theta} = \mathbf{0}
	\end{align*}
	since $\mathbf{Q}_1^{\T} \mathbf{Q}_2=\mathbf{0}$. On the other hand, if
	\[
	\mathbf{0}=\mathbf{H} \bs{\gamma} = \mathbf{R}_{1}^{\T} \mathbf{Q}_1^{\T}\bs{\gamma},
	\]
	we have $\mathbf{Q}_1^{\T}\bs{\gamma}=0$ since $\mathbf{R}_1$ is invertible. Thus, $\bs{\gamma}$ is in the perpendicular subspace of the space spanned by the columns of $\mathbf{Q}_1$. That is, $\bs{\gamma}$ is in the space spanned by the columns of $\mathbf{Q}_2$. Thus, there exists a vector $\bs{\theta}$ such that $\bs{\gamma} = \mathbf{Q}_2\bs{\theta}$. These complete the proof.
\end{proof}

\begin{lemma}{[\cite{Lai:Schumaker:07}]}
	\label{LEM:normequity}
	Let $\{B_{\xi }\}_{\xi \in \mathcal{K}}$ be the Bernstein polynomial basis for spline space
	$\mathbb{S}$ with smoothness $r$, where
	$\mathcal{K}$ stands for an index set.  Then there exist positive
	constants $c$, $C$ depending on the smoothness $r$ and the shape parameter $\delta$ in Condition (C6) such that
	\[
	c|\triangle |^{2}\sum_{\xi \in \mathcal{K}}|\gamma_{\xi }|^{2}\le
	\Vert \sum_{\xi \in \mathcal{K}}\gamma_{\xi }B_{\xi }\Vert _{L^2}^{2}\le C|\triangle |^{2}\sum_{\xi \in \mathcal{K}}|\gamma_{\xi}|^{2}
	\label{EQ:normequity}
	\]
	for all $\gamma_{\xi },\xi \in \mathcal{K}$.
\end{lemma}

With the above stability condition, \cite{Lai:Wang:13} established the following
uniform rate at which the empirical inner product approximates the
theoretical inner product.

%%%%%%%%%%%%%%%%%%%%%%%%%%%%%%%%%%%%%%
\begin{lemma}{[Lemma 2 of the Supplement of \cite{Lai:Wang:13}]}
	\label{LEM:Rnorder}
	Let $g_{1}=\sum_{\xi \in \mathcal{K}}c_{\xi
	}B_{\xi }$, $g_{2}=\sum_{\zeta \in \mathcal{K}}\widetilde{c}_{\zeta
	}B_{\zeta }$ be any spline functions in $\mathbb{S}$. Under Conditions (C4) and (C6),
	\[
	\sup\limits_{g_{1},g_{2}\in \mathbb{S}}\left|
	\frac{\left\langle g_{1},g_{2}\right\rangle_{n}-\left\langle g_{1},g_{2}\right\rangle _{L^2}}{\left\|
		g_{1}\right\| _{L^2}\left\| g_{2}\right\| _{L^2}}\right| =O_{P}\left\{(N\log n)^{1/2}/{n}^{1/2}\right\}.
	\]
\end{lemma}

For any smooth bivariate function $g(\cdot)$ and $\lambda>0$, define
\begin{equation}
s_{\lambda,g}=\mathrm{argmin}_{s\in \mathbb{S}} \sum_{i=1}^{n}\{g(\mathbf{X}_{i})-s(\mathbf{X}_{i})\}^{2}+\lambda\mathcal{E}_{\upsilon}(s)
\label{DEF:s_lambdag}
\end{equation}
the penalized least squares splines of $g(\cdot)$. Then the non-penalized  solution $s_{0,g}$ is  the discrete least squares
spline estimator of $g(\cdot)$.

\begin{lemma} {[Corollary of Theorem 6 in \cite{Lai:08}]}
	\label{LEM:bias}
	Assume $g(\cdot)$ is in Sobolev space $W^{\ell +1,\infty }(\Omega )$. For bi-integer $(\alpha_{1},\alpha _{2})$
	with $0\leq {\alpha _{1}}+{\alpha _{2}}%
	\leq \upsilon $, there exists an absolute constant $C$ depending on $r$ and $\delta$, such that with probability approaching 1,
	\[
	\Vert D_{x_{1}}^{\alpha _{1}}D_{x_{2}}^{\alpha _{2}}\left(
	g-s_{0,g}\right) \Vert _{\infty}\le
	C\frac{F_{2}}{F_{1}}|\triangle |^{\ell +1-\alpha _{1}-\alpha
		_{2}}|g|_{\ell +1,\infty},
	\]
	where $F_2$ appears in Assumption (C3) and $F_1>0$ is a constant
	in a different version of Assumption C2 \citep{Lai:08}.
\end{lemma}

We remark that the current version of Assumption (C2) is an improvement
of the original Assumption (C2). The improvement requires an extensive
study. We leave it to a future publication.

%%%%%%%%%%%%%%%%%%%%%%%%%%%%%%%%%%%%%%
\begin{lemma}
	\label{LEM:norder}
	Suppose $g(\cdot)$ is in the Sobolev space $W^{\ell +1,\infty }(\Omega )$, and let $s_{\lambda,g}$ be its penalized spline estimator defined in (\ref{DEF:s_lambdag}). Under Conditions (C2), (C3) and (C6),
	\begin{eqnarray*}
		\left\|g-s_{\lambda, g}\right\|_{n}&\!\!\!=\!\!\!&
		O_{P}\left\{\frac{F_{2}}{%
			F_{1}}\left| \triangle \right| ^{\ell +1}\left|g\right|
		_{\ell+1,\infty}\right.\\
		&&+\left.\frac{\lambda }{n\left| \triangle \right| ^{2}}%
		\left( \left|g\right| _{\upsilon,\infty}+\frac{F_{2}}{F_{1}}\left| \triangle \right| ^{\ell+1-\upsilon}\left|g\right|
		_{\ell +1,\infty}\right) \right\}.
	\end{eqnarray*}
\end{lemma}

\textit{Proof.} Note that $s_{\lambda,g}$ is
characterized by the orthogonality relations
\begin{eqnarray}
n\left\langle g-s_{\lambda, g},u\right\rangle_{n}=\lambda
\left\langle s_{\lambda, g},u\right\rangle _{\mathcal{E}_{\upsilon
}},\quad \textrm{for all }u\in \mathbb{S},  \label{EQ:pls}
\end{eqnarray}
while $s_{0, g}$ is characterized by
\begin{eqnarray}
\left\langle g-s_{0, g},u\right\rangle_{n}=0,\quad \textrm{for all }u\in \mathbb{S}.  \label{EQ:ls}
\end{eqnarray}
By (\ref{EQ:pls}) and (\ref{EQ:ls}), $ n\left\langle s_{0,
	g}-s_{\lambda, g},u\right\rangle_{n}=\lambda \left\langle
s_{\lambda,g},u\right\rangle _{\mathcal{E}_{\upsilon }}$, for all
$u\in \mathbb{S}$. Replacing $u$ by $s_{0, g}-s_{\lambda,g}$ yields
that
\begin{eqnarray}
n\left\| s_{0, g}-s_{\lambda, g}\right\| _{n}^{2}=\lambda
\left\langle s_{\lambda, g},s_{0, g}-s_{\lambda, g}\right\rangle _{\mathcal{E}%
	_{\upsilon }}.  \label{EQ:pls-ls}
\end{eqnarray}
Thus, by Cauchy-Schwarz inequality,
\begin{eqnarray*}
	n\left\| s_{0, g}-s_{\lambda, g}\right\| _{n}^{2}&\leq&
	\lambda \left\| s_{\lambda, g}\right\| _{\mathcal{E}_{\upsilon
	}}\left\| s_{0, g}-s_{\lambda, g}\right\| _{\mathcal{E}_{\upsilon
	}}\\
	&\leq& \lambda \left\| s_{\lambda, g}\right\|
	_{\mathcal{E}_{\upsilon }} \sup_{f\in \mathbb{S}}\left\{ \frac{%
		\left\| f\right\| _{\mathcal{E}_{\upsilon }}}{\left\| f\right\|
		_{n}},\left\| f\right\| _{n}\neq 0\right\}
	\left\| s_{0, g}-s_{\lambda, g}\right\|
	_{n}.
\end{eqnarray*}
Similarly, using (\ref{EQ:pls-ls}), we have
\[
n\left\| s_{0, g}-s_{\lambda, g}\right\| _{n}^{2}=\lambda \left\{
\left\langle s_{\lambda, g},s_{0, g}\right\rangle
_{\mathcal{E}_{\upsilon
}}-\left\langle s_{\lambda, g},s_{\lambda, g}\right\rangle _{\mathcal{E}%
	_{\upsilon }}\right\} \geq 0.
\]
Therefore, by Cauchy-Schwarz inequality,
\[
\left\| s_{\lambda, g}\right\|
_{\mathcal{E}_{\upsilon }}^{2}\leq \left\langle s_{\lambda
	,g},s_{0, g}\right\rangle _{\mathcal{E}_{\upsilon }}\leq \left\|
s_{\lambda, g}\right\| _{\mathcal{E}_{\upsilon }}\left\| s_{0, g}\right\| _{\mathcal{E}_{\upsilon}},
\]
which implies that $\left\| s_{\lambda, g}\right\|
_{\mathcal{E}_{\upsilon }}\leq \left\| s_{0, g}\right\|
_{\mathcal{E}_{\upsilon }}$. Therefore,
\[
\left\| s_{0, g}-s_{\lambda, g}\right\| _{n}\leq
n^{-1}\lambda \left\| s_{0, g}\right\|
_{\mathcal{E}_{\upsilon }} \sup_{f\in \mathbb{S}}\left\{ \frac{%
	\left\| f\right\| _{\mathcal{E}_{\upsilon }}}{\left\| f\right\|
	_{n}},\left\| f\right\| _{n}\neq 0\right\}.
\]
By Lemma \ref{LEM:bias}, with probability approaching 1,
\begin{eqnarray}
\left\| s_{0, g}\right\| _{\mathcal{E}_{\upsilon }} &\leq &C_{1}A_\Omega\left\{
\left| g\right| _{\upsilon,\infty}+\sum_{\alpha
	_{1}+\alpha _{2}=\upsilon }\left\| D_{x_{1}}^{\alpha
	_{1}}D_{x_{2}}^{\alpha _{2}}\left(
g - s_{0, g}\right) \right\| _{\infty}\right\}  \notag\\
&\leq &C_{2}A_\Omega\left( \left| g\right| _{\upsilon,\infty}+\frac{F_{2}}{F_{1}}%
\left| \triangle \right| ^{\ell +1-\upsilon}\left| g\right| _{\ell
	+1,\infty}\right),
\label{EQ:s_0g_Ev}
\end{eqnarray}
where $A_\Omega$ denotes the area of $\Omega$.
By Markov's inequality, for any $f\in \mathbb{S}$, $\left\|f\right\|
_{\mathcal{E}_{\upsilon }}\leq C \left| \triangle \right|
^{-2}\left\|f\right\| _{L^2}$.
Lemma (\ref{LEM:Rnorder}) implies that
\[
\sup_{f\in\mathcal{S}}\left\{ \left. \left\|f\right\| _{n}\right/
\left\|f\right\| _{L^2}\right\} \geq 1-O_{P}\left\{(N\log n)^{1/2}/{n}^{1/2}\right\}.
\]
Thus, we have
\begin{eqnarray}
\sup_{f\in \mathbb{S}}\left\{ \frac{%
	\left\| f\right\| _{\mathcal{E}_{\upsilon }}}{\left\| f\right\|
	_{n}},\left\| f\right\| _{n}\neq 0\right\} &\!\!\!\leq\!\!\!& C\left| \triangle \right|
^{-2}\left[ 1-O_{P}\left\{(N\log n)^{1/2}/{n}^{1/2}\right\}
\right] ^{-1/2}\notag\\
&\!\!\!=\!\!\!&O_{P}\left( \left| \triangle \right| ^{-2}\right).
\label{EQ:ev-vs-n}
\end{eqnarray}
Therefore,
\begin{eqnarray*}
	\left\| s_{0, g}-s_{\lambda, g}\right\| _{n}
	&=&O_P\left\{\frac{\lambda}{n\left| \triangle \right|^{2}}
	\left( \left| g\right| _{\upsilon,\infty}+\frac{F_{2}}{F_{1}}%
	\left| \triangle \right| ^{\ell +1-\upsilon}\left| g\right| _{\ell
		+1,\infty}\right)
	\right\},\\
	\left\| g-s_{\lambda ,g}\right\| _{n}&\leq& \left\|
	g-s_{0,g}\right\| _{n}+\left\| s_{0,g}-s_{\lambda, g}\right\| _{n}.
\end{eqnarray*}
By Lemma \ref{LEM:bias},
\[
\left\|
g-s_{0,g}\right\| _{n} \leq \left\|
g-s_{0,g}\right\| _{\infty}=O_{P}\left(\frac{F_{2}}{%
	F_{1}}\left| \triangle \right| ^{\ell +1}\left|g\right|
_{\ell+1,\infty}\right).
\]
Thus, the desired result is established. $\square $

%%%%%%%%%%%%%%%%%%%%%%%%%%%%%%%%%%%%%%
\begin{lemma}
	Under Assumptions (A1), (A2), (C4)-(C6), there exist constants $0 < c_{U} < C_{U} <
	\infty $, such that with probability approaching 1 as $n\rightarrow \infty$, $c_{U}\mathbf{I}_{p\times p}
	\leq n\mathbf{U}_{11} \leq C_{U}\mathbf{I}_{p\times p}$, where $\mathbf{U}_{11}$ is given in (\ref{DEF:V-inverse}).
	\label{LEM:U11}
\end{lemma}

\textit{Proof.}
Denote by
\[
\bs{\Gamma} _{\lambda }=\frac{1}{n}\left(\mathbf{B}^{\T} \mathbf{B}+\lambda\mathbf{P}\right)
=\left[ \frac{1}{n}\sum_{i=1}^{n}B_{\xi }\left(
\mathbf{X}_{i}\right) B_{\zeta }\left( \mathbf{X}_{i}\right)
+\frac{\lambda }{n}\langle B_{\xi },B_{\zeta }\rangle
_{\mathcal{E}_{\upsilon }}\right] _{\xi ,\zeta \in \mathcal{K}}
\label{DEF:Gamma_lambda}
\]
a symmetric positive definite matrix. Then for $\mathbf{V}_{22}$ defined in (\ref{DEF:V}), we can rewrite it as
$\mathbf{V}_{22}=n\mathbf{Q}_{2}^{\T}\bs{\Gamma} _{\lambda }
\mathbf{Q}_{2}$.
Let $\alpha _{\min}(\lambda )$ and $\alpha _{\max}(\lambda )$ be the
smallest and largest eigenvalues of $\bs{\Gamma}_{\lambda }$. As shown in the proof of
Theorem 2 in the Supplement of \cite{Lai:Wang:13}, there exist positive constants $0<c_3<C_3$ such that under Conditions (C4) and (C5), with probability approaching 1, we have
\[
c_3|\triangle |^{2} \leq \alpha _{\min}(\lambda ) \leq
\alpha _{\max}(\lambda ) \leq  C_3\left(|\triangle |^{2}+\frac{\lambda}{n|\triangle |^{2}}\right).
\]
Therefore, we have
\[
c_{4}\left(|\triangle |^{2}+\frac{\lambda}{n|\triangle |^{2}}\right)^{-1} \|\mathbf{a}\|^{2}
\leq n \mathbf{a}^{\T}\mathbf{V}_{22}^{-1}\mathbf{a}
= \mathbf{a}^{\T}(\mathbf{Q}_2^{\T} \bs{\Gamma} _{\lambda}\mathbf{Q}_2)^{-1}\mathbf{a}
\leq C_{4}|\triangle |^{-2} \|\mathbf{a}\|^{2}.
\]
Thus, by Assumption (A2), we have with probability approaching 1
\begin{eqnarray}
c_{5}\left(|\triangle |^{2}+\frac{\lambda}{n|\triangle |^{2}}\right)^{-1}|\triangle |^{2}\|\mathbf{a}\|^{2}
&\!\!\!\leq\!\!\!& \mathbf{a}^{\T}\mathbf{V}_{12}\mathbf{V}_{22}^{-1} \mathbf{V}_{21}\mathbf{a} \nonumber\\
&\!\!\!=\!\!\!&\mathbf{a}^{\T}\mathbf{Z}^{\T}\mathbf{B}\mathbf{Q}_2\mathbf{V}_{22}^{-1} \mathbf{Q}_2^{\T}\mathbf{B}^{\T}\mathbf{Z}\mathbf{a}
\leq C_{5}\|\mathbf{a}\|^{2}.
\label{EQ:eigen}
\end{eqnarray}
According to (\ref{DEF:V-inverse}) and (\ref{DEF:U11}), we have
\[
(n\mathbf{U}_{11})^{-1}=n^{-1}(\mathbf{V}_{11}-\mathbf{V}_{12}\mathbf{V}_{22}^{-1}\mathbf{V}_{21})
=n^{-1}(\mathbf{Z}^{\T}\mathbf{Z} -\mathbf{V}_{12}\mathbf{V}_{22}^{-1} \mathbf{V}_{21}).
\]
The desired result follows from Assumptions (A1), (A2) and (\ref{EQ:eigen}). \hfill $\Box$

%%%%%%%%%%%%%%%%%%%%%%%%%%%%%%%%%%%%%%%%%%%%%%%%%%%%%%%%%%%%%%%%%%%%%%
\vskip .05in \noindent \textbf{D.1. Proof of Theorem \ref{THM:beta-normality}}

Let $\mu_i=\mathbf{Z}_{i}^{\T}\bs{\beta}_{0}+g_{0}(\mathbf{X}_{i})$, $\bs{\mu}^{\T}=(\mu_1,\ldots,\mu_n)$, and let $\bs{\epsilon}^{\T}=(\epsilon_1,\ldots,\epsilon_n)$. Define
\begin{eqnarray}
\widetilde{\bs{\beta }}_{\mu}&=&
\mathbf{U}_{11}
\mathbf{Z}^{\T}\left(\mathbf{I}-\mathbf{B}\mathbf{Q}_2\mathbf{V}_{22}^{-1} \mathbf{Q}_2^{\T}\mathbf{B}^{\T}\right)\bs{\mu}, \label{betatilde_mu}\\
\widetilde{\bs{\beta }}_{\epsilon}&=&
\mathbf{U}_{11}
\mathbf{Z}^{\T}\left(\mathbf{I}-\mathbf{B}\mathbf{Q}_2\mathbf{V}_{22}^{-1} \mathbf{Q}_2^{\T}\mathbf{B}^{\T}\right)\bs{\epsilon}. \label{betatilde_e}
\end{eqnarray}
Then $\widehat{\bs{\beta }}-\bs{\beta }_{0}=(
\widetilde{\bs{\beta }}_{\mu}-\bs{\beta }_{0})+\widetilde{%
	\bs{\beta }}_{\epsilon}$.

%%%%%%%%%%%%%%%%%%%%%%%%%%%%%%%%%%%%%
\begin{lemma}
	\label{LEM:betamu-beta}
	Under Assumptions (A1), (A2), (C1)-(C5), $\|\widetilde{\bs{\beta }}_{\mu}-\bs{\beta }_{0}\| =o_{P}\left(n^{-1/2}\right)$ for $\widetilde{\bs{\beta }}_{\mu}$ in (\ref{betatilde_mu}).
\end{lemma}

\textit{Proof.}
Let $\mathbf{g}_{0}=(g_{0}(\mathbf{X}_{1}),\ldots, g_{0}(\mathbf{X}_{n}))^{\T}$. It is clear that
\begin{eqnarray*}
	\widetilde{\bs{\beta}}_{\mu}-\bs{\beta}_{0}&=&
	\mathbf{U}_{11}
	\mathbf{Z}^{\T}\left(\mathbf{I}-\mathbf{B}\mathbf{Q}_2\mathbf{V}_{22}^{-1} \mathbf{Q}_2^{\T}\mathbf{B}^{\T}\right)\mathbf{g}_{0}\\
	&=&\mathbf{U}_{11}\mathbf{Z}^{\T}\left[\mathbf{g}_{0}- \mathbf{B}\mathbf{Q}_2
	\{\mathbf{Q}_2^{\T}(\mathbf{B}^{\T} \mathbf{B}+\lambda\mathbf{P})\mathbf{Q}_2\}^{-1} \mathbf{Q}_2^{\T}\mathbf{B}^{\T}\mathbf{g}_{0}
	\right]\\
	&=&n\mathbf{U}_{11}\mathbf{A},
\end{eqnarray*}%
where $\mathbf{A}=\left( A_{1},\ldots ,A_{p}\right)^{\T}$, with
\begin{equation*}
A_{j}=n^{-1}\mathbf{Z}_{j}^{\T}\left[\mathbf{g}_{0}-
\mathbf{B}\mathbf{Q}_2\{\mathbf{Q}_2^{\T}(\mathbf{B}^{\T} \mathbf{B}+\lambda\mathbf{P})\mathbf{Q}_2\}^{-1} \mathbf{Q}_2^{\T}\mathbf{B}^{\T}
\mathbf{g}_{0}
\right]
\end{equation*}%
for $\mathbf{Z}_{j}^{\T}=(Z_{1j},...,Z_{nj})$. Next we derive the order
of $A_{j}$, $1\leq j\leq p$, as follows. For any $g_{j}\in \mathbb{S}$, by (\ref{EQ:pls}) we
have
\[
A_{j}=\left\langle z_{j}, g_{0} -s_{\lambda,g_{0}} \right\rangle_{n}
=\left\langle z_{j}-g_{j},g_{0} -s_{\lambda,g_{0}} \right\rangle_{n}
+\frac{\lambda}{n}\left\langle s_{\lambda,g_{0}}, g_{j} \right\rangle_{\mathcal{E}_{\upsilon}}.
\]
For any $j=1,\ldots, p$, let $h_{j}(\cdot)$ be the function $h(\cdot)$ that minimizes $E\{
Z_{ij}-h(\mathbf{X}_{i})\} ^{2}$ as defined in (\ref{EQ:h_j}). According to Lemma \ref{LEM:bias}, there exists a function  $\widetilde{h}_{j}\in \mathbb{S}$ satisfy
\begin{equation}
\|\widetilde{h}_{j}-h_{j}\|_{\infty}\leq C \frac{F_{2}}{%
	F_{1}}\left| \triangle \right| ^{\ell +1}\left|h_{j}\right|_{\ell+1,\infty} ,
\label{EQ:h_j_tilde}
\end{equation}
then
\[
A_{j}=\langle z_{j}-h_{j}, g_{0} -
s_{\lambda,g_{0}} \rangle_{n}+\langle h_{j}-%
\widetilde{h}_{j}, g_{0} -s_{\lambda,g_{0}} \rangle_{n}
+\frac{\lambda}{n}\langle s_{\lambda,g_{0}}, \widetilde{h}_{j} \rangle_{\mathcal{E}_{\upsilon}}
= A_{j,1}+A_{j,2}+A_{j,3}.
\]
Since $h_{j}$ satisfies $\left\langle z_{j}-h_{j},\psi \right\rangle_{L_{2}(\Omega)} =0$ for any $\psi \in L_{2}(\Omega)$, $E\left(A_{j,1}\right) =0$. According to Proposition 1 in \cite{Lai:Wang:13},
\begin{eqnarray*}
	\left\|g_{0}-s_{\lambda, g_{0}}\right\| _{\infty}&=&O_{P}\left\{ \frac{F_{2}}{%
		F_{1}}\left| \triangle \right| ^{\ell +1}\left|g_{0}\right|
	_{\ell+1,\infty}\right.\\
	&&+\left.\frac{\lambda }{n\left| \triangle \right| ^{3}}%
	\left( \left|g_{0}\right| _{2,\infty}+\frac{F_{2}}{F_{1}}\left| \triangle \right| ^{\ell -1}\left|
	g_{0}\right| _{\ell +1,\infty}\right) \right\} .
\end{eqnarray*}
Next,
\begin{eqnarray*}
	\textrm{Var}\left( A_{j,1}\right)&\!\!\!=\!\!\!&\frac{1}{n^{2}}\sum\limits_{i=1}^{n}E\left[
	\left\{Z_{ij} -h_{j}(\mathbf{X}_{i})\right\}\left(g_{0}-s_{\lambda, g_{0}}\right) \right] ^{2}\\
	&\!\!\!\leq\!\!\!& \frac{\|g_{0}-s_{\lambda, g_{0}}\|_{\infty}^2 }{n}
	\left\Vert z_{j}-h_{j}\right\Vert _{L^{2}}^{2},
\end{eqnarray*}
which together with $E\left( A_{j,1}\right) =0$ implies that
\begin{eqnarray*}
	\left\vert A_{j,1}\right\vert&\!\!\!=\!\!\!&O_{P}\left\{ \frac{F_{2}}{%
		n^{1/2}F_{1}}\left| \triangle \right| ^{\ell +1}\left|g_{0}\right|
	_{\ell+1,\infty}\right.\\
	&&+\left.\frac{\lambda }{n^{3/2}\left| \triangle \right| ^{3}}%
	\left( \left|g_{0}\right| _{2,\infty}+\frac{F_{2}}{F_{1}}\left| \triangle \right| ^{\ell -1}\left|
	g_{0}\right| _{\ell +1,\infty}\right) \right\}.
\end{eqnarray*}
Cauchy-Schwartz inequality, Lemma \ref{LEM:norder} and (\ref{EQ:h_j_tilde}) imply that
\begin{eqnarray*}
	|A_{j,2}| &\leq& \| h_{j}-\widetilde{h}_{j}\|_{n}\left\Vert g_{0} -s_{\lambda,g_{0}}
	\right\Vert _{n}\\
	&=&O_{P}\left(\frac{F_{2}}{%
		F_{1}}\left| \triangle \right| ^{\ell +1}\left|h_{j}\right|
	_{\ell+1,\infty}  \right)\times O_{P}\left\{\frac{F_{2}}{%
		F_{1}}\left| \triangle \right| ^{\ell +1}\left|g_{0}\right|
	_{\ell+1,\infty}\right.\\
	&&+\left.\frac{\lambda }{n\left| \triangle \right| ^{2}}%
	\left( \left|g_{0}\right| _{2,\infty}+\frac{F_{2}}{F_{1}}\left| \triangle \right| ^{\ell -1}\left|g_{0}\right|
	_{\ell +1,\infty}\right) \right\}.
\end{eqnarray*}
Finally, by (\ref{EQ:s_0g_Ev}), we have
\begin{eqnarray*}
	\left\vert A_{j,3}\right\vert &\leq& \frac{\lambda}{n} \|s_{\lambda,g_{0}}\|_{\mathcal{E}_{\upsilon}} \|\widetilde{h}_{j} \|_{\mathcal{E}_{\upsilon}}
	\leq \frac{\lambda}{n} \|s_{0,g_{0}}\|_{\mathcal{E}_{\upsilon}} \|\widetilde{h}_{j} \|_{\mathcal{E}_{\upsilon}}\\
	&\leq& \frac{\lambda}{n} C_{1} \left( \left| g_{0}\right| _{2,\infty}+\frac{F_{2}}{F_{1}}%
	\left| \triangle \right| ^{\ell -1}\left| g_{0}\right| _{\ell
		+1,\infty}\right)\\
	&&\times \left( \left| h_{j}\right| _{2,\infty}+\frac{F_{2}}{F_{1}}%
	\left| \triangle \right| ^{\ell -1}\left| h_{j}\right| _{\ell
		+1,\infty}\right).
\end{eqnarray*}
Combining all the above results yields that
\begin{eqnarray*}
	|A_{j}|&\!\!\!=\!\!\!&O_{P}\left[n^{-1/2}\left\{ \frac{F_{2}}{%
		F_{1}}\left| \triangle \right| ^{\ell +1}\left|g_{0}\right|
	_{\ell+1,\infty}+\frac{\lambda }{n\left| \triangle \right| ^{3}}%
	\left( \left|g_{0}\right| _{2,\infty}\right.\right.\right.\\
	&&+\left.\left.\left.\frac{F_{2}}{F_{1}}\left| \triangle \right| ^{\ell -1}\left|
	g_{0}\right| _{\ell +1,\infty}\right)\right\}
	\right]
\end{eqnarray*}
for $j=1,\ldots,p$. By Assumptions (C3)-(C5), $|A_{j}| =o_{P}(n^{-1/2})$, for $j=1,\ldots,p$.
In addition, we have $nU_{11}=O_P(1)$ according to Lemma %
\ref{LEM:U11}. Therefore,
$\|\widetilde{\bs{\beta}}_{\mu}-\bs{\beta}_{0}\|
=o_{P}\left( n^{-1/2}\right)$ .
\hfill $\Box$

%%%%%%%%%%%%%%%%%%%%%%%%%%%%%%%%%%%%%%%%%%%%%
\begin{lemma}
	\label{LEM:betahateasymp}
	Under Assumptions (A1)-(A3) and (C1)-(C6), as $n\rightarrow\infty $,
	\begin{equation*}
	\left[ \mathrm{Var}\left( \widetilde{\bs{\beta }}_{\epsilon}\left\vert
	\left\{\left( \mathbf{X}_{i},\mathbf{Z}_{i}\right),i=1,\ldots,n\right\}
	\right. \right) \right]^{-1/2}\widetilde{\bs{\beta
	}}_{\epsilon}\longrightarrow N\left( 0,\mathbf{I}_{p\times p}\right),
	\end{equation*}
	where $\widetilde{\bs{\beta }}_{\epsilon}$ is given in (\ref{betatilde_e}).
\end{lemma}

\textit{Proof.}
Note that
\[
\widetilde{\bs{\beta}}_{\epsilon}=\mathbf{U}_{11}
\mathbf{Z}^{\T}\left(\mathbf{I}-\mathbf{B}\mathbf{Q}_2\mathbf{V}_{22}^{-1}
\mathbf{Q}_2^{\T}\mathbf{B}^{\T}\right)\bs{\epsilon}.
\]
For any $\mathbf{b}\in \mathbb{R}^{p}$ with $\left\| \mathbf{b}%
\right\|=1$, we can write $\mathbf{b}^{\T}\widetilde{\bs{\beta }}%
_{\epsilon}=\sum\limits_{i=1}^{n}\alpha_{i}\epsilon _{i}$, \ where
\begin{equation*}
\alpha_{i}^{2}=n^{-2}\mathbf{b}^{\T}(n\mathbf{U}_{11})\left(\mathbf{Z}_{i}^{\T} -\mathbf{V}_{12}\mathbf{V}_{22}^{-1}
\mathbf{Q}_2^{\T}\mathbf{B}_{i}\right)
\left(\mathbf{Z}_{i}-\mathbf{B}_{i}^{\T}\mathbf{Q}_2\mathbf{V}_{22}^{-1}
\mathbf{V}_{21}\right)(n\mathbf{U}_{11})
\mathbf{b},
\end{equation*}%
and conditioning on $\left\{\left( \mathbf{X}_{i},\mathbf{Z}_{i}\right),i=1,\ldots,n\right\} $, $\alpha_{i}\epsilon _{i}$'s
are independent. By Lemma \ref{LEM:U11}, we have
\begin{equation*}
\max_{1\leq i\leq n}\alpha_{i}^{2}\leq Cn^{-2}\max_{1\leq i\leq n}\left\{ \left\|
\mathbf{Z}_{i}\right\| ^{2}+\left\| \mathbf{V}_{12}\mathbf{V}_{22}^{-1}
\mathbf{Q}_2^{\T}\mathbf{B}_{i}\right\|^{2}\right\} ,
\end{equation*}%
where for any $\mathbf{a}\in \mathbb{R}^{p}$,
\[
\mathbf{a}^{\T}\mathbf{V}_{12}\mathbf{V}_{22}^{-1} \mathbf{Q}_{2}^{\T}\mathbf{B}_{i}\mathbf{a}
= n^{-1}\mathbf{a}^{\T}\mathbf{V}_{12} (\mathbf{Q}_{2}^{\T}\bs{\Gamma} _{\lambda }
\mathbf{Q}_{2})^{-1} \mathbf{Q}_{2}^{\T}\mathbf{B}_{i} \mathbf{a}
\leq C n^{-1}|\triangle|^{-2}\mathbf{a}^{\T}\mathbf{Z}^{\T}\mathbf{B}\mathbf{B}_{i}\mathbf{a},
\]
and the $j$-th component of $ n^{-1}\mathbf{Z}^{\T}\mathbf{B}\mathbf{B}_{i}$
is $\frac{1}{n}\sum_{i^{\prime }=1}^{n} Z_{i^{\prime }j} \sum_{\xi\in\mathcal{K}}B_{\xi}\left(
\mathbf{X}_{i^{\prime }}\right) B_{\xi}\left( \mathbf{X}_{i}\right)$. Using Assumptions (A1) and (A2), we have
\[
E\left\{ \frac{1}{n}\sum_{i^{\prime }=1}^{n} Z_{i^{\prime }j}\sum_{\xi\in\mathcal{K}}B_{\xi}\left(
\mathbf{X}_{i^{\prime }}\right) B_{\xi}\left( \mathbf{X}_{i}\right)\right\}
^{2}=O(1),
\]
for large $n$, thus with probability approaching $1$,
\begin{eqnarray*}
	&&\max_{1\leq i\leq n}\left| \frac{1}{n}\sum_{i^{\prime }=1}^{n} \sum_{\xi\in\mathcal{K}}Z_{i^{\prime }j}B_{\xi}\left(
	\mathbf{X}_{i^{\prime }}\right) B_{\xi}\left( \mathbf{X}_{i}\right) \right|
	=O_{P}(1),\\
	&&\max_{1\leq i\leq n} \|\mathbf{V}_{12}%
	\mathbf{V}_{22}^{-1}\mathbf{Q}_{2}^{\T}\mathbf{B}_{i}\|^{2}=O_{P}(|\triangle|^{-2}).
\end{eqnarray*}
Therefore, $\max_{1\leq i\leq n}\alpha_{i}^{2}=O_{P}\left( n^{-2}|\triangle|^{-2}\right) $.
Next, with probability approaching $1$,
\begin{eqnarray}
\sum_{i=1}^{n}\alpha_{i}^{2}&=&\mathrm{Var}%
\left[\mathbf{b}^{\T}\widetilde{\bs{\beta }}_{\epsilon}\left\vert
\left\{\left( \mathbf{X}_{i},\mathbf{Z}_{i}\right),i=1,\ldots,n\right\} \right. \right] \notag\\
&=& \mathbf{b}^{\T}\mathbf{U}_{11}\mathbf{Z}^{\T}
\left(\mathbf{I}-\mathbf{B}\mathbf{Q}_2\mathbf{V}_{22}^{-1} \mathbf{Q}_2^{\T}\mathbf{B}^{\T}\right)
\left(\mathbf{I}-\mathbf{B}\mathbf{Q}_2\mathbf{V}_{22}^{-1} \mathbf{Q}_2^{\T}\mathbf{B}^{\T}\right)
\mathbf{Z}\mathbf{U}_{11}\mathbf{b}\sigma^{2}\notag\\
&=&
n^{-1}\mathbf{b}^{\T}\left( n\mathbf{U}_{11}\right) \left\{n^{-1}\sum_{i=1}^{n}
(\mathbf{Z}_{i}-\widehat{\mathbf{Z}}_{i})(\mathbf{Z}_{i}-\widehat{\mathbf{Z}}_{i})^{\T}\right\} \left(n
\mathbf{U}_{11}\right)\mathbf{b} \sigma^2, \label{EQ:sum_a_2}
\end{eqnarray}
where $\widehat{\mathbf{Z}}_{i}$ is the $i$-th column of
$\mathbf{Z}^{\T}\mathbf{B}\mathbf{Q}_2\mathbf{V}_{22}^{-1} \mathbf{Q}_2^{\T}\mathbf{B}^{\T}$.
Using Lemma \ref{LEM:U11} again, we have $\sum_{i=1}^{n}\alpha_{i}^{2}\geq c n^{-1}$.
So $\max\limits_{1\leq i\leq n}\alpha_{i}^{2}/\sum_{i=1}^{n}\alpha_{i}^{2}=O_{P}\left( n^{-1}|\triangle|^{-2}\right)
=o_{P}\left( 1\right) $ from Assumption (C4). By Linderberg-Feller CLT, we have
\[
\sum_{i=1}^{n}\alpha_{i}\epsilon _{i}/\left( \sum_{i=1}^{n}\alpha_{i}^{2}\right)
^{-1/2}\longrightarrow N\left( 0,1\right).
\]
Then the desired result follows. \hfill $\Box$

For any $j=1,\ldots, p$ and $\lambda>0$, define
\begin{equation}
s_{\lambda,z_{j}}=\mathrm{argmin}_{s\in \mathbb{S}} \sum_{i=1}^{n}\{z_{j}(\mathbf{X}_{i})-s(\mathbf{X}_{i})\}^{2}+\lambda\mathcal{E}_{\upsilon}(s),
\label{DEF:s_lambda_zj}
\end{equation}
where $z_{j}$ is the coordinate mapping that maps $\mathbf{z}$ to its $j$-th component.

%%%%%%%%%%%%%%%%%%%%%%%%%%%%%%%%%%%%%%%%
\begin{lemma}
	Under Assumptions (A1), (A2), (C2), (C3) and (C6), for $s_{\lambda, z_{j}}$ defined in (\ref{DEF:s_lambda_zj}),
	$\left\|s_{0,z_{j}}-s_{\lambda, z_{j}}\right\| _{n}=
	O_{P}(\lambda n^{-1}|\triangle|^{-5})$, $j=1,\ldots,p$.
	\label{LEM:s_0-s_lambda}
\end{lemma}

\textit{Proof.} Note that
\[
n\left\langle z_{j} -s_{\lambda,z_{j}},u\right\rangle_{n}=\lambda
\left\langle s_{\lambda,z_{j}},u\right\rangle _{\mathcal{E}_{\upsilon
}},~\left\langle z_{j} -s_{0,z_{j}},u\right\rangle_{n}=0,\quad \textrm{for all }u\in \mathcal{S},
\]
Inserting $u=s_{0,z_{j}}-s_{\lambda,z_{j}}$ in the above yields that
\[
n\left\| s_{0,z_{j}}-s_{\lambda,z_{j}}\right\| _{n}^{2}=\lambda
\left\langle s_{\lambda,z_{j}},s_{0,z_{j}}-s_{\lambda,z_{j}}\right\rangle _{\mathcal{E}%
	_{\upsilon }} =\lambda (\langle s_{\lambda,z_{j}}, s_{0,z_{j}}\rangle_{\mathcal{E}_{\upsilon }} -\langle
s_{\lambda,z_{j}}, s_{\lambda,z_{j}}\rangle_{\mathcal{E}_{\upsilon }}).
\]
By Cauchy-Schwarz inequality, $\left\| s_{\lambda,z_{j}}\right\|
_{\mathcal{E}_{\upsilon }}^{2}\leq \left\langle s_{\lambda
	,z_{j}},s_{0,z_{j}}\right\rangle _{\mathcal{E}_{\upsilon }}\leq \left\|
s_{\lambda,z_{j}}\right\| _{\mathcal{E}_{\upsilon }}\left\| s_{0,z_{j}}
\right\| _{\mathcal{E}_{\upsilon}}$, which implies
\begin{equation}
\left\| s_{\lambda,z_{j}}\right\|
_{\mathcal{E}_{\upsilon }}\leq \left\| s_{0,z_{j}}\right\|
_{\mathcal{E}_{\upsilon }}.
\label{EQ:s_lambda}
\end{equation}
By (\ref{EQ:ev-vs-n}), we have for large $n$
\[
n\left\| s_{0,z_{j}}-s_{\lambda,z_{j}}\right\| _{n}^{2}
\leq \lambda \left\| s_{\lambda,z_{j}}\right\|
_{\mathcal{E}_{\upsilon }}\left\| s_{0,z_{j}}-s_{\lambda,z_{j}}\right\|_{n}
\times O_{P}(|\triangle|^{-2}),
\]
thus, $\left\| s_{0,z_{j}}-s_{\lambda,z_{j}}\right\| _{n}\leq
\left\| s_{0,z_{j}}\right\|
_{\mathcal{E}_{\upsilon }}\times O_{P}(\lambda n^{-1}|\triangle|^{-2})$.
Markov's inequality implies that
\begin{equation}
\left\| s_{0, z_{j}}\right\|_{\mathcal{E}_{\upsilon}} \le \frac{C_1}{|\triangle|^2} \left\|s_{0, z_{j}}\right\|_{\infty} .
\label{EQ:s_0z_Ev}
\end{equation}
Note that $\left\|s_{0,z_{j}}\right\|_{\infty} \leq  C|\triangle|^{-2} \max_{\xi \in \mathcal{K}}\left| n^{-1}\sum_{i=1}^{n}B_{\xi }\left( \mathbf{X}_{i}\right) Z_{ij}\right|$ with probability approaching one. According to Assumptions (A1) and (A2),
\[
\max_{\xi \in \mathcal{K}}\left| n^{-1}\sum_{i=1}^{n}B_{\xi }\left( \mathbf{X}_{i}\right) Z_{ij}\right|=O_{P}(|\triangle|).
\]
The desired results follows. \hfill $\Box$

%%%%%%%%%%%%%%%%%%%%%%%%%%%%%%%%%%%%%%%%
\begin{lemma}
	\label{LEM:varbetatilde}
	Under Assumptions (A1)-(A3) and (C1)-(C6), for the covariance
	matrix $\bs{\Sigma} $ defined
	in (\ref{DEF:Sigma}), we have $c_{\Sigma}^{\ast}\mathbf{I}_{p}\leq \bs{\Sigma }%
	\leq C_{\Sigma}^{\ast}\mathbf{I}_{p}$, and \par
	$\mathrm{Var}\left( \widetilde{\bs{\beta }}%
	_{\epsilon}\left\vert \left\{\left( \mathbf{X}_{i},\mathbf{Z}_{i}\right),i=1,\ldots,n\right\}\right. \right)
	=n^{-1}\bs{\Sigma }+o_{P}\left(1\right) $.
\end{lemma}
\textit{Proof.} According to (\ref{EQ:sum_a_2}),
\begin{eqnarray*}
	\mathrm{Var}\left( \widetilde{\bs{\beta }}_{\epsilon}\left\vert
	\left\{\left( \mathbf{X}_{i},\mathbf{Z}_{i}\right)\right\}\right. \right)\!=\!
	n^{-1} (n\mathbf{U}_{11})\!
	\left\{ n^{-1}\sum_{i=1}^{n}(\mathbf{Z}_{i}-\widehat{\mathbf{Z}}_{i})(\mathbf{Z}_{i}-\widehat{\mathbf{Z}}_{i})^{\T}\right\}
	\!(n\mathbf{U}_{11}) \sigma^2.
\end{eqnarray*}
By the definition of $\mathbf{U}_{11}^{-1}$ in (\ref{DEF:U11}), we have
\[
(n\mathbf{U}_{11})^{-1}=\frac{1}{n}\sum_{i=1}^{n}\mathbf{Z}_{i} (\mathbf{Z}_{i}-\widehat{\mathbf{Z}}_{i})^{\T}
=\left(\langle z_{j},z_{j^{\prime}}-s_{\lambda,z_{j^{\prime}}}\rangle_{n}\right)_{1\leq j,j^{\prime}\leq p}.
\]
As in the proof of Lemma \ref{LEM:betamu-beta}, let $\widetilde{h}_{j}\in \mathbb{S}$ and $h_{j}$ satisfy (\ref{EQ:h_j_tilde}). Then,
\begin{equation}
\langle z_{j},z_{j^{\prime}}-s_{\lambda,z_{j^{\prime}}}\rangle_{n}
=\langle z_{j}-\widetilde{h} _{j}, z_{j^{\prime}}-s_{\lambda,z_{j^{\prime}}}\rangle_{n}
+\frac{\lambda}{n}\langle s_{\lambda,z_{j^{\prime}}}, \widetilde{h}_{j} \rangle_{\mathcal{E}_{\upsilon}}.
\label{EQ:zj}
\end{equation}
According to (\ref{EQ:s_0g_Ev}), (\ref{EQ:s_lambda}) and (\ref{EQ:s_0z_Ev}), we have
\begin{eqnarray*}
	\left\vert \langle s_{\lambda,z_{j^{\prime}}}, \widetilde{h}_{j^{\prime}}\rangle_{\mathcal{E}_{\upsilon}}\right\vert
	&\!\!\!\leq \!\!\!& \|s_{\lambda,z_{j^{\prime}}}\|_{\mathcal{E}_{\upsilon}} \|\widetilde{h}_{j^{\prime}} \|_{\mathcal{E}_{\upsilon}}
	\leq  \|s_{0,z_{j^{\prime}}}\|_{\mathcal{E}_{\upsilon}} \|\widetilde{h}_{j^{\prime}} \|_{\mathcal{E}_{\upsilon}}\notag \\
	&\!\!\!\leq \!\!\!&\frac{C}{|\triangle|^{2}} \|s_{0,z_{j^{\prime}}}\|_{\infty} \|\widetilde{h}_{j^{\prime}} \|_{\mathcal{E}_{\upsilon}}\notag \\
	&\!\!\!\leq\!\!\!&\frac{CC^{*}}{|\triangle|^{3}} \left( \left| h_{j}^{\prime}\right| _{2,\infty}+\frac{F_{2}}{F_{1}}%
	\left| \triangle \right| ^{\ell +1-\upsilon}\left| h_{j}^{\prime}\right| _{\ell
		+1,\infty}\right). \label{EQ:zj2}
\end{eqnarray*}
Note that
\[
\langle z_{j}-\widetilde{h} _{j}, z_{j^{\prime}}-s_{\lambda,z_{j^{\prime}}}\rangle_{n}
=\langle z_{j}-h_{j}, z_{j^{\prime}}-h_{j^{\prime}}\rangle_{n}
+ \langle h_{j}-\widetilde{h}_{j}, h_{j^{\prime}}-\widetilde{h}_{j^{\prime}}\rangle_{n}
+\langle z_{j}-h_{j}, h_{j^{\prime}}-\widetilde{h}_{j^{\prime}}\rangle_{n}
\]
\begin{equation}
+\langle h_{j}-\widetilde{h}_{j}, z_{j^{\prime}}-h_{j^{\prime}}\rangle_{n}+\langle z_{j}-h_{j}, \widetilde{h}_{j^{\prime}} -s_{\lambda,z_{j^{\prime}}}\rangle_{n}
+\langle h_{j}-\widetilde{h}_{j}, \widetilde{h}_{j^{\prime}}-s_{\lambda,z_{j^{\prime}}} \rangle_{n}.
\label{EQ:zj1}
\end{equation}
By (\ref{EQ:h_j_tilde}), the second term on the right side of (\ref{EQ:zj1}) satisfies that
\begin{equation*}
\left|\langle h_{j}-\widetilde{h}_{j}, h_{j^{\prime}}-\widetilde{h}_{j^{\prime}}\rangle _{\infty}\right|
\leq \|h_{j}-\widetilde{h}_{j}\|_{\infty}
\|h_{j^{\prime}}-\widetilde{h}_{j^{\prime}}\|_{\infty}
=o_{P}(1).
\label{EQ:zj12}
\end{equation*}
By Lemma \ref{LEM:Rnorder} and (\ref{EQ:h_j_tilde}), the third term on the right side of (\ref{EQ:zj1}) satisfies that
\begin{equation*}
\left\vert \langle z_{j}-h_{j}, h_{j^{\prime}}-\widetilde{h}_{j^{\prime}}\rangle  _{n}\right\vert
\leq \left\{\| z_{j}-h_{j}\|_{L^{2}} (1+o_{P}(1))\right\}\|h_{j^{\prime}}-\widetilde{h}_{j^{\prime}}\|_{\infty}
=o_{P}(1).
\label{EQ:zj13}
\end{equation*}
Similarly, $\left|\langle h_{j}-\widetilde{h}_{j}, z_{j^{\prime}}-h_{j^{\prime}}\rangle_{n}\right|=o_{P}(1)$.
From the triangle inequality, we have
\[
\|\widetilde{h}_{j} -s_{\lambda,z_{j}} \| _{n} \leq
\|\widetilde{h}_{j} -h_{j}\|_{n} + \|h_{j}-s_{0,z_{j}}\|_{n} + \|s_{0,z_{j}}-s_{\lambda,z_{j}}\|_{n}.
\]
According to (\ref{EQ:h_j_tilde}) and Lemma \ref{LEM:s_0-s_lambda}, we have
\[
\| \widetilde{h}_{j} -s_{\lambda,z_{j}} \| _{n} \leq
\|h_{j}-s_{0,z_{j}}\|_{n} + o_{P}(1).
\]
Define $h_{j,n}^{\ast}=\mathrm{argmin}_{h\in \mathbb{S}}\|z_j-h\|_{L^{2}}$, then, from the triangle inequality,
\[
\|h_{j}-s_{0,z_{j}}\|_{n}\leq\|h_{j}-h_{j,n}^{\ast }\|_{n}
+\|h_{j,n}^{\ast}- s_{0,z_{j}}\|_{n}
\]
Note that $\|h_{j}-h_{j,n}^{\ast}\|_{L^{2}}=o_{P}(1)$. Lemma \ref{LEM:Rnorder} implies that $\|h_{j}-h_{j,n}^{\ast }\|_{n}=o_{P}(1)$.
Next note that
$\|s_{0,z_{j}}-h_{j,n}^{\ast}\|_{L^{2}}^{2}=\|z_{j}-s_{0,z_{j}}\|_{L^{2}}^{2}-\|z_{j}- h_{j,n}^{\ast}\|_{L^{2}}^{2}$
and
$\|z_{j}-s_{0,z_{j}}\|_{n}\leq \|z_{j}- h_{j,n}^{\ast}\|_{n}$.
Using Lemma \ref{LEM:Rnorder} again, we have
\[
\|s_{0,z_{j}}-h_{j,n}^{\ast}\|_{L^{2}}^{2}=o_{P}(\|z_{j}-h_{j,n}^{\ast}\|_{L^{2}}^{2})+o_{P}(\|z_{j}- s_{0,z_{j}}\|_{L^{2}}^{2}).
\]
Since there exists a constant $C$ such that $\|z_{j}- h_{j,n}^{\ast}\|_{L^{2}} \leq C$, so we have
\[
\|z_{j}-s_{0,z_{j}}\|_{L^{2}}\leq \|z_{j}- h_{j,n}^{\ast}\|_{L^{2}}+ \|h_{j,n}^{\ast}-s_{0,z_{j}}\|_{L^{2}}
\leq C+ \|h_{j,n}^{\ast}-s_{0,z_{j}}\|_{L^{2}}.
\]
Therefore, $\|h_{j,n}^{\ast}-s_{0,z_{j}}\|_{L^{2}}=o_{P}(1)$. Lemma \ref{LEM:Rnorder} implies that
$\|h_{j,n}^{\ast }-s_{0,z_{j}}\|_{n}=o_{P}(1)$. As a consequence,
\begin{equation}
\|s_{0,z_{j}} -h_{j}\|_{n}=o_{P}(1).
\label{EQ:zj-psijstar}
\end{equation}
For the fifth item, by Lemma \ref{LEM:Rnorder} and (\ref{EQ:zj-psijstar}), we have
\begin{eqnarray*}
	\left\vert \langle z_{j}-h_{j}, \widetilde{h}_{j^{\prime}} -s_{\lambda,z_{j^{\prime}}} \rangle_{n} \right\vert
	&\!\!\!\leq\!\!\!& \left\{\| z_{j}-h_{j}\|_{L^{2}} (1+o_{P}(1))\right\}
	\left\{ \|h_{j}-s_{0,z_{j}}\|_{n} + o_{P}(1)\right\}\\
	&\!\!\!=\!\!\!&o_{P}(1).
	\label{EQ:zj5}
\end{eqnarray*}
Similarly, for the sixth item, we have
\begin{equation}
\left\vert \langle h_{j}-\widetilde{h}_{j}, \widetilde{h}_{j^{\prime}} -s_{\lambda,z_{j^{\prime}}} \rangle_{n}  \right\vert
\leq  \|h_{j}-\widetilde{h}_{j}\|_{n}
\left\{ \|h_{j}-s_{0,z_{j}}\|_{n} + o_{P}(1)\right\}=o_{P}(1).
\label{EQ:zj6}
\end{equation}
Combining the above results from (\ref{EQ:zj}) to (\ref{EQ:zj6}) gives that
\begin{equation*}
\langle z_{j},z_{j^{\prime}}-s_{\lambda,z_{j^{\prime}}}\rangle_{n}
= \langle z_{j}-h_{j},z_{j^{\prime}}-h_{j^{\prime}}^{\ast}\rangle_{n}  +o_{P}(1).
\end{equation*}%
Therefore,
\[
(n\mathbf{U}_{11})^{-1}=\frac{1}{n}\sum_{i=1}^{n}(\mathbf{Z}_{i}-\widetilde{\mathbf{Z}}_{i}) (\mathbf{Z}_{i}-\widetilde{\mathbf{Z}}_{i})^{\T}
+o_{P}(1)=E[(\mathbf{Z}_{i}-\widetilde{\mathbf{Z}}_{i}) (\mathbf{Z}_{i}-\widetilde{\mathbf{Z}}_{i})^{\T}]+o_{P}(1).
\]
Hence,
\[\hspace{1cm}
\mathrm{Var}\left( \widetilde{\bs{\beta }}%
_{\epsilon}\left\vert \left\{\left( \mathbf{X}_{i},\mathbf{Z}_{i}\right),i=1,\ldots,n\right\}\right. \right)
=n^{-1}\bs{\Sigma }^{-1}+o_{P}\left(1\right) . \hspace{1cm}\square
\]

%%%%%%%%%%%%%%%%%%%%%%%%%%%%%%%%%%%%%%%%%%%%%%%%%%%%%%%%%%%%%%%%%%%%%%
\vskip .05in \noindent \textbf{D.2. Proof of Theorem \ref{THM:g-convergence}}

Let $\mathbf{H}_{\mathbf{Z}}=\mathbf{I}-\mathbf{Z} (\mathbf{Z}^{\T}\mathbf{Z})^{-1}\mathbf{Z}^{\T}$, then
\[
\widehat{\bs{\theta}}=\mathbf{U}_{22}\mathbf{Q}_{2}^{\T}\mathbf{B}^{\T} \mathbf{H}_{\mathbf{Z}}\mathbf{Y}
=\mathbf{U}_{22}\mathbf{Q}_{2}^{\T}\mathbf{B}^{\T} \mathbf{H}_{\mathbf{Z}}\mathbf{g}_{0}+\mathbf{U}_{22}\mathbf{Q}_{2}^{\T}\mathbf{B}^{\T} \mathbf{H}_{\mathbf{Z}}\bs{\epsilon}
=\widetilde{\bs{\theta}}_{\mu}+\widetilde{\bs{\theta}}_{\epsilon}.
\]
According to Lemma \ref{LEM:bias},
$\|s_{0,g_{0}}-g_{0}\|_{\infty}\leq C\frac{F_{2}}{F_{1}}|\triangle|^{\ell+1}|g_{0}|_{\ell+1,\infty}$.
Denote by $\bs{\gamma}_0=\mathbf{Q}_2\bs{\theta}_{0}$ the spline coefficients of $s_{0,g_{0}}$.
Then we have the following decomposition: $\widehat{\bs{\theta}}-\bs{\theta}_{0}
=\widetilde{\bs{\theta}}_{\mu}-\bs{\theta}_{0}
+\widetilde{\bs{\theta}}_{\epsilon}$. Note that
\begin{eqnarray*}
	\widetilde{\bs{\theta}}_{\mu}-\bs{\theta}_{0}&\!\!\!=\!\!\!&
	\mathbf{U}_{22}\mathbf{Q}_{2}^{\T}\mathbf{B}^{\T} \mathbf{H}_{\mathbf{Z}}\mathbf{g}_{0}-\bs{\theta}_{0}\\
	&\!\!\!=\!\!\!&\mathbf{U}_{22}\mathbf{Q}_{2}^{\T}\mathbf{B}^{\T} \mathbf{H}_{\mathbf{Z}} (\mathbf{g}_{0}-\mathbf{B}\mathbf{Q}_{2}\bs{\theta}_{0})
	-\lambda\mathbf{U}_{22}\mathbf{Q}_{2}^{\T}\mathbf{P}\mathbf{Q}_{2}\bs{\theta}_{0}.
\end{eqnarray*}
According to (\ref{DEF:U22}), for any $\mathbf{a}$
\[
\mathbf{a}^{\T}\mathbf{U}_{22}^{-1}\mathbf{a}
=\mathbf{a}^{\T}\mathbf{Q}_{2}^{\T}\left(\mathbf{B}^{\T}  \mathbf{H}_{\mathbf{Z}}\mathbf{B}+\lambda\mathbf{P} \right)\mathbf{Q}_{2}\mathbf{a}.
\]
Since $\mathbf{H}_{\mathbf{Z}}$ is idempotent, so its eigenvalues $\pi_{j}$ is either 0 or 1. Without loss of generality we can arrange the eigenvalues in decreasing order
so that $\pi_{j}=1$, $j=1,\ldots, m$ and $\pi_{j}=1$, $j=m + 1,\ldots, n$. Therefore, we have
\[
\mathbf{a}^{\T}(n\mathbf{U}_{22})^{-1}\mathbf{a}=\frac{1}{n}\sum_{j=1}^{m} \pi_{j}\mathbf{a}^{\T}\mathbf{Q}_{2}^{\T} \mathbf{B}^{\T}\mathbf{e}_{j}\mathbf{e}_{j}^{\T}\mathbf{B}\mathbf{Q}_{2}\mathbf{a}
+\frac{\lambda}{n}
\mathbf{a}^{\T}\mathbf{Q}_{2}^{\T}\mathbf{P}\mathbf{Q}_{2}\mathbf{a},
\]
where $\mathbf{e}_{j}$ be the indicator vector which is a zero
vector except for an entry of one at position $j$. Using Markov's inequality, we have
\[
\frac{\lambda }{n}\mathcal{E}_{\upsilon }\left(\sum_{\xi \in
	\mathcal{K}}a_{\xi }B_{\xi }\right)\leq
\frac{\lambda }{n}\frac{C_{1}}{|\triangle |^{2}}C_{2}\Vert
\mathbf{a}\Vert ^{2}.
\]
Thus, by Conditions (C4) and (C5), $n\mathbf{a}^{\T}\mathbf{U}_{22}\mathbf{a}\leq C|\triangle |^{-2}$.
Next
\begin{eqnarray*}
	&&\|\mathbf{U}_{22}\mathbf{Q}_{2}^{\T}\mathbf{B}^{\T} \mathbf{H}_{\mathbf{Z}} (\mathbf{g}_{0}-\mathbf{B}\mathbf{Q}_{2}\bs{\theta}_{0})\|
	\leq C^{1/2}|\triangle|^{-1} n^{-1}\|\mathbf{B}^{\T} \mathbf{H}_{\mathbf{Z}} (\mathbf{g}_{0}-\mathbf{B}\mathbf{Q}_{2}\bs{\theta}_{0})\|\\
	&&\hspace{2cm}\leq C^{1/2}|\triangle|^{-1} n^{-1}\left[\sum_{\xi\in\mathcal{K}}\{\mathbf{B}_{\xi}^{\T} \mathbf{H}_{\mathbf{Z}} (\mathbf{g}_{0}-\mathbf{B}\mathbf{Q}_{2}\bs{\theta}_{0})\}^2\right]^{1/2}\\
	&&\hspace{2cm}=O_{P}\left(\frac{F_{2}}{F_{1}}|\triangle|^{\ell}|g_{0}|_{\ell+1,\infty}\right),
\end{eqnarray*}
and
\[
\lambda\|\mathbf{U}_{22}\mathbf{Q}_{2}^{\T}\mathbf{P}\mathbf{Q}_{2}\bs{\theta}_{0}\|
\leq \frac{C\lambda}{n|\triangle|^{4}}\|s_{0,g_{0}}\|_{\mathcal{E}_{\upsilon}}
\leq \frac{C\lambda}{n|\triangle|^{4}}\left(|g_{0}|_{2,\infty}+\frac{F_2}{F_1}|\triangle|^{\ell-1}|g_{0}|_{\ell+1,\infty}\right).
\]
Thus,
\[
\|\widetilde{\bs{\theta}}-\bs{\theta}_{0}\|
=O_{P}\left\{ \frac{%
	\lambda }{n\left| \triangle \right| ^{4}}|g_{0}| _{2,\infty}+\left( 1+\frac{\lambda }{n\left|\triangle \right| ^{5}}\right)
\frac{F_{2}}{F_{1}}\left| \triangle \right| ^{\ell}\left|
g_{0}\right| _{\ell +1,\infty}\right\}.
\]
For any $\bs{\alpha}$ with $\left\Vert \bs{\alpha} \right\Vert =1$, we
write $\bs{\alpha}^{\T}\widetilde{\bs{\theta}}_{\epsilon}=\sum_{i=1}^{n}\alpha_{i}\epsilon _{i}$ and
\[
\alpha_{i}^{2}=\bs{\alpha}^{\T}\mathbf{U}_{22}\mathbf{Q}_{2}^{\T}\mathbf{B}^{\T} \mathbf{H}_{\mathbf{Z}}\mathbf{B}\mathbf{Q}_{2}\mathbf{U}_{22}\bs{\alpha}.
\]
Following the same arguments as those in Lemma \ref{LEM:betahateasymp}, we have $\max_{1\leq i\leq n} \alpha_{i}^{2} =O_{P}(n^{-2}|\triangle|^{-4})$.
Thus,
\[
\|\widetilde{\bs{\theta}}_{\epsilon}\| \leq  |\triangle|^{-1}
|\bs{\alpha}^{\T}\widetilde{\bs{\theta}}
_{\epsilon}| = |\triangle|^{-1}|\sum_{i=1}^{n}\alpha_{i}\epsilon _{i}|=O_{P}(|\triangle|^{-2}n^{-1/2}).
\]
Therefore,
\[
\|\widehat{\bs{\theta}}
-\bs{\theta}_{0}\|=O_{P}\left\{ \frac{\lambda }{n\left| \triangle \right| ^{4}}|g_{0}| _{2,\infty}+\left( 1+\frac{\lambda }{n\left|\triangle \right| ^{5}}\right)
\frac{F_{2}}{F_{1}}\left| \triangle \right| ^{\ell}\left|
g_{0}\right| _{\ell +1,\infty}+ \frac{1}{\sqrt{n}|\triangle|^{2}}\right\}.
\]
Observing that $\widehat{g}(\mathbf{x})=\mathbf{B}(\mathbf{x})\widehat{\bs{\gamma}}
=\mathbf{B}(\mathbf{x})\mathbf{Q}_{2}\widehat{\bs{\theta}}$, we have
\[
\| \widehat{g}-g_{0}\|_{L^{2}} \leq \|\widehat{g}-s_{0,g_{0}}\|_{L^{2}}+C\frac{F_{2}}{F_{1}}|\triangle|^{\ell+1}|g_{0}|_{\ell+1,\infty}.
\]
According to Lemma \ref{LEM:normequity}, we have.
\begin{eqnarray*}
	\| \widehat{g}-g_{0}\|_{L^{2}} &\!\!\!\leq\!\!\!& C\left(|\triangle|\|\widehat{\bs{\gamma}}
	-\bs{\gamma}_{0}\|+\frac{F_{2}}{F_{1}}|\triangle|^{\ell+1}|g_{0}|_{\ell+1,\infty}\right)=O_{P}\left\{ \frac{\lambda }{n\left| \triangle \right| ^{3}}|g_{0}| _{2,\infty}\right.\\
	&&+\left.\left(1+\frac{\lambda }{n\left|\triangle \right| ^{5}}\right)
	\frac{F_{2}}{F_{1}}\left| \triangle \right| ^{\ell+1}\left|
	g_{0}\right| _{\ell +1,\infty}+ \frac{1}{\sqrt{n}|\triangle|}\right\}.
\end{eqnarray*}
The proof is completed.

%%%%%%%%%%%%%%%%%%%%%%%%%%%%%%%%%%%%%%%%%%%%%%%%%%%%%%%
%%%%%%%%%%%%%%%%%%%%%%%%%%%%%%%%%%%%%%%%%%%%%%%%%%%%%%%
%%%%%%%%%%%%%%%%%%%%%%%%%%%%%%%%%%%%%%%%%%%%%%%%%%%%%%%
\bibliographystyle{asa}
\bibliography{reference}

\begin{thebibliography}{43}
\newcommand{\enquote}[1]{``#1''}
\expandafter\ifx\csname natexlab\endcsname\relax\def\natexlab#1{#1}\fi

\bibitem[{Abbott et~al.(2008)Abbott, Lin, Martian, and
  Einerson}]{abbott2008atmospheric}
Abbott, M.~L., Lin, C.-J., Martian, P., and Einerson, J.~J. (2008),
  \enquote{Atmospheric mercury near Salmon Falls Creek Reservoir in southern
  Idaho,} \textit{Applied Geochemistry}, 23, 438--453.

\bibitem[{Awanou et~al.(2005)Awanou, Lai, and Wenston}]{Awanou:Lai:Wenston:05}
Awanou, G., Lai, M.~J., and Wenston, P. (2005), \enquote{The multivariate
  spline method for scattered data fitting and numerical solutions of partial
  differential equations,} \textit{Wavelets and splines: Athens 2005}, 24--74.

\bibitem[{Brown et~al.(2015)Brown, Chen, Voytek, and
  Amirbahman}]{brown2015effect}
Brown, L.~E., Chen, C.~Y., Voytek, M.~A., and Amirbahman, A. (2015),
  \enquote{The effect of sediment mixing on mercury dynamics in two intertidal
  mudflats at Great Bay Estuary, New Hampshire, USA,} \textit{Marine
  Chemistry}, 177, 731--741.

\bibitem[{Chen et~al.(2011)Chen, Liang, and Wang}]{Chen:Liang:Wang:11}
Chen, R., Liang, H., and Wang, J. (2011), \enquote{On determination of linear
  components in additive models,} \textit{Journal of Nonparametric Statistics},
  23, 367--383.

\bibitem[{Craven and Wahba(1979)}]{Craven:Wahba:79}
Craven, P. and Wahba, G. (1979), \enquote{Smoothing noisy data with spline
  functions,} \textit{Numerische Mathematik}, 31, 377--403.

\bibitem[{Eilers(2006)}]{Eilers:06}
Eilers, P. (2006), \enquote{P-spline smoothing on difficult domains,} .

\bibitem[{Furrer et~al.(2011)Furrer, Nychka, and
  Sainand}]{Furrer:Nychka:Sainand:11}
Furrer, R., Nychka, D., and Sainand, S. (2011), \textit{Package `fields'. R
  package version 6.6.1.}, [online] Available at
  \emph{http://cran.r-project.org/web/packages/fields/fields.pdf}.

\bibitem[{Green and Silverman(1993)}]{green1993nonparametric}
Green, P.~J. and Silverman, B.~W. (1993), \textit{Nonparametric regression and
  generalized linear models: a roughness penalty approach}, CRC Press.

\bibitem[{Green and Silverman(1994)}]{Green:Silverman:94}
--- (1994), \textit{Nonparametric regression and generalized linear models},
  Chapman and Hall, London.

\bibitem[{H\"{a}rdle et~al.(2000)H\"{a}rdle, W., and Gao}]{Hardle:Liang:Gao:00}
H\"{a}rdle, W., Liang, H., and Gao, J.~T. (2000), \textit{Partially linear
  models}, Heidelberg: Springer Physica-Verlag.

\bibitem[{He and Shi(1996)}]{He:Shi:96}
He, X. and Shi, P. (1996), \enquote{Bivariate tensor-troduct B-splines in a
  partly linear model,} \textit{Journal of Multivariate Analysis}, 58,
  162--181.

\bibitem[{Huang(2003)}]{Huang:03}
Huang, J. (2003), \enquote{Asymptotics for polynomial spline regression under
  weak conditions,} \textit{Statistics \& Probability Letters}, 65, 207--216.

\bibitem[{Huang et~al.(2007)Huang, Zhang, and Zhou}]{Huang:Zhang:Zhou:07}
Huang, J.~Z., Zhang, L., and Zhou, L. (2007), \enquote{Efficient estimation in
  marginal partially linear models for longitudinal/clustered data using
  splines,} \textit{Scandinavian Journal of Statistics}, 34, 451--477.

\bibitem[{Lai(2008)}]{Lai:08}
Lai, M.~J. (2008), \enquote{Multivariate splines for data fitting and
  approximation,} \textit{Conference Proceedings of the 12th Approximation
  Theory}, 210--228.

\bibitem[{Lai and Schumaker(1998)}]{Lai:Schumaker:98}
Lai, M.~J. and Schumaker, L.~L. (1998), \enquote{Approximation power of
  bivariate splines,} \textit{Advances in Computational Mathematics}, 9,
  251--279.

\bibitem[{Lai and Schumaker(2007)}]{Lai:Schumaker:07}
--- (2007), \textit{Spline functions on triangulations}, Cambridge University
  Press.

\bibitem[{Lai and Wang(2013)}]{Lai:Wang:13}
Lai, M.~J. and Wang, L. (2013), \enquote{Bivariate penalized splines for
  regression,} \textit{Statistica Sinica}, 23, 1399--1417.

\bibitem[{Li and Ruppert(2008)}]{Li:Ruppert:08}
Li, Y. and Ruppert, D. (2008), \enquote{On the asymptotics of penalized
  splines,} \textit{Biometrika}, 95, 415--436.

\bibitem[{Liang et~al.(1999)Liang, H\"{a}rdle, and
  Carroll}]{Liang:Hardle:Carroll:99}
Liang, H., H\"{a}rdle, W., and Carroll, R.~J. (1999), \enquote{Estimation in a
  semiparametric partially linear errors-in-variables model,} \textit{The
  Annals of Statistics}, 27, 1519--1535.

\bibitem[{Liang and Li(2009)}]{Liang:Li:09}
Liang, H. and Li, R. (2009), \enquote{Variable selection for partially linear
  models with measurement errors,} \textit{Journal of the American Statistical
  Association}, 104, 234--248.

\bibitem[{Liu et~al.(2011)Liu, Wang, and Liang}]{Liu:Wang:Liang:11}
Liu, X., Wang, L., and Liang, H. (2011), \enquote{Estimation and variable
  selection for semiparametric additive partial linear models,}
  \textit{Statistica Sinica}, 21, 1225--1248.

\bibitem[{Ma et~al.(2013)Ma, Song, and Wang}]{Ma:Song:Wang:13}
Ma, S., Song, Q., and Wang, L. (2013), \enquote{Simultaneous variable selection
  and estimation in semiparametric modeling of longitudinal/clustered data,}
  \textit{Bernoulli}, 19, 252--274.

\bibitem[{Ma et~al.(2006)Ma, Chiou, and Wang}]{Ma:Chiou:Wang:06}
Ma, Y., Chiou, J.-M., and Wang, N. (2006), \enquote{Efficient semiparametric
  estimator for heteroscedastic partially linear models,} \textit{Biometrika},
  93, 75--84.

\bibitem[{Mammen and van~de Geer(1997)}]{Mammen:DeGeer:97}
Mammen, E. and van~de Geer, S. (1997), \enquote{Penalized quasi-likelihood
  estimation in partial linear models,} \textit{The Annals of Statistics}, 25,
  1014--1035.

\bibitem[{Marx and Eilers(2005)}]{Marx:Eilers:05}
Marx, B. and Eilers, P. (2005), \enquote{Multidimensional penalized signal
  regression,} \textit{Technometrics}, 47, 13--22.

\bibitem[{Miller and Wood(2014)}]{miller2014finite}
Miller, D.~L. and Wood, S.~N. (2014), \enquote{Finite area smoothing with
  generalized distance splines,} \textit{Environmental and ecological
  statistics}, 21, 715--731.

\bibitem[{Persson and Strang(2004)}]{Persson:Strang:04}
Persson, P.~O. and Strang, G. (2004), \enquote{A simple mesh generator in
  MATLAB,} \textit{SIAM Review}, 46, 329--345.

\bibitem[{Ramsay(2002)}]{Ramsay:02}
Ramsay, T. (2002), \enquote{Spline smoothing over difficult regions,}
  \textit{{Journal of the Royal Statistical Society, Series~B}}, 64, 307--319.

\bibitem[{Sangalli et~al.(2013)Sangalli, Ramsay, and
  Ramsay}]{Sangalli:Ramsay:Ramsay:13}
Sangalli, L., Ramsay, J., and Ramsay, T. (2013), \enquote{Spatial spline
  regression models,} \textit{Journal of the Royal Statistical Society, Series
  B}, 75, 681--703.

\bibitem[{Shewchuk(1996)}]{Shewchuk:96}
Shewchuk, J.~R. (1996), \enquote{Triangle: engineering a 2D quality mesh
  generator and Delaunay triangulator,} in \textit{Applied Computational
  Geometry Towards Geometric Engineering}, eds. Lin, M.~C. and Manocha, D.,
  Berlin, Heidelberg: Springer Berlin Heidelberg, pp. 203--222.

\bibitem[{Speckman(1988)}]{Speckman:88}
Speckman, P. (1988), \enquote{Kernel smoothing in partial linear models,}
  \textit{Journal of the Royal Statistical Society. Series B (Methodological)},
  50, 413--436.

\bibitem[{Steve et~al.(2010)Steve, Christian, and
  Gareth}]{steve2009distribution}
Steve, J., Christian, K., and Gareth, H. (2010), \enquote{Distribution of
  mercury and trace metals in shellfish and sediments in the Gulf of Maine,} p.
  online available at
  \emph{https://pdfs.semanticscholar.org/b93b/256048b5f55553573f022b07888295265689.pdf}.

\bibitem[{von Golitschek and Schumaker(2002)}]{vonGolitschek:Schumaker:02}
von Golitschek, M. and Schumaker, L.~L. (2002), \enquote{Bounds on projections
  onto bivariate polynomial spline spaces with stable local bases,}
  \textit{Constructive approximation}, 18, 241--254.

\bibitem[{Wahba(1990)}]{Wahba:90}
Wahba, G. (1990), \textit{Spline models for observational data}, SIAM
  Publications, Philadelphia.

\bibitem[{Wang et~al.(2017)Wang, Mu, and Wang}]{Wang:Mu:Wang:17}
Wang, G., Mu, J., and Wang, L. (2017), \enquote{Efficient smoothing parameter
  and triangulation selection for bivariate penalized spline regression,}
  \textit{Manuscript}, online available at
  \emph{http://people.wm.edu/~gwang01/triangulation.pdf}.

\bibitem[{Wang and Ranalli(2007)}]{Wang:Ranalli:07}
Wang, H. and Ranalli, M.~G. (2007), \enquote{Low-rank smoothing splines on
  complicated domains,} \textit{Biometrics}, 63, 209--217.

\bibitem[{Wang et~al.(2011)Wang, Liu, Liang, and
  Carroll}]{Wang:Liu:Liang:Carroll:11}
Wang, L., Liu, X., Liang, H., and Carroll, R. (2011), \enquote{Estimation and
  variable selection for generalized additive partial linear models,}
  \textit{Annals of Statistics}, 39, 1827--1851.

\bibitem[{Wang et~al.(2014)Wang, Xue, Qu, and Liang}]{Wang:Xue:Qu:Liang:14}
Wang, L., Xue, L., Qu, A., and Liang, H. (2014), \enquote{Estimation and model
  selection in generalized additive partial linear models for correlated data
  with diverging number of covariates,} \textit{Annals of Statistics}, 42,
  592--624.

\bibitem[{Wood(2003)}]{Wood:03}
Wood, S.~N. (2003), \enquote{Thin plate regression splines,} \textit{Journal of
  the Royal Statistical Society, Series B}, 65, 95--114.

\bibitem[{Wood et~al.(2008)Wood, Bravington, and
  Hedley}]{Wood:Bravington:Hedley:08}
Wood, S.~N., Bravington, M.~V., and Hedley, S.~L. (2008), \enquote{Soap film
  smoothing,} \textit{{Journal of the Royal Statistical Society, Series~B}},
  70, 931--955.

\bibitem[{Xiao et~al.(2013)Xiao, Li, and Ruppert}]{Xiao:Li:Ruppert:13}
Xiao, L., Li, Y., and Ruppert, D. (2013), \enquote{Fast bivariate P-splines:
  the sandwich smoother,} \textit{Journal of the Royal Statistical Society,
  Series B}, 75, 577--599.

\bibitem[{Zhang et~al.(2011)Zhang, Cheng, and Liu}]{Zhang:Cheng:Liu:11}
Zhang, H., Cheng, G., and Liu, Y. (2011), \enquote{Linear or nonlinear?
  Automatic structure discovery for partially linear models,} \textit{Journal
  of American Statistical Association}, 106, 1099--1112.

\bibitem[{Zhou and Pan(2014)}]{Zhou:Pan:14}
Zhou, L. and Pan, H. (2014), \enquote{Smoothing noisy data for irregular
  regions using penalized bivariate splines on triangulations,}
  \textit{Computational Statistics}, 29, 263--281.

\end{thebibliography}

\end{document}